\documentclass[11pt]{article}
\usepackage[top=0.7in,bottom=0.9in,left=1.0in,right=1.0in]{geometry}
\usepackage{amsmath,amssymb,amsthm}
\allowdisplaybreaks
\usepackage{mathtools}
\usepackage{caption}
\usepackage{subcaption}
\usepackage{graphicx}
\usepackage[hidelinks]{hyperref}
\usepackage{color}
\usepackage{bbold}
\usepackage{enumitem}
\usepackage{booktabs}
\usepackage{multirow}
\usepackage{setspace}
\usepackage{natbib}
\bibpunct[, ]{(}{)}{,}{a}{}{,}%
\def\BIBand{and}%

\newcommand\aseq{\overset{\mathrm{a.s.}}{=}}
\newcommand\asl{\overset{\mathrm{a.s.}}{<}}
\newcommand\asleq{\overset{\mathrm{a.s.}}{\leq}} 

\newtheorem{lemma}{Lemma}
\newtheorem{theorem}{Theorem}
\newtheorem{proposition}{Proposition}
\newtheorem{corollary}{Corollary}

\theoremstyle{remark}
\newtheorem{example}{Example}
\newtheorem{remark}{Remark}

\begin{document}

\title{\textbf{Randomized Routing to Remote Queues}}

\author{Shuangchi He\thanks{National University of Singapore, heshuangchi@nus.edu.sg}\and
Yunfang Yang\thanks{National University of Singapore, yang.yf@nus.edu.sg}\and
Yao Yu\thanks{Amazon, yyu15@ncsu.edu}}
\date{May 6, 2025}
\maketitle

\begin{abstract}
	We study load balancing for a queueing system where parallel stations are distant from customers. In the presence of traveling delays, the join-the-shortest-queue (JSQ) policy induces queue length oscillations and prolongs the mean waiting time. A variant of the JSQ policy, dubbed the randomized join-the-shortest-queue (RJSQ) policy, is devised to mitigate the oscillation phenomenon. By the RJSQ policy, customers are sent to each station with a probability approximately proportional to its service capacity; only a small fraction of customers are purposely routed to the shortest queue. The additional probability of routing a customer to the shortest queue, referred to as the balancing fraction, dictates the policy's performance. When the balancing fraction is within a certain range, load imbalance between the stations is negligible in heavy traffic, so that complete resource pooling is achieved. We specify the optimal order of magnitude for the balancing fraction, by which heuristic formulas are proposed to fine-tune the RJSQ policy. A joint problem of capacity planning and load balancing is considered for geographically separated stations. With well planned service capacities, the RJSQ policy sends all but a small fraction of customers to the nearest stations, rendering the system asymptotically equivalent to an aggregated single-server system with all customers having minimum traveling delays. If each customer's service requirement does not depend on the station, the RJSQ policy is asymptotically optimal for reducing workload.
\end{abstract}
\section{Introduction}
\label{sec:Introduction}

A \emph{remote queue} refers to a queueing system that is distant from customers when they decide to get service there. It takes time for a customer to reach a remote queue. The amount of time from departing for until arriving at the remote queue is referred to as the customer's \emph{traveling delay}. If a customer can get service from multiple remote queues, both the traveling delays to these queues and the waiting times for service should be considered to determine the destination.

We study a service system consisting of multiple stations, where customers can be served by any station. The stations are away from customers when they join the system; there is a \emph{dispatcher} who sends each customer to one of the stations. Given the presence of traveling delays, a routing decision must be made based on information available before the customer's departure for a station, although it may become obsolete when the customer arrives there. In this sense, the dispatcher has to rely on \emph{delayed} information to determine destinations. As shown below, an optimal routing policy for queues without information delays may not work well for remote queues.
\begin{figure}
	\caption{Oscillation Phenomenon Caused by JSQ}
	\centering
	\begin{subfigure}[b]{0.495\textwidth}
		\includegraphics[width=3.2in]{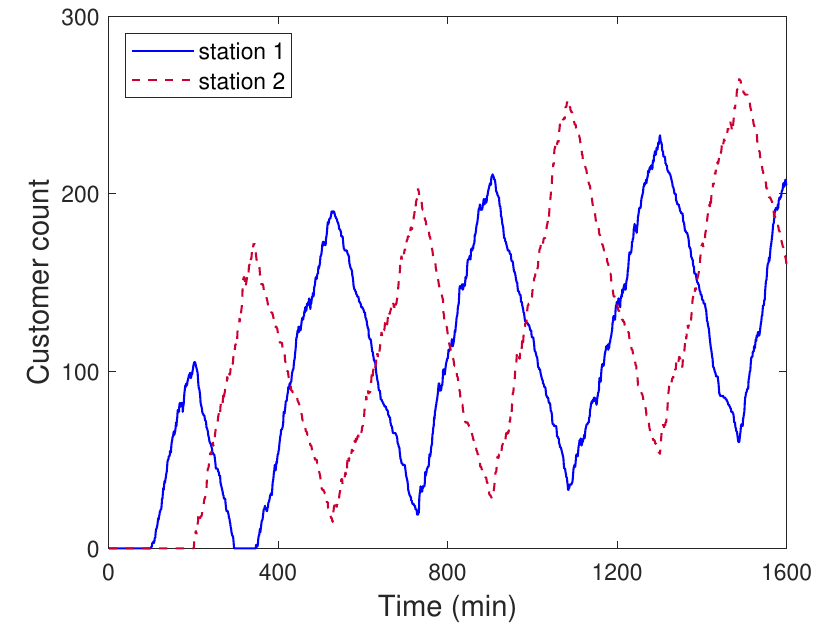}
		\caption{{\footnotesize Select station~1 if customer counts are equal}}
		\label{fig:oscillation-1}
	\end{subfigure}
	\begin{subfigure}[b]{0.495\textwidth}
		\includegraphics[width=3.2in]{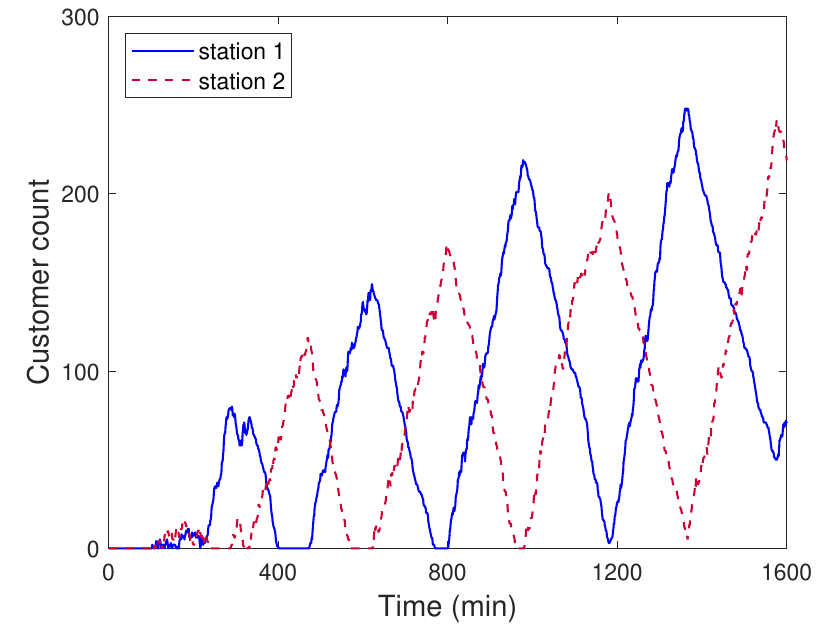}
		\caption{{\footnotesize Select a station at random if customer counts are equal}}
		\label{fig:oscillation-random}
	\end{subfigure}\label{fig:oscillations}
\end{figure}

Consider a system of two stations. It takes a customer 100 minutes to reach either station, which has a single server and unlimited waiting space, serving customers on the first-come, first-served (FCFS) basis. The service times are exponentially distributed with mean 1 minute in both stations. Customers join the system according to a Poisson process with rate 1.98 per minute; hence, the traffic intensity is 0.99. When a customer joins the system, the dispatcher checks the numbers of customers (i.e., \emph{customer counts}) in the two stations and sends the customer to the station having fewer customers---that is, the \emph{join-the-shortest-queue} (JSQ) policy is used. For the time being, we assume the customer is sent to station~1 if the customer counts are equal. 

Suppose both stations are initially empty. Within the first 100 minutes, all customers will be sent to station~1, making its customer count keep increasing during minutes 100--200. When the first customer arrives at station~1, incoming customers begin to head for station~2. Since station~2 is empty before minute 200, almost all incoming customers during minutes 100--200 will be sent to station~2, making its customer count keep increasing during minutes 200--300. The customer count in station~1 will begin decreasing after minute 200, but not until the customer count in station~2 becomes greater, will customers be sent to station~1. In this way, the customer count processes in the two stations will \emph{oscillate} in opposite directions, as is evident in Figure~\ref{fig:oscillation-1}. This phenomenon poses operational issues, because there could be many customers waiting in one station while the other is empty; consequently, the system's mean waiting time will be prolonged. One may question whether the oscillation issue is caused by the dispatcher's preference for station~1 when the customer counts are equal. In Figure~\ref{fig:oscillation-random}, we plot a pair of sample paths, assuming the dispatcher selects a station at random when equal customer counts are observed. Though not showing up at the beginning, the oscillation phenomenon is triggered once the difference between the customer counts gets noticeable; then, the customer count processes repeat the same pattern as in Figure~\ref{fig:oscillation-1}.
\begin{figure}[t]
	\caption{Average Customer Count under JSQ with Different Traveling Delays}
	\centering
	\includegraphics[width=3.2in]{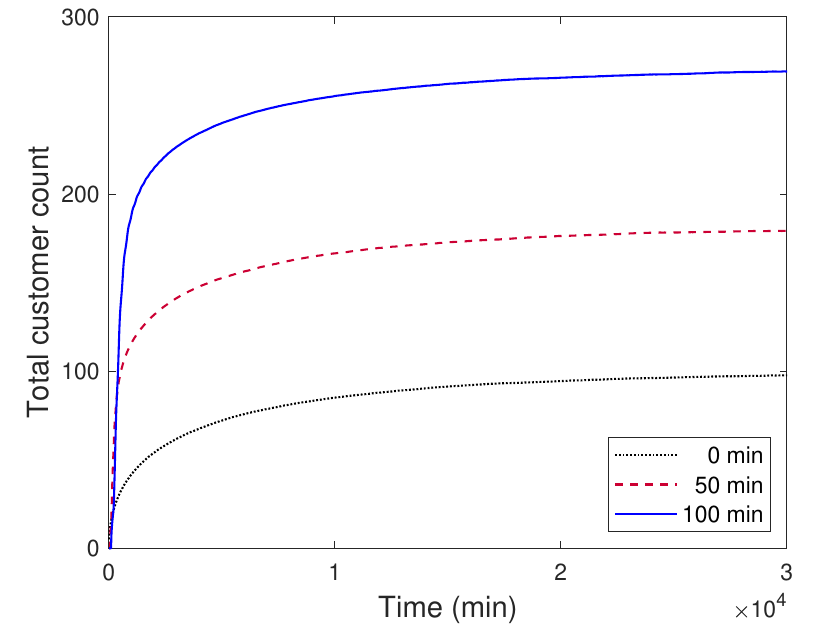}
	\label{fig:mean-paths}
\end{figure}

Without traveling delays, the JSQ policy can minimize both the system's mean waiting time and individual customers' expected waiting times, when the service time distribution has a nondecreasing hazard rate \citep{Winston.1977,Weber.1978, Whitt.1986}. The above example, however, illustrates that the JSQ policy induces queue length oscillations in the presence of traveling delays. Figure~\ref{fig:mean-paths} depicts how the total customer count (averaged over $ 3.0 \times 10^5 $ sample paths) in this two-station system evolves under the JSQ policy. When the traveling delays are 100 minutes, the average total customer count in the steady state is nearly three times greater than that without delays. By Little's law, the mean waiting time is also nearly three times longer.

As pointed out by \citet{Mitzenmacher.2000}, the issue of queue length oscillations arises because information delays lead to \emph{herd behavior}---that is, herds of customers move toward the same locations. Nowadays, mobile apps allow customers to access real-time information about service systems away from them. Since they rely on delayed information, herd behavior may cause operational inefficiency. For instance, a major bank in Singapore allows its app users to check queue lengths at its branches across the island.\footnote{Also available at \url{https://queue.uob.com/}.} With this information, a customer may debate which branch to visit so as to avoid excessive waiting. If the traveling delay is considerable, it may not be best to choose the shortest queue, because there could be many customers on the way to this branch where a long queue could form soon. The oscillation issue has been observed in hospitals, where delay announcements based on moving averages in the past few hours may cause waiting time oscillations across emergency departments within a region \citep{Dong.et.al.2019}. Herd behavior induced by information delays also emerges in bike-sharing networks, where customers search for available bikes through an app; getting to a nearby stop at peak hours, they occasionally find all bikes gone \citep{Shu.et.al.2013,Kabra.et.al.2020}. In ride-hailing systems, many empty cars may drive toward the same region to find passengers, resulting in a temporary imbalance between supply and demand \citep{Braverman.et.al.2019,Ozkan.Ward.2020,Hosseini.et.al.2024,Alwan.et.al.2004}.

An effective routing policy for remote queues must mitigate the oscillation issue. Ideally, each station's workload should be balanced so that incoming customers' waiting times would be equalized across the stations. Then, all the servers would be busy as long as there are a number of customers in the system, as if there was only a single queue with all the servers pooled together as an aggregated server---this characteristic is known as \emph{complete resource pooling}. The purpose of this study is to devise, analyze, and optimize a routing policy that retains complete resource pooling in the presence of traveling delays.

The policy under consideration is dubbed the \emph{randomized join-the-shortest-queue} (RJSQ) policy, by which only a small fraction of customers are \emph{purposely} sent to the shortest queue: the dispatcher sends customers to each station with a probability roughly proportional to its service capacity, while the routing probability is \emph{slightly} elevated for the shortest queue. The additional probability of routing a customer to the shortest queue is referred to as the \emph{balancing fraction}. When the system is in heavy traffic and the balancing fraction is within a certain range, \emph{load imbalance} between the stations will be negligible (Theorem~\ref{theorem:load-balancing}), and the system will be asymptotically equivalent to an aggregated single-server system with synchronized customer arrivals (Theorem~\ref{theorem:workload-difference}). The latter result implies complete resource pooling; consequently, the scaled multidimensional queue length process will converge to a reflected Brownian motion in a one-dimensional subspace (Theorem~\ref{theorem:CRP}).

Owing to the system's complex dynamics, it is difficult to analyze the RJSQ policy, either exactly or asymptotically under a scaling scheme finer than the standard diffusion scaling. Hence, we cannot identify the optimal balancing fraction. Since complete resource pooling can be attained when the balancing fraction is within a range, the RJSQ policy's performance should be insensitive to the balancing fraction around its optimal value. It would be useful for fine-tuning the RJSQ policy, if the \emph{optimal order of magnitude} can be specified for the balancing fraction. We establish an upper bound for the load imbalance (Theorem~\ref{theorem:load-balancing-gap}) and obtain the optimal order of magnitude for the balancing fraction that minimizes the upper bound (Corollary~\ref{corollary:optimal-gap}). Using these results, we propose heuristic formulas to specify near-optimal balancing fractions (see \eqref{eq:root-excess} and \eqref{eq:reciprocal-root}).

If the stations are geographically separated, the \emph{time to service} (i.e., the sum of the traveling delay and the waiting time) would be a customer's main concern. We study a joint problem of capacity planning and load balancing where times to service are minimized by the RJSQ policy. Assuming the dispatcher is aware of each customer's origin, we solve a linear program to determine the service capacities of the stations, along with a \emph{routing plan} that allows most customers to be served by their nearest stations. Sending all but a small fraction of customers to the nearest stations, the RJSQ policy renders the system asymptotically equivalent to an aggregated single-server system with all customers having minimum traveling delays (Theorem~\ref{theorem:workload-optimality}); the additional customers sent to the shortest queue can be chosen from nearby origins, which add minimally to their traveling delays. If each customer's service requirement does not depend on the station, the RJSQ policy is  asymptotically optimal for reducing workload (Corollary~\ref{corollary:workload-optimality}).

To establish these results, one essential step is to analyze load imbalance between the stations (Theorems~\ref{theorem:load-balancing} and~\ref{theorem:load-balancing-gap}). Although it is possible to use weak convergence results to prove Theorem~\ref{theorem:load-balancing}, it is challenging to specify the order of magnitude of the load imbalance. We decompose each customer arrival process at a station into four parts (see \eqref{eq:Rmnk-decomposition}), which enable us to analyze randomness from different sources separately, and exploit strong approximations to estimate their variations and increments. Using Freedman's inequality, we prove a lemma on the magnitude of variations and increments of a square-integrable martingale (Lemma~\ref{lemma:martingale}). Both this lemma and a result on increments of the standard Brownian motion (Lemma~\ref{lemma:Wiener-intervals}) are vital for estimating how much the load imbalance may change over an arbitrary time interval. With these tools, we obtain the upper bound in Theorem~\ref{theorem:load-balancing-gap}, which is sufficiently tight for our heuristic purposes (Proposition~\ref{prop:lower-bound}).

Besides the proposed approach to load imbalance analysis, we construct two auxiliary systems that provide lower bounds for the system of remote queues (which is referred to as the \emph{distributed} system in Sections~\ref{sec:Diffusion-Limit} and~\ref{sec:Routing-Capacity-Planning}). Both lower bound systems have a single-server station with the aggregate service capacity of the distributed system. Customer arrivals at the station in the first system are synchronized with the distributed system; its \emph{stationed workload} can never exceed that in the distributed system (Proposition~\ref{prop:stationed-workload-comparision}). Customers in the second system arrive at the station according to their traveling delays to the nearest stations in the distributed system; its \emph{system workload} can never exceed that in the distributed system (Proposition~\ref{prop:JNQ-workload-comparision}). We prove the heavy-traffic limit in Theorem~\ref{theorem:CRP} and the asymptotic optimality result in Corollary~\ref{corollary:workload-optimality} by establishing asymptotic equivalence between the distributed system and the lower bound systems.

For ease of analysis, we assume customers must go to destinations specified by the dispatcher. This assumption is \emph{not} essential in practice, because the RJSQ policy works well for a range of balancing fractions: In a practical system, the ``dispatcher'' may just be a mobile app or a website that provides real-time information and recommends service facilities, not knowing whether a customer would follow the recommendation. Thankfully, as long as the dispatcher could induce some additional customers to choose the shortest queue, the system workload would be balanced and the oscillation phenomenon would be barely seen.

The RJSQ policy relies on the system's history only through queue length information upon each customer's appearance, not using past decisions. As mentioned above, customers may not follow the dispatcher's recommendations in practice, and the dispatcher may not know each customer's true destination. Hence, past recommendations may not be that useful. Even if customers would strictly follow the dispatcher's recommendations, incorporating historical information will result in a high-dimensional state space, an optimal routing policy in which would be complex. Such a policy, however, may not bring substantial improvement over the RJSQ policy. This is because when traveling delays are considerable, differences between the service completion processes will make up an uncontrollable part of the load imbalance, which cannot be reduced by a routing policy without knowing exact service times in advance (see Example~\ref{example:historical-information}).

The contributions of this study are twofold: First, we establish a tractable analytical framework, exploiting both weak convergence and strong approximation techniques, for queueing systems with information delays. The proposed approach enables us to analyze the order of magnitude of load imbalance between the stations, and thus to fine-tune the routing policy. We construct lower bound systems, which have much simpler dynamics, for the system of remote queues; they are useful for both performance analysis and optimization purposes. Second, we devise a simple and near-optimal routing policy, along with heuristic formulas for its key parameter. Without needing past information, the RJSQ policy is insensitive to random traveling delays when the balancing fraction is around the optimal value. The RJSQ policy is also convenient to implement in practice, especially for service systems with multiple geographically separated facilities.

The rest of this paper is organized as follows: Section~\ref{sec:Literature} is devoted to literature review. To introduce the RJSQ policy, we present the system model in Section~\ref{sec:Remote-Queues} and the heavy-traffic formulation in Section~\ref{sec:Heavy-Traffic}. Section~\ref{sec:Balancing-Fraction} illustrates two conditions the balancing fraction should satisfy in heavy traffic. In Section~\ref{sec:Load-Balancing}, we analyze load imbalance between the stations and obtain the optimal order of magnitude for the balancing fraction. These results are used in Section~\ref{sec:Diffusion-Limit} to prove complete resource pooling and obtain a diffusion limit. A joint problem of capacity planning and load balancing is studied in Section~\ref{sec:Routing-Capacity-Planning}. In Section~\ref{sec:Numerical}, we conduct numerical experiments and propose heuristic formulas for the balancing fraction. While the proofs of the theorems are sketched out in Section~\ref{sec:Sketches}, the complete proofs are collected in the appendices. The paper concludes in Section~\ref{sec:Conclusion}, where future research topics are discussed.

\section{Related Literature}
\label{sec:Literature}

We outline relevant studies to position our work in the voluminous literature in related fields. Some preliminary results of this study can be found in two PhD theses, \citet{Yu.2017} and \citet{Yang.2023}.

It is well known that herd behavior may appear in sequential decision making where individuals observe previous decisions and rationally ignore their private information---that is, an \emph{information cascade} occurs. Because the majority's decision does not reflect most individuals' private information, herd behavior is often error-prone \citep{Banerjee.1992,Bikhchandani.et.al.1992}. When customers can get service from multiple facilities, they may use queues to infer each facility's value of service and join longer queues regardless of their private information \citep{Veeraraghavan.Debo.2009,Veeraraghavan.Debo.2011}. In contrast to these studies, our work is concerned with herd behavior induced by information delays, rather than information cascades.

Queueing control using delayed information has been studied in the context of telecommunications networks and distributed computing \citep{Mirchandaney.et.al.1989,Mitzenmacher.2000,Ying.Shakkottai.2011}. Analyzing a queueing system with information delays is generally difficult under the Markovian framework, because keeping track of the system's history over delays may result in a high-dimensional state space. To circumvent the curse of dimensionality, most studies are focused on either optimization of simplistic models \citep{Altman.Nain.1992,Altman.Koole.1995,Kuri.Kumar.1995,Artiges.1995} or performance analysis under specific control policies \citep{Altman.et.al.1995,Litvak.Yechiali.2003}. \citet{Kuri.Kumar.1995} proved that in a two-station system with Bernoulli arrivals and geometrically distributed service times, the \emph{join-the-shortest-expected-queue} policy minimizes the discounted cost when information delays are one unit of time; this policy, however, is not optimal when information delays are longer. Since optimal policies had been largely unknown, load balancing using delayed information was proposed as an open problem in applied probability \citep{Lipshutz.2019}.

To allow for more general assumptions, recent studies have been focused on approximate models and asymptotic analysis subject to information delays. \citet{Pender.et.al.2017} analyzed a fluid model for a two-station, infinite-server system where customers choose their stations according to a multinomial logit model; they identified the range of delays under which the queue lengths oscillate indefinitely. \citet{Pender.et.al.2020} extended the result to an arbitrary number of stations and justified the fluid model by proving heavy-traffic limits. \citet{Anselmi.Dufour.2020} analyzed the \emph{power-of-$d$-choice} policy with memory, which leverages the most recent queue lengths in sampled servers; they proved a fluid limit for an underloaded system with many servers. In these studies, customers' waiting times are either zero or relatively short. In contrast, our work considers a critically loaded system with a fixed number of single-server stations, where both a long traveling delay and a long waiting time could be a customer's concerns.

\citet{Atar.Lipshutz.2021} analyzed the \emph{join-the-shortest-estimated-queue} policy that exploits past decisions; they proved a heavy-traffic limit for the queue length process, which is the solution to a constrained stochastic delay equation. The asymptotic regime in their paper is different from ours: Their mean service times are $O(1/n)$ unit of time in the $n$th system, while information delays are $O(1)$ (see their Assumptions~1 and~2). Under the heavy-traffic condition, a customer's waiting time would be $O(1/\sqrt{n})$---on a \emph{lower} order of magnitude than the information delay. In their regime, differences between the service completion processes over a delay interval would be $O(\sqrt{n})$, rendering the imbalance between the stations' customer counts on the same order; hence, their limiting queue length process does not undergo \emph{state space collapse}. In our paper, the mean service times are $O(1)$, while the traveling delays are $O(\sqrt{n})$ in the $n$th system. A customer's waiting time would be on the \emph{same} order of magnitude as the traveling delay. Since differences between the service completion processes are much smaller in our regime, the RJSQ policy can equalize workload across the stations, without needing past information.

Our study complements the literature on asymptotic analysis of the JSQ policy, particularly when the number of servers is fixed \citep{Foschini.Salz.1978,Reiman.1984a,Turner.2000,Eryilmaz.Srikant.2012,Banerjee.et.al.2024,Jhunjhunwala.et.al.2025}. The RJSQ policy is similar to the policies studied by \citet{Turner.2000} and \citet{Banerjee.et.al.2024}. \citet{Turner.2000} considered a system with two identical stations and three arrival streams: While the two dedicated streams are sent to their respective stations, the third stream is routed to the shorter queue. When the arrival rate of the third stream is $O(1/\sqrt{n})$, the two-dimensional scaled queue length process converges in distribution to a reflected Brownian motion in the positive quadrant as $n$ goes large. \citet{Banerjee.et.al.2024} analyzed a policy that sends customers to each station with a probability depending on the ranking of its queue length; the additional fraction of customers routed to the shortest queue is also $O(1/\sqrt{n})$. They proved a multidimensional diffusion limit, along with an interchange-of-limits result, for the ranked queue length process. In comparison with these two papers, our study reveals how the RJSQ policy mitigates the oscillation phenomenon caused by information delays. More importantly, we develop an approach to fine-tuning the policy.

Our study is also relevant to \emph{spatial queues} where customers appear at random locations and either customers or servers need to move to initiate service \citep{Massey.Whitt.1993,Altman.Levy.1994,Cinlar.1995,Kroese.Schmidt.1996,Huang.Serfozo.1999,Breuer.2003,Simatos.Tibi.2010,Aldous.2017,Carlsson.et.al.2024}. While most of these studies use random measures to keep track of customers' locations, our work does not require such a representation. Instead, we prove asymptotic equivalence between the distributed system and two lower bound systems, thus establishing complete resource pooling and asymptotic optimality under the RJSQ policy. The joint problem of capacity planning and load balancing in Section~\ref{sec:Routing-Capacity-Planning} is similar to that studied by \citet{Carlsson.et.al.2024}: They considered a service system with geographically separated stations in a region, where customers choose stations based on both travel distances and waiting times. In their system, real-time information is \emph{unavailable} to customers, who have to use mean waiting times to decide their destinations. \citet{Carlsson.et.al.2024} characterized the system's demand equilibria as a partition of the region into certain service zones; if the stations are appropriately capacitated, the best-possible social welfare can be achieved through decentralized customer behavior. In comparison with their study, our work takes real-time information into account. Using the RJSQ policy, their system may achieve not only the best-possible social welfare but also complete resource pooling.

Our approach relies on strong approximation techniques to fine-tune the RJSQ policy. A strong approximation usually requires more restrictive conditions than the corresponding weak convergence result, but in return, it provides pathwise convergence at a certain rate. Strong approximations have been used in asymptotic queueing analysis to obtain limit processes with rates of convergence \citep{Rosenkrantz.1980,Csorgo.et.al.1987b,Zhang.et.al.1990,Glynn.Whitt.1991,Horvath.1992,Chen.Mandelbaum.1994,Chen.1996,Zhang.1997,Guo.Liu.2015}. By strong approximations, one may exploit sample path arguments to prove limit theorems for complex systems such as nonstationary queues and networks \citep{Mandelbaum.Massey.1995,Mandelbaum.et.al.1998,Yin.Zhang.2007,Pender.Ko.2017,Chakraborty.Honnappa.2022}, state-dependent queues \citep{Mandelbaum.Pats.1998b}, and queues under non-standard temporal and spatial scalings \citep{Glynn.1998,Chen.Shen.2000}. Although a diffusion limit is obtained for the queue length process in Theorem~\ref{theorem:CRP}, we do not intend to quantify and optimize the rate of convergence, because how it affects the RJSQ policy's performance is unclear. We use strong approximations to quantify the load imbalance and fine-tune the balancing fraction (Corollary~\ref{corollary:optimal-gap}) so that the performance gap between the RJSQ policy and a universal lower bound is minimized (Corollary~\ref{corollary:workload-optimality}). Hence, our study enriches the literature on asymptotic queueing control.

\section{A Randomized Routing Policy}
\label{sec:Remote-Queues}

Consider a system of $s$ stations ($s \geq 2$), each of which has a single server and unlimited waiting space, serving customers on the FCFS basis. A dispatcher sends each customer to one of the stations, and the customer arrives after a random time. The servers cannot be idle if there are customers waiting, and the customers leave the system upon service completion. We assume service times in each station are independent and identically distributed (i.i.d.), with mean $ 1/\mu_{k} $ in station~$ k $ and $ 0 < \mu_{1} \leq \cdots \leq \mu_{s} $. The system's total service capacity is
\begin{equation}\label{eq:capacity}
	\mu = \sum_{k = 1}^{s} \mu_{k} .
\end{equation}
Let $ \{ \boldsymbol{\gamma}(j): j \in \mathbb{N} \} $ be a sequence of i.i.d.\ random vectors with $ \boldsymbol{\gamma}(j) = ( \gamma_{1}(j), \ldots, \gamma_{s}(j) ) $, where $ \gamma_{k}(j) $ is the $ j $th customer's traveling delay to station~$ k $. Assume $ \boldsymbol{\gamma}(j) $ follows a \emph{discrete} distribution,
\begin{equation}\label{eq:delay-distribution}
	\mathbb{P}[\boldsymbol{\gamma}(j) = \boldsymbol{d}_{m}] = p_{m} > 0 \quad \mbox{for $ m = 1, \ldots, b $,}
\end{equation}
with $ \boldsymbol{d}_{m} = (d_{m,1}, \ldots, d_{m,s}) \in \mathbb{R}_{+}^{s} $ and $ \sum_{m = 1}^{b} p_{m} = 1 $. This discrete distribution could be an approximation of a general distribution on $ \mathbb{R}_{+}^{s} $; we may also envision customers randomly appearing from $ b $ locations, with $ d_{m,k} $ being the traveling delay from location~$ m $ to station~$ k $. For ease of exposition, a customer is said to be from \emph{origin}~$ m $ if the traveling delay vector takes value $ \boldsymbol{d}_{m} $.

The dispatcher's routing policy is based on a \emph{routing plan}, $\{ r_{m,k} : m = 1,\ldots, b; k = 1, \ldots, s \}$, where $ r_{m,k} $ is the long-run proportion of customers from origin~$ m $ that are sent to station~$ k $. Clearly, 
\begin{equation}\label{eq:rmk-inequality}
	0 \leq r_{m,k} \leq 1 \quad \mbox{for $ m = 1, \ldots, b $ and $ k = 1, \ldots, s $,}
\end{equation}
and
\begin{equation}\label{eq:rmk-sum}
	\sum_{k = 1}^{s} r_{m,k} = 1 \quad \mbox{for $ m = 1, \ldots, b$.}
\end{equation}
In heavy traffic, the routing plan also satisfies
\begin{equation}\label{eq:rmk-heavy-traffic}
	\mu_{k} = \sum_{m = 1}^{b} p_{m} r_{m,k}\mu \quad\mbox{for $ k = 1, \ldots, s $.}
\end{equation} 
A customer from origin~$ m $ should be sent to station~$ k $ with a probability close to $ r_{m,k} $.

We allow the dispatcher to be either \emph{aware} or \emph{unaware} of the origins of customers. If the origins are unknown, the dispatcher has to rely on a routing plan that does not differentiate the origins---that is, $ r_{1,k} = \cdots = r_{b,k} $ for all $ k = 1, \ldots, s $. The only routing plan that satisfies both this requirement and \eqref{eq:rmk-heavy-traffic} is given by $ r_{m,k} = \mu_{k}/\mu $. If the origins are known, the dispatcher may use a routing plan either estimated from historical data or determined by certain criteria. For instance, to minimize the mean traveling delay, the dispatcher may obtain a routing plan by solving the following linear program:
\[
\begin{array}{l@{\quad}l@{}}
	\textnormal{minimize} & \displaystyle \sum_{k = 1}^{s} \sum_{m = 1}^{b} p_{m} r_{m,k} d_{m,k} \\
	\textnormal{subject to} & \textnormal{\eqref{eq:rmk-inequality}--\eqref{eq:rmk-heavy-traffic}.}
\end{array}
\]

Let $ Q_{k}(t) $ be the customer count in station~$ k $ at time $ t $. The dispatcher monitors customer counts in the stations to determine each customer's destination. The dispatcher may use the \emph{weighted queue length} (or simply the \emph{queue length}), defined by $ L_{k}(t) = Q_{k}(t)/\mu_{k} $, to estimate the waiting time of a (virtual) customer who arrives at station~$ k $ at time~$ t $. When the JSQ policy is used, every customer is sent to the shortest queue---that is, the $ j $th customer is sent to station~$ k $ only if $ L_{k}(a(j)-) \leq L_{\ell}(a(j)-) $ for all $ \ell \neq k $, where $ a(j) $ is the time the $ j $th customer appears. 

As illustrated by Figure~\ref{fig:oscillations}, the JSQ policy causes the oscillation issue because many customers are sent to the same station within a traveling delay. To mitigate this problem, the dispatcher should send part of the customers, rather than all of them, to the shortest queue. If station~$ k $ has the shortest queue upon a customer's appearance from origin~$ m $, the dispatcher may send the customer there with a probability \emph{slightly} greater than $ r_{m,k} $ (with $ r_{m,k} = \mu_{k}/\mu $ if the origin is unknown) so that workload may still be balanced across the stations. We call such a routing policy the \emph{randomized join-the-shortest-queue} (RJSQ) policy. Specifically, upon the $ j $th customer's appearance, the dispatcher ranks the stations in increasing order of queue length and obtains $ ( \zeta_{1}(j), \ldots, \zeta_{s}(j) ) $, a permutation of $ \{1, \ldots, s\} $, such that $ L_{\zeta_{1}(j)}(a(j)-) \leq \cdots \leq L_{\zeta_{s}(j)}(a(j)-) $. (The dispatcher may follow any rule to break ties when two or more stations have equal queue lengths.) In other words, station~$ \zeta_{\ell}(j) $ has the $ \ell $th shortest queue. We use $ \pi_{k}(j) $ to denote the position of $ k $ in $ ( \zeta_{1}(j), \ldots, \zeta_{s}(j) ) $---that is, $ \pi_{k}(j) = \ell $ if and only if $ \zeta_{\ell}(j) = k $. If the customers' origins are unknown, the dispatcher randomly chooses $ \xi(j)  $, the $ j $th customer's destination, according to the following conditional distribution:
\begin{equation}\label{eq:conditional-distribution}
	\mathbb{P}[\xi(j) = k|\pi_{1}(j) = \ell_{1}, \ldots, \pi_{s}(j) = \ell_{s}] = \frac{\mu_{k}}{\mu} + \varepsilon_{\ell_{k}} \quad\mbox{for $ k = 1, \ldots, s $,}
\end{equation}
where $ ( \ell_{1}, \ldots, \ell_{s} ) $ is a permutation of $ \{1, \ldots, s\} $ and $ \varepsilon_{\ell} $ is the perturbation coefficient for routing a customer to the $ \ell $th shortest queue. Clearly, $ \sum_{\ell = 1}^{s} \varepsilon_{\ell} = 0 $. Moreover, we assume the perturbation coefficients satisfy the following two conditions:
\begin{enumerate}[label=(\roman*),ref=\roman*]
	\item \label{item:cond-1} $ \varepsilon_{1} = \chi > 0 $ and $ \varepsilon_{2}, \ldots, \varepsilon_{s} < 0 $;
	\item \label{item:cond-2} $ 0 \leq \mu_{k}/\mu + \varepsilon_{\ell} \leq 1 $ for $ k, \ell = 1, \ldots, s $.
\end{enumerate}
By this policy, the $ j $th customer will be sent to station~$ \zeta_{1}(j) $, which has the shortest queue, with probability $ \mu_{\zeta_{1}(j)}/\mu + \chi $, and to station~$ \zeta_{\ell}(j) $ with a probability less than $ \mu_{\zeta_{\ell}(j)}/\mu $ for $ \ell = 2, \ldots, s $. We refer to the constant $ \chi $ as the \emph{balancing fraction}. 

If the customers' origins are known, the dispatcher may take this information into account and determine the $ j $th customer's destination by
\begin{equation}\label{eq:conditional-distribution-m}
	\mathbb{P}[\xi(j) = k | \pi_{1}(j) = \ell_{1}, \ldots, \pi_{s}(j) = \ell_{s}, \boldsymbol{\gamma}(j) = \boldsymbol{d}_{m}] = r_{m,k} + \varepsilon^{m}_{\ell_{k}}(j)
\end{equation}
for $ m = 1, \ldots, b $ and $ k = 1, \ldots, s $, where $ \varepsilon^{m}_{\ell}(j) $ is the perturbation coefficient for routing the $ j $th customer from origin~$ m $ to the $ \ell $th shortest queue, with 
\begin{equation}\label{eq:sum-epsilon-m}
	\sum_{\ell = 1}^{s} \varepsilon^{m}_{\ell}(j) = 0 . 
\end{equation}
Since $ \boldsymbol{\gamma}(j) $ is independent of $ \pi_{1}(j), \ldots, \pi_{s}(j) $, then by \eqref{eq:delay-distribution} and \eqref{eq:rmk-heavy-traffic}--\eqref{eq:conditional-distribution-m}, 
\begin{equation}\label{eq:perturbation-parameters}
	\varepsilon_{\ell} = \sum_{m = 1}^{b} p_{m}\varepsilon^{m}_{\ell}(j) \quad\mbox{for $ \ell = 1, \ldots, s $ and $ j \in \mathbb{N} $.}
\end{equation}
In addition, we assume that for $ m = 1, \ldots, s $ and $ j \in \mathbb{N} $, the perturbation coefficients for routing the $ j $th customer from origin~$ m $ satisfy the following two conditions:
\begin{enumerate}[label=(\roman*),ref=\roman*,resume]
	\item \label{item:cond-1-m} $ \varepsilon^{m}_{1}(j) \geq 0 $ and $ \varepsilon^{m}_{2}(j), \ldots, \varepsilon^{m}_{s}(j) \leq 0 $;
	\item \label{item:cond-2-m} $ \varepsilon^{m}_{1}(j), \ldots, \varepsilon^{m}_{s}(j) $ are determined by $ \pi_{1}(j), \ldots, \pi_{s}(j) $ subject to $ 0 \leq r_{m,k} + \varepsilon^{m}_{\pi_{k}(j)}(j) \leq 1 $ for $ k = 1, \ldots, s $.
\end{enumerate}
Given $ \varepsilon_{1}, \ldots, \varepsilon_{s} $ that satisfy conditions~\eqref{item:cond-1} and \eqref{item:cond-2}, it is simple to find $ \varepsilon^{m}_{1}(j), \ldots, \varepsilon^{m}_{s}(j) $ that satisfy conditions \eqref{item:cond-1-m} and \eqref{item:cond-2-m} subject to \eqref{eq:sum-epsilon-m} and \eqref{eq:perturbation-parameters}. By conditions \eqref{item:cond-1}, \eqref{item:cond-1-m}, and \eqref{eq:perturbation-parameters}, 
\begin{equation}\label{eq:pm-epsilon-chi}
	|p_{m}\varepsilon^{m}_{\ell}(j)| \leq \chi \quad\mbox{for $ m = 1, \ldots, b $, $ \ell = 1, \ldots, s $, and $ j \in \mathbb{N} $.}
\end{equation}

\section{Randomized Routing in Heavy Traffic}
\label{sec:Heavy-Traffic}

Because exact analysis is generally intractable, we analyze the RJSQ policy asymptotically by considering a sequence of systems in heavy traffic. The systems are indexed by $ n \in \mathbb{N} $, each having $ s $ stations and all sharing a common routing plan. All of them are initially empty and driven by a collection of mutually independent random variables and random vectors as described below. By convention, we add a subscript $ n $ to the quantities that may change with the system index.

For $ k = 1, \ldots, s $, let $ \{ w_{k}(i) : i \in \mathbb{N} \} $ be a sequence of i.i.d.\ nonnegative random variables with mean $ 1 $ and coefficient of variation $ c_{k} < \infty $, where $ w_{k}(i) $ is the amount of service required by the $ i $th customer served by station~$ k $. Let $ S_{k}(t) = \max \{ j \in \mathbb{N}_{0} : \sum_{i = 1}^{j} w_{k}(i)/\mu_{k} \leq t \} $ for $ t \geq 0 $, which is the number of service completions in station~$ k $ if the server has been working $ t $ units of time.

Customers join the $ n $th system according to a renewal process with rate $ \lambda_{n} > 0$. Let $ \{ z(j) : j \in \mathbb{N} \} $ be a sequence of i.i.d.\ nonnegative random variables with mean $ 1 $ and coefficient of variation $ c_{0} < \infty $, where $ z(j) $ is the normalized inter-appearance time between the $ (j-1) $st and $ j $th customers. Then, $ E_{n}(t) = \max \{ i \in \mathbb{N}_{0} : \sum_{j = 1}^{i} z(j)/\lambda_{n} \leq t \} $ is the number of customers that have joined the system by time $t$. Assume the sequence of systems is in heavy traffic, in the sense that
\begin{equation}\label{eq:heavy-traffic}
	\lim_{n \to \infty} \sqrt{n} (1 - \rho_{n}) = \beta \quad\mbox{for $ \beta \in \mathbb{R} $,}
\end{equation}
where $ \rho_{n} = \lambda_{n}/\mu $ is the traffic intensity of the $n$th system. The systems are stable only if $ \beta > 0 $. In this case, the mean waiting time is expected to be $ O(\sqrt{n}) $ units of time in the $ n $th system (i.e., on the same order of magnitude as $1/(1-\rho_{n})$) when an appropriate routing policy is used.

We focus on the case where customers' traveling delays are on the same order of magnitude as their potential waiting times, because otherwise the routing problem could be practically trivial: If a customer's traveling delays are significantly longer than the potential waiting times in the stations, the long trip would be the main concern and the customer would simply head for the nearest station. If the traveling delays are negligible in comparison with the potential waiting times, the best option would probably be joining the shortest queue. Since traveling delays are assumed to be $ O(\sqrt{n}) $ units of time in the $ n $th system, we scale up the primitive delay vector by
\begin{equation}\label{eq:commute-delays}
	\boldsymbol{\gamma}_{n}(j) = \sqrt{n} \boldsymbol{\gamma}(j) \quad\mbox{for $ j \in \mathbb{N} $,}
\end{equation}
where $ \boldsymbol{\gamma}_{n}(j) = ( \gamma_{n,1}(j), \ldots, \gamma_{n,s}(j) ) $ is the delay vector of the $ j $th customer in the $n$th system.

The RJSQ policy is used in the $ n $th system with the balancing fraction $ \chi_{n} > 0 $. When the $ j $th customer appears, the dispatcher ranks the $ s $ stations as $ ( \zeta_{n,1}(j), \ldots, \zeta_{n,s}(j) ) $, with station~$ \zeta_{n,\ell}(j) $ having the $ \ell $th shortest queue. Let $ \pi_{n,k}(j) $ be the position of $ k $ in $ ( \zeta_{n,1}(j), \ldots, \zeta_{n,s}(j) ) $. Depending on whether the customers' origins are known, the dispatcher may determine $ \xi_{n}(j) $, the $ j $th customer's destination, by either \eqref{eq:conditional-distribution} or \eqref{eq:conditional-distribution-m}, with the perturbation coefficients denoted by $ \varepsilon_{n,1}, \ldots, \varepsilon_{n,s} $ or $ \varepsilon^{m}_{n,1}(j), \ldots, \varepsilon^{m}_{n,s}(j) $, respectively. Specifically, the dispatcher generates a sequence of i.i.d.\ standard uniform random variables $ \{ u(j) : j \in \mathbb{N} \} $. If the origin is unknown, the dispatcher will send the $ j $th customer to station~$ k $ if and only if $ u(j) \in [\kappa_{n,k-1}(j), \kappa_{n,k}(j)) $, where
\begin{equation}\label{eq:kappa} 
	\kappa_{n,0}(j) = 0 \quad \mbox{and} \quad \kappa_{n,k}(j) = \sum_{\ell = 1}^{k} \Big(\frac{\mu_{\ell}}{\mu} + \varepsilon_{n, \pi_{n,\ell}(j)}\Big) \quad\mbox{for $ k = 1, \ldots, s $.}
\end{equation}
The $ j $th customer's destination is thus given by
\begin{equation}\label{eq:destination-n} 
	\xi_{n}(j) = \sum_{k = 1}^{s} k \cdot \mathbb{1}_{ \{ \kappa_{n,k-1}(j) \leq u(j) < \kappa_{n,k}(j) \} }.
\end{equation}
If the origin is known, the dispatcher will determine the $ j $th customer's destination by
\begin{equation}\label{eq:destination-n-m} 
	\xi_{n}(j) = \sum_{m = 1}^{b} \sum_{k = 1}^{s} k \cdot \mathbb{1}_{ \{ \kappa^{m}_{n,k-1}(j) \leq u(j) < \kappa^{m}_{n,k}(j), \boldsymbol{\gamma}(j) = \boldsymbol{d}_{m} \} },
\end{equation}
where 
\begin{equation}\label{eq:kappa-m} 
	\kappa^{m}_{n,0}(j) = 0 \quad \mbox{and} \quad \kappa^{m}_{n,k}(j) = \sum_{\ell = 1}^{k} \big(r_{m,\ell} + \varepsilon^{m}_{n, \pi_{n,\ell}(j)}(j)\big) \quad \mbox{for $ m = 1, \ldots, b $ and $ k = 1, \ldots, s $.}
\end{equation}
In addition to conditions~(\ref{item:cond-1}) and (\ref{item:cond-2}), we assume  
\begin{enumerate}[label=(\roman*),ref=\roman*,resume]
	\item \label{item:cond-1-n} there exists a positive number $ \delta_{0} $ such that $\varepsilon_{n,\ell} \leq - \delta_{0} \chi_{n}$ for $\ell = 2, \ldots, s$ and $n \in \mathbb{N}$.
\end{enumerate}

Let $ A_{n,k}(t) $ be the number of customers that have arrived at station~$ k $ by time~$t$, and $ B_{n,k}(t) $ be the server's cumulative busy time by time $t$. Then, the customer count is given by
\begin{equation}\label{eq:queue-length}
	Q_{n,k}(t) = A_{n,k}(t) - S_{k}(B_{n,k}(t)).
\end{equation}
We define the scaled customer count and the scaled weighted queue length by
\begin{equation}\label{eq:diffusion-scaled}
	\tilde{Q}_{n,k}(t) = \frac{1}{\sqrt{n}}Q_{n,k}(nt) \quad \mbox{and} \quad \tilde{L}_{n,k}(t) = \frac{1}{\mu_{k}}\tilde{Q}_{n,k}(t).
\end{equation}

\section{Balancing Fraction}
\label{sec:Balancing-Fraction}

In this section, we specify two conditions the balancing fraction should satisfy to mitigate queue length oscillations. Rather than making a tedious analysis, we illustrate how the balancing fraction may affect the oscillation phenomenon by a simple example.

Consider a sequence of systems that satisfies \eqref{eq:heavy-traffic}. Each system has two stations with  $ \mu_{1} = \mu_{2} = \mu/2 $ and $ c^{2}_{1} = c^{2}_{2} $. We wish to attain $ \tilde{Q}_{n,1}(t) \approx \tilde{Q}_{n,2}(t) $ when $n$ is large. For the time being, assume traveling delays are zero. Under the JSQ policy, the two-dimensional scaled customer count process converges in distribution to a reflected Brownian motion in a one-dimensional subspace:
\begin{equation}\label{eq:two-station-weak-convergence}
	(\tilde{Q}_{n,1}, \tilde{Q}_{n,2}) \Rightarrow \Big( \frac{\tilde{Q}}{2}, \frac{\tilde{Q}}{2} \Big) \quad \mbox{as $ n \to \infty $},
\end{equation}
where $ \tilde{Q} $ is a one-dimensional reflected Brownian motion with drift $ -\beta \mu $ and variance $ \mu(c^{2}_{0} + c^{2}_{1}) $ (\citealp{Reiman.1984a}). This phenomenon is known as \emph{state space collapse} and can be explained as follows: As seen from \eqref{eq:diffusion-scaled}, $\tilde{Q}_{n,1}$ and $\tilde{Q}_{n,2}$ are scaled and accelerated versions of $Q_{n,1}$ and $Q_{n,2}$, evolving $ n $ times faster. Suppose that at time~$ t $, the difference between the customer counts becomes ``noticeable''---that is, $ |Q_{n,1}(t) - Q_{n,2}(t)| \approx \delta \sqrt{n}$ for some $ \delta > 0 $. Sending all to the shorter queue, the JSQ policy can equalize $Q_{n,1}$ and $Q_{n,2}$ within $ O(\sqrt{n}) $ units of time---or equalize the accelerated processes $\tilde{Q}_{n,1}$ and $\tilde{Q}_{n,2}$ within $ O(1/\sqrt{n}) $ unit of time. As $n$ goes large, the JSQ policy ``instantaneously'' wipes out a noticeable difference between $\tilde{Q}_{n,1}$ and $\tilde{Q}_{n,2}$. Thus, the two-dimensional scaled customer count process converges to a one-dimensional process along a line.

As illustrated in Section~\ref{sec:Introduction}, the JSQ policy may send most customers to the same station within a traveling delay. If traveling delays are $ O(\sqrt{n}) $, from time to time, there are $ O(\sqrt{n}) $ more customers on trips to one station than to the other, which could be too many to equalize the customer counts. Hence, the queue that was shorter may become much longer than the other periodically. To mitigate this issue, we should send fewer customers to the shorter queue so that the difference between the customer counts will not change too much within a delay interval. Since the customer counts are $ O(\sqrt{n}) $, the additional number of customers sent to the shorter queue should be $ o(\sqrt{n}) $ within a traveling delay (i.e., within $ O(\sqrt{n}) $ units of time), which requires the balancing fraction to satisfy 
\begin{equation}\label{eq:epsilon-1}
	\lim_{n \to \infty} \chi_{n} = 0.
\end{equation}
On the other hand, the RJSQ policy should ``quickly'' equalize the customer counts. If $ |Q_{n,1}(t) - Q_{n,2}(t)| \approx \delta \sqrt{n} $, the RJSQ policy equalizes $Q_{n,1}$ and $Q_{n,2}$ within $ O(\sqrt{n}/\chi_{n}) $ units of time, or equalizes $\tilde{Q}_{n,1}$ and $\tilde{Q}_{n,2}$ within $ O(1/(\sqrt{n}\chi_{n})) $ unit of time. Taking $ 1/(\sqrt{n}\chi_{n}) = o(1) $, we have
\begin{equation}\label{eq:epsilon-2}
	\lim_{n \to \infty} \sqrt{n} \chi_{n} = \infty .
\end{equation}

According to the above discussion, we expect \eqref{eq:two-station-weak-convergence} still holds in the presence of traveling delays, if the RJSQ policy is used with balancing fractions satisfying \eqref{eq:epsilon-1} and \eqref{eq:epsilon-2}. Let us re-examine the numerical example in Section~\ref{sec:Introduction}: Rather than using the JSQ policy, each customer is now sent to the shorter queue with probability 0.55 and to the longer queue with probability 0.45 ($\chi = 0.05$). A pair of sample paths of the customer count processes is plotted in Figure~\ref{fig:paths-RJSQ}, where no clear pattern of oscillations is observed. Figure~\ref{fig:mean-paths-RJSQ} depicts how the total customer count (averaged over $ 3 \times 10^5 $ sample paths) evolves under the RJSQ policy with $\chi = 0.05$. Though slightly inferior to the JSQ policy when the traveling delays are zero, the RJSQ policy appears insensitive to longer delays. When the traveling delays are 100 minutes, the average total customer count in the steady state is about 11\% greater than that with zero delays.
\begin{figure}[t]
	\begin{minipage}{0.495\textwidth}
		\caption{Customer Counts under RJSQ}
		\centering
		\includegraphics[width=3.2in]{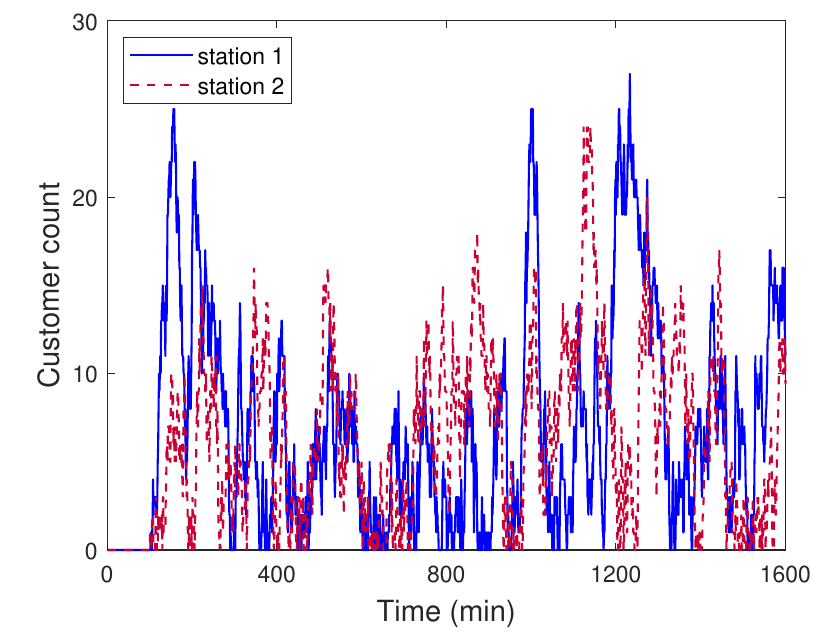}
		\label{fig:paths-RJSQ}
	\end{minipage}
	\begin{minipage}{0.495\textwidth}
		\caption{Average Customer Count under RJSQ}
		\centering
		\includegraphics[width=3.2in]{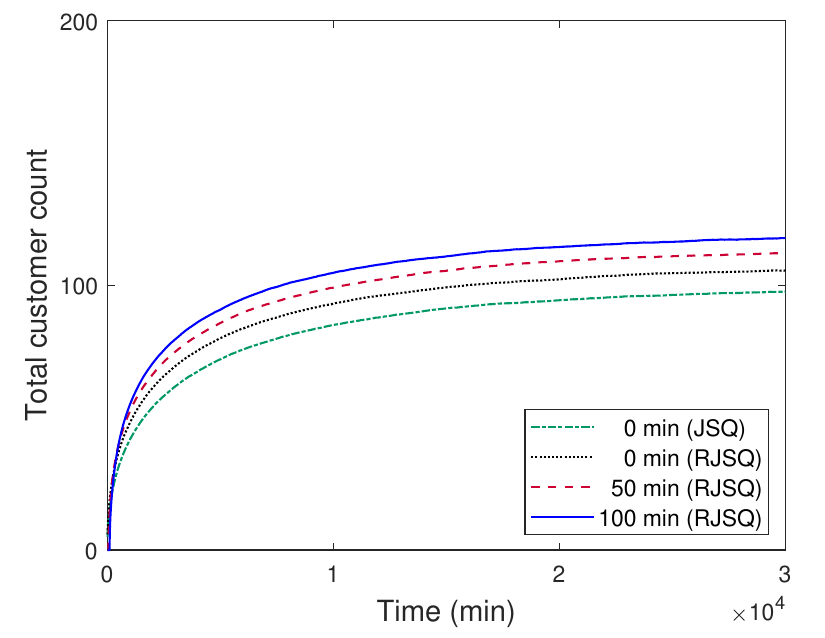}
		\label{fig:mean-paths-RJSQ}
	\end{minipage}
\end{figure}

\section{Analysis of Load Imbalance}
\label{sec:Load-Balancing}

Let us revisit the heavy-traffic setting introduced in Section~\ref{sec:Heavy-Traffic}, where the dispatcher follows either \eqref{eq:destination-n} or \eqref{eq:destination-n-m} to determine destinations, depending on whether the origins are known. Our analysis begins with an estimation of \emph{load imbalance} between the stations, measured by 
\begin{equation}\label{eq:load-imbalance}
	\sup_{0 \leq t \leq nT} \max_{k,\ell = 1, \ldots, s} | L_{n,k}(t) - L_{n,\ell}(t) | \quad \mbox{for $T > 0$.}
\end{equation}
To mitigate the oscillation phenomenon, the load imbalance must be reduced.

The first theorem justifies our previous argument: if the balancing fraction satisfies \eqref{eq:epsilon-1} and \eqref{eq:epsilon-2}, the scaled load imbalance will converge to zero. Since weighted queue lengths are estimates of waiting times, the RJSQ policy will equalize waiting times across the stations. 

\begin{theorem}\label{theorem:load-balancing}
	Assume the RJSQ policy is used in a sequence of systems that are initially empty under condition \eqref{eq:heavy-traffic}, with $ \chi_{n} $ satisfying \eqref{eq:epsilon-1} and \eqref{eq:epsilon-2}. If $ \varepsilon_{n,1}, \ldots, \varepsilon_{n,s} $ satisfy conditions~\eqref{item:cond-1}, \eqref{item:cond-2}, and \eqref{item:cond-1-n} (along with conditions~\eqref{item:cond-1-m} and \eqref{item:cond-2-m} when the origins are known), then for $ T > 0 $,
	\[  
	\sup_{0 \leq t \leq T} \max_{k,\ell = 1, \ldots, s} | \tilde{L}_{n,k}(t) - \tilde{L}_{n,\ell}(t) | \Rightarrow 0  \quad \mbox{as $ n \to \infty $.}
	\]
\end{theorem}

Conditions~\eqref{eq:epsilon-1} and \eqref{eq:epsilon-2} provide a \emph{range} for the balancing fraction---for instance, we may take $\chi_{n} = n^{-\alpha}$ for $ \alpha \in (0, 1/2) $. Although the scaled load imbalance is negligible for any balancing fraction within the range, it would be useful if the order of magnitude of the load imbalance can be specified. To this end, we strengthen Theorem~\ref{theorem:load-balancing} by giving an \emph{almost sure} upper bound. The refined result requires two additional assumptions: First, both the normalized inter-appearance time distribution and the normalized service time distributions have finite fourth moments,
\begin{equation}\label{eq:fourth-moments}
	\mathbb{E}[z(1)^{4}] + \sum_{k = 1}^{s} \mathbb{E}[w_{k}(1)^{4}] < \infty.
\end{equation}
Second, the balancing fraction satisfies
\begin{equation}\label{eq:log-log-n-bound}
	\lim_{n \to \infty} \frac{\sqrt{n} \chi_{n}}{\sqrt{\log \log n}} = \infty.
\end{equation}

\begin{theorem}\label{theorem:load-balancing-gap}
	Assume the conditions of Theorem~\ref{theorem:load-balancing} hold. If conditions~\eqref{eq:fourth-moments} and \eqref{eq:log-log-n-bound} also hold, then there exist two positive numbers $ C $ and $ C' $ such that with probability 1,
	\begin{equation}\label{eq:upper-bound}
		\sup_{0 \leq t \leq T} \max_{k,\ell = 1, \ldots, s} | \tilde{L}_{n,k}(t) - \tilde{L}_{n,\ell}(t) | < \max \Big\{ C \chi_{n}, C' \frac{\log \log n + \log (n \chi_{n}^{2})}{\sqrt{n} \chi_{n}} \Big\}
	\end{equation}
	for $ T > 0 $ and $ n $ sufficiently large. If condition~\eqref{eq:fourth-moments} holds and $ \lim_{n \to \infty} \sqrt{n} \chi_{n}/\log \log n = \infty$, then with probability 1,
	\begin{equation}\label{eq:almost-sure-convergence}
		\lim_{n \to \infty} \sup_{0 \leq t \leq T} \max_{k,\ell = 1, \ldots, s} | \tilde{L}_{n,k}(t) - \tilde{L}_{n,\ell}(t) | = 0.
	\end{equation}
\end{theorem}

\begin{remark}\label{remark:trade-off}
	To better understand the almost sure bound, we need to analyze the causes of load imbalance: On the one hand, the arrival and service completion processes in a station induce random variations in workload; by sending more customers to the shortest queue, the RJSQ policy compensates for imbalance more quickly with a larger balancing fraction---this effect is captured by the second term on the right side of \eqref{eq:upper-bound}. On the other hand, without knowing past routing decisions, the RJSQ policy may send $O(\sqrt{n}\chi_{n})$ more customers (or $O(\chi_{n})$ more under the diffusion scaling) to the shortest queue within a traveling delay, possibly resulting in queue length oscillations of the same order---this effect is captured by the first term on the right side of \eqref{eq:upper-bound}.
\end{remark}

Theorem~\ref{theorem:load-balancing-gap} poses a trade-off for fine-tuning the RJSQ policy. By minimizing the almost sure bound, we find the optimal order of magnitude for the balancing fraction.

\begin{corollary}\label{corollary:optimal-gap}
	Assume the conditions of Theorem~\ref{theorem:load-balancing} hold. If conditions~\eqref{eq:fourth-moments} and \eqref{eq:log-log-n-bound} also hold, then the order of the upper bound on the right side of \eqref{eq:upper-bound} is minimized with $ \chi_{n} = C'' n^{-1/4}\sqrt{\log n} $ for $ C'' > 0 $, in which case with probability 1,
	\begin{equation}\label{eq:optimal-bound}
		\sup_{0 \leq t \leq T} \max_{k,\ell = 1, \ldots, s} | \tilde{L}_{n,k}(t) - \tilde{L}_{n,\ell}(t) | = O\big(n^{-1/4}\sqrt{\log n}\big) \quad \mbox{for $T > 0$}.
	\end{equation}
\end{corollary}

\begin{remark}\label{remark:general-distribution}
	We assume the traveling delays are discrete random variables for ease of analysis. Since a general distribution can be approximated by discrete distributions, the main results in this study should hold under more general traveling delay distributions. Please check Section~\ref{sec:Numerical} for a numerical example where traveling delays follow a continuous distribution.
\end{remark}

It is difficult to prove the asymptotic tightness of the minimized upper bound in \eqref{eq:optimal-bound}. However, by the example below, we can show that under certain conditions, there exist two positive numbers $ \check{C} $ and $ \check{C}' $ such that with probability 1,
\begin{equation}\label{eq:lower-bound}
	\sup_{0 \leq t \leq T} \max_{k,\ell = 1, \ldots, s} | \tilde{L}_{n,k}(t) - \tilde{L}_{n,\ell}(t) | > \check{C} \chi_{n} + \check{C}' n^{-1/4}\sqrt{\log \log n}
\end{equation}
for infinitely many $ n \in \mathbb{N} $ and any $\chi_{n}$ that satisfies \eqref{eq:epsilon-1} and \eqref{eq:log-log-n-bound}. The gap between the upper and lower bounds is thus within a sublogarithmic factor of $ \sqrt{\log n/\log \log n} $.

\begin{example}\label{example:lower-bound}
	Consider a sequence of systems, each having a sole customer origin and two stations. The respective traveling delays to the two stations are $ \sqrt{n} d_{1,1} $ and $ \sqrt{n} d_{1,2} $ in the $ n $th system, with $ 0 < d_{1,1} < d_{1,2} $. When the queue lengths are equal, the dispatcher sends a customer to station~1 with probability $ \mu_{1}/\mu + \chi_{n} $ because the traveling delay is shorter. Since the system is initially empty, customers arriving at station~1 form a renewal process with rate $ \lambda_{n}(\mu_{1}/\mu + \chi_{n}) $ during $ (\sqrt{n}d_{1,1}, 2\sqrt{n}d_{1,1}] $, while station~2 remains empty by time $ \sqrt{n}d_{1,2} $. Let $ d' = 2d_{1,1} \wedge d_{1,2} $. The lower bound in \eqref{eq:lower-bound} follows from the proposition below.
\end{example}

\begin{proposition}\label{prop:lower-bound}
	Assume the conditions of Theorem~\ref{theorem:load-balancing} hold for the sequence of systems in Example~\ref{example:lower-bound}. If conditions~\eqref{eq:fourth-moments} and \eqref{eq:log-log-n-bound} also hold, then there exist two positive numbers $ \check{C} $ and $ \check{C}' $ such that with probability 1, $ \tilde{L}_{n,1}(n^{-1/2}d') > \check{C}\chi_{n} + \check{C}' n^{-1/4}\sqrt{\log \log n} $ for infinitely many $n \in \mathbb{N}$.
\end{proposition}

The RJSQ policy retains the capability of load balancing, but it cannot equalize workload as quickly as the JSQ policy. The next proposition is concerned with the JSQ policy when traveling delays are zero. To obtain a tighter upper bound, we need a condition stronger than \eqref{eq:fourth-moments}:
\begin{equation}\label{eq:exponential-bound}
	\mathbb{E}[\exp(t z(1))] + \sum_{k = 1}^{s} \mathbb{E}[\exp(t w_{k}(1))] < \infty \quad \mbox{in a neighborhood of $t = 0$.}
\end{equation}
\begin{proposition}\label{prop:JSQ}
	Assume the JSQ policy is used in a sequence of systems with zero traveling delays that are initially empty under condition \eqref{eq:heavy-traffic}. If condition~\eqref{eq:exponential-bound} also holds, then with probability~1,
	\[
	\sup_{0 \leq t \leq T} \max_{k,\ell = 1, \ldots, s} | \tilde{L}_{n,k}(t) - \tilde{L}_{n,\ell}(t) | = O(n^{-1/2}\log n) \quad \mbox{for $T > 0$.}
	\]
\end{proposition}

When traveling delays are zero, the JSQ policy maintains the load imbalance in \eqref{eq:load-imbalance} within a $O(\log n)$ bound; when traveling delays are $O(\sqrt{n})$, the RJSQ policy maintains the load imbalance within a $O(n^{1/4}\sqrt{\log n})$ bound. (The minimized upper bound in \eqref{eq:optimal-bound} will not change if condition~\eqref{eq:exponential-bound} is assumed in Corollary~\ref{corollary:optimal-gap}.) The difference between the two bounds is deemed the cost for mitigating queue length oscillations by sending fewer customers to the shortest queue.

The RJSQ policy relies on the system's history only through the queue length information upon each customer's appearance, not using any past decisions. If more historical information is taken into account, how much improvement can be made? The following example provides useful clues.

\begin{example}\label{example:historical-information}
	Consider a sequence of systems under condition \eqref{eq:heavy-traffic}. Each system has two stations with $\mu_{1} = \mu_{2} = 1$; all service times are exponentially distributed. In the $n$th system, customers appear from a sole origin and their traveling delays are $ \sqrt{n} d $ minutes to both stations, with $d > 0$. Suppose the dispatcher adopts a routing policy that may use historical information but \emph{cannot} know a customer's exact service time until completion. Assume $n$ is large and consider a moment~$\tau_{n}$ at which a customer joins the system and finds long queues in both stations---in the sense that both servers would be busy during $(\tau_{n}, \tau_{n} + \sqrt{n} \delta]$ for some $ \delta \in (0,d] $. By \eqref{eq:queue-length}, the difference between the queue lengths at time $\tau_{n} + \sqrt{n}\delta$ is
	\[
	L_{n,1}(\tau_{n}+\sqrt{n}\delta) - L_{n,2}(\tau_{n}+\sqrt{n}\delta) =  A_{n,1}(\tau_{n}+\sqrt{n}\delta) - A_{n,2}(\tau_{n}+\sqrt{n}\delta) - D_{n}(\tau_{n}) - D_{n}(\tau_{n}, \tau_{n}+\sqrt{n}\delta],
	\]
	where $ D_{n}(t) = S_{1}(B_{n,1}(t)) - S_{2}(B_{n,2}(t)) $ and $ D_{n}(t_{1}, t_{2}] = D_{n}(t_{2}) - D_{n}(t_{1})$ for $t \geq 0$ and $0 \leq t_{1} < t_{2}$. Because the servers would be busy, the respective service completions in the stations during $(\tau_{n}, \tau_{n} + \sqrt{n} \delta]$ are approximately two independent Poisson random variables, both having mean $\sqrt{n} \delta$ and being independent of the system's history until $\tau_{n}$. By the central limit theorem, $D_{n}(\tau_{n}, \tau_{n}+\sqrt{n}\delta] /(\sqrt{2\delta} n^{1/4})$ approximately follows a standard Gaussian distribution, and thus
	\[
	\mathbb{P}[D_{n}(\tau_{n}, \tau_{n}+\sqrt{n}\delta] > \sqrt{2\delta} n^{1/4}] \approx \mathbb{P}[D_{n}(\tau_{n}, \tau_{n}+\sqrt{n}\delta] < -\sqrt{2\delta} n^{1/4}] \approx 0.1587.
	\]
	With traveling delays of $\sqrt{n}d$ minutes, $ A_{n,1}(\tau_{n}+\sqrt{n}\delta) - A_{n,2}(\tau_{n}+\sqrt{n}\delta) $ is determined by the system's history until $ \tau_{n} - \sqrt{n}(d - \delta)$. Since the service completions during $(\tau_{n}, \tau_{n} + \sqrt{n} \delta]$ are approximately independent of the history until $\tau_{n}$, $D_{n}(\tau_{n}, \tau_{n}+\sqrt{n}\delta]$  constitutes an \emph{uncontrollable} part of the difference between the queue lengths at time $\tau_{n} + \sqrt{n}\delta$. Note that
	\begin{align*}
		& \mathbb{P}\big[|L_{n,1}(\tau_{n}+\sqrt{n}\delta) - L_{n,2}(\tau_{n}+\sqrt{n}\delta)| > \sqrt{2\delta} n^{1/4}\big] \\
		& \quad \geq \mathbb{P}\big[D_{n}(\tau_{n}, \tau_{n}+\sqrt{n}\delta] > \sqrt{2\delta} n^{1/4}, A_{n,1}(\tau_{n}+\sqrt{n}\delta) - A_{n,2}(\tau_{n}+\sqrt{n}\delta) - D_{n}(\tau_{n}) \leq 0 \big] \\
		& \qquad + \mathbb{P}\big[D_{n}(\tau_{n}, \tau_{n}+\sqrt{n}\delta] < - \sqrt{2\delta} n^{1/4}, A_{n,1}(\tau_{n}+\sqrt{n}\delta) - A_{n,2}(\tau_{n}+\sqrt{n}\delta) - D_{n}(\tau_{n}) > 0 \big] \\
		& \quad \approx \mathbb{P}\big[D_{n}(\tau_{n}, \tau_{n}+\sqrt{n}\delta] > \sqrt{2\delta} n^{1/4} \big] \mathbb{P}\big[A_{n,1}(\tau_{n}+\sqrt{n}\delta) - A_{n,2}(\tau_{n}+\sqrt{n}\delta) - D_{n}(\tau_{n}) \leq 0 \big] \\
		& \qquad + \mathbb{P}\big[D_{n}(\tau_{n}, \tau_{n}+\sqrt{n}\delta] < - \sqrt{2\delta} n^{1/4} \big] \mathbb{P}\big[ A_{n,1}(\tau_{n}+\sqrt{n}\delta) - A_{n,2}(\tau_{n}+\sqrt{n}\delta) - D_{n}(\tau_{n}) > 0 \big] \\
		& \quad \approx 0.1587.
	\end{align*}
	The load imbalance is at least $O(n^{1/4})$. In comparison with the $O(n^{1/4}\sqrt{\log n})$ bound by the RJSQ policy, the improvement by exploiting the complete history cannot exceed a sublogarithmic factor.
\end{example}

\section{Complete Resource Pooling}
\label{sec:Diffusion-Limit}

In this section, we prove the above sequence of systems (referred to as the \emph{distributed} systems hereafter) is asymptotically equivalent to a sequence of single-server systems. Consequently, the scaled queue length process converges to a reflected Brownian motion in a one-dimensional subspace. 

The single-server systems are constructed as follows: All of them are initially empty. Customers join the $ n $th distributed system and the $ n $th single-server system according to the same renewal process. The sole station in the latter system has a work-conserving server with service rate $ \mu $ and unlimited waiting space, serving customers on the FCFS basis. If the $ j $th customer in the distributed system is sent to station~$ k $, the traveling delay in the single-server system is also $ \gamma_{n,k}(j) $. Since the corresponding customers arrive at their stations simultaneously, we refer to the single-server system as the \emph{synchronized service pool} (SSP). A customer in the SSP is said to belong to \emph{class} $ k $ if the counterpart is sent to station~$ k $ in the distributed system. A customer's service requirement in the SSP is the same as that of the counterpart; the service time of the $ i $th class-$ k $ customer in the SSP is $ w_{k}(i)/\mu $, whereas the corresponding service time in the distributed system is $ w_{k}(i)/\mu_{k} $.

Let $ W_{n,k}(t) $ be the \emph{stationed workload} (i.e., the amount of unfinished service in a station) in station~$ k $ of the distributed system at time~$ t $, and $ W_{n,k}^{\star}(t) $ be that of class-$ k $ customers in the station of the SSP. The total stationed workloads are $ W_{n}(t) = \sum_{k = 1}^{s} W_{n,k}(t) $ and $ W_{n}^{\star}(t) = \sum_{k = 1}^{s} W_{n,k}^{\star}(t) $. Let $U_{n}(t)$ be the \emph{en route workload} (i.e., the total amount of service required by the en route customers) in the distributed system, and $U^{\star}_{n}(t)$ be that in the SSP. Clearly, $U_{n}^{\star}(t) = U_{n}(t)$. The sum of the total stationed workload and the en route workload is referred to as the \emph{system workload}. 

The difference between the system workloads is
\begin{equation}\label{eq:workload-difference}
	\Gamma^{\star}_{n}(t) = W_{n}(t) - W_{n}^{\star}(t),
\end{equation}
the scaled version of which is defined by $ \tilde{\Gamma}^{\star}_{n}(t) = \Gamma^{\star}_{n}(nt)/\sqrt{n}$. As the next proposition implies, the distributed system can never have less workload than the SSP, regardless of the routing policy.
\begin{proposition}\label{prop:stationed-workload-comparision}
	If both the distributed system and the SSP are initially empty, then $\Gamma^{\star}_{n}$ has continuous, piecewise linear sample paths, with $ \Gamma^{\star}_{n}(t) \geq 0 $ for all $ t \geq 0 $.
\end{proposition}

The RJSQ policy renders the workload in the distributed system asymptotically close to that in the SSP. Thus, the distributed system attains complete resource pooling.

\begin{theorem}\label{theorem:workload-difference}
	Assume the conditions of Theorem~\ref{theorem:load-balancing} hold. Then for $T > 0$, $ \sup_{0 \leq t \leq T} \tilde{\Gamma}^{\star}_{n}(t) \Rightarrow 0 $ as $ n \to \infty $. If conditions~\eqref{eq:fourth-moments} and \eqref{eq:log-log-n-bound} also hold, then with probability 1, 
	\[
	\sup_{0 \leq t \leq T} \tilde{\Gamma}^{\star}_{n}(t)< \max \Big\{ 2C \mu \chi_{n}, 2C' \mu \frac{\log \log n + \log (n \chi_{n}^{2})}{\sqrt{n} \chi_{n}} \Big\} \quad \mbox{for $n$ sufficiently large,}
	\]
	where $C$ and $C'$ are specified in Theorem~\ref{theorem:load-balancing-gap}. The order of this upper bound is minimized with $ \chi_{n} = C'' n^{-1/4}\sqrt{\log n} $ for $ C'' > 0 $, in which case $ \sup_{0 \leq t \leq T} \tilde{\Gamma}^{\star}_{n}(t) = O(n^{-1/4}\sqrt{\log n}) $ with probability~1.
\end{theorem}

\begin{remark}\label{remark:not-minimum}
	We allow a customer's service requirement to depend on the station in the distributed system. Since the corresponding customer has the same service requirement and traveling delay, the workload in the SSP depends on the routing policy in the distributed system. Although it is known that complete resource pooling asymptotically minimizes workload at all times \citep{Laws.1992,Harrison.Lopez.1999}, proving this result requires a lower bound independent of the routing policy. Since the lower bound by the SSP may change with $\chi_{n}$, it is also unclear whether the RJSQ policy can be optimized for workload minimization by taking $ \chi_{n} = C'' n^{-1/4}\sqrt{\log n} $.\footnote{The routing policy has no influence on the SSP in certain special cases. For instance, if all the traveling delays are equal and each customer's service requirement does not depend on the station, the total stationed workload in the distributed system will be asymptotically minimized with $ \chi_{n} = C'' n^{-1/4}\sqrt{\log n} $. Please refer to Remark~\ref{remark:asymptotic-optimality}.} A similar issue is discussed in \citet{Harrison.Lopez.1999} for parallel-server systems without traveling delays.
\end{remark}

Because the workloads are asymptotically close, the customer counts in the two systems should also be close. In heavy traffic, the scaled customer count process in a single-server queue should converge to a reflected Brownian motion (Chapter~6 in \citealp{Chen.Yao.2001}). Using this result, we prove a diffusion limit for the queue length process in the distributed system:
\begin{theorem}\label{theorem:CRP}
	Assume the conditions of Theorem~\ref{theorem:load-balancing} hold. Then, $ (\tilde{L}_{n,1}, \ldots, \tilde{L}_{n,s}) \Rightarrow ( \tilde{Q}/\mu , \ldots, \tilde{Q}/\mu  ) $ as $ n \to \infty $, where $ \tilde{Q} $ is a one-dimensional reflected Brownian motion starting from $ 0 $ with drift $ -\beta \mu $ and variance $ \mu c^{2}_{0} + \sum_{k = 1}^{s} \mu_{k} c^{2}_{k} $.
\end{theorem}

\section{Capacity Planning and Load Balancing}
\label{sec:Routing-Capacity-Planning}

In designing a service system with geographically separated stations, the capacity of each station should be carefully planned so that most customers can visit their nearest stations and get service without excessive waiting. Fix the total service capacity $\mu$. To minimize the mean traveling delay, we may determine $\mu_{1}, \ldots, \mu_{s}$ along with a routing plan by solving
\[
\begin{array}{l@{\quad}l@{}}
	\textnormal{minimize} & \displaystyle \sum_{k = 1}^{s} \sum_{m = 1}^{b} p_{m} r_{m,k} d_{m,k} \\
	\textnormal{subject to} & \textnormal{\eqref{eq:capacity} and \eqref{eq:rmk-inequality}--\eqref{eq:rmk-heavy-traffic}.}
\end{array}
\]
An optimal routing plan given by this linear program must satisfy
\begin{equation}\label{eq:nearest-station}
	r_{m,k} > 0 \quad\mbox{only if $d_{m,k} = \underline{d}_{m}$},
\end{equation}
where $\underline{d}_{m} = \min\{ d_{m,\ell} : \ell = 1, \ldots, s \}$. A distributed system is said to be \emph{geographically best-capacitated} (GBC), if the service capacities are given by \eqref{eq:rmk-heavy-traffic} for a routing plan that satisfies \eqref{eq:nearest-station}. According to such a routing plan, the RJSQ policy will send all but a small fraction of customers to their nearest stations. Consequently, the system's mean time to service will be asymptotically minimized. In this section, we strengthen this observation by proving the asymptotic optimality of the RJSQ policy for reducing workload in a GBC system. 

Consider a sequence of distributed systems. As in Section~\ref{sec:Diffusion-Limit}, we use a sequence of single-server systems as a frame of reference: they are identical to the SSPs, except the traveling delays of customers from origin $m$ are all equal to  $\sqrt{n}\underline{d}_{m}$ in the $n$th system, as if the sole station is the nearest to all. We refer to this system as the \emph{minimum-delay service pool} (MDSP). The arrival process at the station of the MDSP does \emph{not} depend on the routing policy in the distributed system.

Let $W^{\dagger}_{n}(t)$ and $U^{\dagger}_{n}(t)$ be the stationed and en route workloads in the MDSP at time $t$. The workload difference between the two systems is
\begin{equation}\label{eq:workload-difference-JNQ}
	\Gamma^{\dagger}_{n}(t) = \big(W_{n}(t) + U_{n}(t)\big) - \big( W_{n}^{\dagger}(t) + U^{\dagger}_{n}(t) \big),
\end{equation}
the scaled version of which is $ \tilde{\Gamma}^{\dagger}_{n}(t) = \Gamma^{\dagger}_{n}(nt)/\sqrt{n}$. As the following proposition implies, the distributed system cannot have less workload than the MDSP.
\begin{proposition}\label{prop:JNQ-workload-comparision}
	If both the distributed system and the MDSP are initially empty, then $\Gamma^{\dagger}_{n}$ has continuous, piecewise linear sample paths, with $ \Gamma^{\dagger}_{n}(t) \geq 0 $ for all $ t \geq 0 $.
\end{proposition}

If the distributed system is GBC, the RJSQ policy renders its workload asymptotically close to that in the MDSP.

\begin{theorem} \label{theorem:workload-optimality}
	Assume the conditions of Theorem~\ref{theorem:load-balancing} hold for a sequence of GBC systems where the customers' origins are known. Then for $T > 0$, $ \sup_{0 \leq t \leq T} \tilde{\Gamma}^{\dagger}_{n}(t) \Rightarrow 0 $ as $ n \to \infty $. If conditions~\eqref{eq:fourth-moments} and \eqref{eq:log-log-n-bound} also hold, then there exist two positive numbers $ C^{\dagger} $ and $ C^{\ddagger} $ such that with probability 1,
	\[
	\sup_{0 \leq t \leq T} \tilde{\Gamma}^{\dagger}_{n}(t) < \max \Big\{ C^{\dagger} \chi_{n}, C^{\ddagger} \frac{\log \log n + \log (n \chi_{n}^{2})}{\sqrt{n} \chi_{n}} \Big\} \quad \mbox{for $n$ sufficiently large.}
	\]
	The order of this upper bound is minimized with $ \chi_{n} = C'' n^{-1/4}\sqrt{\log n} $ for $ C'' > 0 $, in which case $ \sup_{0 \leq t \leq T} \tilde{\Gamma}^{\dagger}_{n}(t) = O(n^{-1/4}\sqrt{\log n})$ with probability 1.
\end{theorem}

\begin{remark}\label{remark:asymptotic-optimality}
	Since a customer's service requirement is station-dependent in the GBC system, the workload in the MDSP depends on the routing policy. Because the lower bound may change with $\chi_{n}$, it is unclear whether the RJSQ policy can be optimized with $ \chi_{n} = C'' n^{-1/4}\sqrt{\log n} $. However, if service requirements do not depend on stations, the routing policy will have no influence on the MDSP, where the system workload will be a \emph{universal} lower bound. In this case, the RJSQ policy will be optimized with $ \chi_{n} = C'' n^{-1/4}\sqrt{\log n} $.
\end{remark}

\begin{corollary} \label{corollary:workload-optimality}
	Assume the conditions of Theorem~\ref{theorem:load-balancing} hold for a sequence of GBC systems where the customers' origins are known. If $\{ w_{k}(i) : i\in\mathbb{N} ; k = 1, \ldots, s \}$ is a sequence of i.i.d.\ random variables, the workload in the GBC system can be asymptotically minimized by the RJSQ policy, in the sense that for $T > 0$, $ \sup_{0 \leq t \leq T} \tilde{\Gamma}^{\dagger}_{n}(t) \Rightarrow 0 $ as $ n \to \infty $. If conditions~\eqref{eq:fourth-moments} and \eqref{eq:log-log-n-bound} also hold, then by taking $ \chi_{n} = C'' n^{-1/4}\sqrt{\log n} $ for $ C'' > 0 $, the RJSQ policy attains $ \sup_{0 \leq t \leq T} \tilde{\Gamma}^{\dagger}_{n}(t) = O\big(n^{-1/4}\sqrt{\log n}\big) $.
\end{corollary}

\begin{remark}\label{remark:geographic}
	Theorem~\ref{theorem:workload-optimality} and Corollary~\ref{corollary:workload-optimality} imply the following approach to capacity planning and load balancing. First, we divide the origins into $s$ disjoint sets $\mathcal{N}_{1}, \ldots, \mathcal{N}_{s}$ so that station~$k$ is the nearest to the origins in $\mathcal{N}_{k}$. Given the total capacity $\mu$, the capacities are set to $ \mu_{k} = \sum_{m \in \mathcal{N}_{k}} p_{m} \mu $ for $k = 1, \ldots, s$, where $\mu_{k} = 0$ if $\mathcal{N}_{k}$ is empty; a station with zero capacity is removed because of its inconvenient location. Without loss of generality, let us assume $\mu_{k} > 0$ for $k = 1, \ldots, s$. The resultant routing plan is given by $r_{m,k} = 1$ for $m \in \mathcal{N}_{k}$ and $r_{m,k} = 0$ otherwise, suggesting customers from each origin are sent to the nearest station; the system is thus GBC. Next, for $k = 1, \ldots, s$, we identify a set of origins outside $\mathcal{N}_{k}$ but still close to station~$k$---for instance, we may specify $\bar{\tau}_{k} > 0$, which is the \emph{tolerance for delays} to station~$k$, and define $ \mathcal{N}'_{k} = \{ m \not\in \mathcal{N}_{k} : \underline{d}_{m} \leq d_{m,k} \leq \underline{d}_{m} + \bar{\tau}_{k} \} $, where $\bar{\tau}_{k}$ should be large enough to make $\mathcal{N}'_{k}$ nonempty. For $m \in \mathcal{N}'_{k}$, the traveling delay from origin~$m$ to station~$k$ can exceed that to the nearest station by at most $\bar{\tau}_{k}$. Let $p'_{k} = \sum_{m\in\mathcal{N}'_{k}} p_{m}$, which is the probability of a customer coming from $\mathcal{N}'_{k}$, and assume the balancing fraction satisfies $0 < \chi \leq \min\{p'_{1}, \ldots, p'_{s} \}$. Finding the shortest queue length in station~$\ell$, the dispatcher may determine a customer's destination as follows: if coming from $\mathcal{N}'_{\ell}$, the customer will be sent to station~$\ell$ with probability $\chi/p'_{\ell}$ and to the nearest station with probability $1 - \chi/p'_{\ell}$; otherwise, the customer will be sent to the nearest station. In this way, station~$\ell$ will be the destination either because the customer is from $\mathcal{N}_{\ell}$, or because the customer is from $\mathcal{N}'_{\ell}$ and is selected with probability $\chi/p'_{\ell}$. The probability of sending the customer to station~$\ell$, which has the shortest queue, is $\mu_{\ell}/\mu + \chi$. Following this approach, the dispatcher will send all but a small fraction of customers to the nearest stations; the additional fraction of customers sent to the shortest queue will be from origins close to the destinations. Thus, the RJSQ policy can achieve asymptotic optimality by adding minimally to the traveling delays of a small fraction of customers. We will illustrate this approach by a numerical example in Section~\ref{sec:Numerical}.
\end{remark}

\section{Numerical Experiments and Heuristics}
\label{sec:Numerical}

Numerical examples are studied in this section. Relying on both theoretical results and numerical evidence, we propose simple yet near-optimal heuristic formulas for the balancing fraction.
\begin{figure}
	\caption{Mean Total Customer Count under RJSQ}
	\centering
	\begin{subfigure}[b]{0.495\textwidth}
		\includegraphics[width=3.2in]{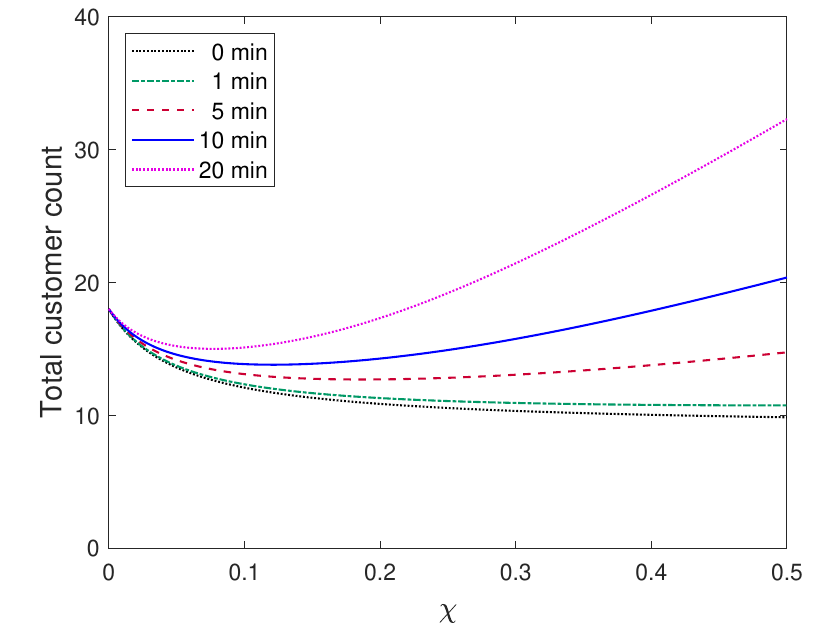}
		\caption{{\footnotesize $\rho = 0.90$}}
		\label{fig:rho09}
	\end{subfigure}
	\begin{subfigure}[b]{0.495\textwidth}
		\includegraphics[width=3.2in]{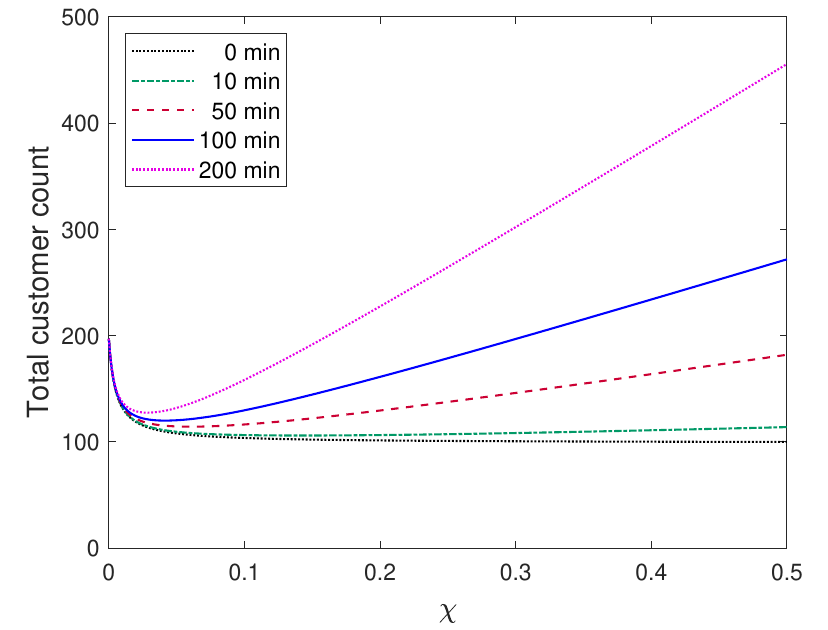}
		\caption{{\footnotesize $\rho = 0.99$}}
		\label{fig:rho099}
	\end{subfigure}\label{fig:QL-vs-rho}
\end{figure}

Let us first revisit the example in Section~\ref{sec:Introduction}: There are two stations with identical capacities. Customers join the system according to a Poisson process with equal traveling delays to both stations. The service times are exponentially distributed with mean 1 minute. The \emph{mean total customer count} (MTCC) is used as the performance measure. With the balancing fraction $\chi$ ranging from 0 to 0.5, simulation results by averaging $2\times 10^{4}$ runs are plotted in Figures~\ref{fig:QL-vs-rho} and~\ref{fig:QL-vs-rho-long}. In each run, the system continuously operates $2 \times 10^{5}$ minutes, including a burn-in period of $6 \times 10^{4}$ minutes. The traffic intensity is set to $\rho = 0.90$ and $0.99$, corresponding to $n = 100$ and $10^{4}$ in \eqref{eq:heavy-traffic} with $\beta = 1$ (i.e., $n = 1/(1-\rho)^{2}$). When $\chi = 0$, the stations are $\mbox{M}/\mbox{M}/1$ queues where the MTCC is $2\rho/(1-\rho)$. The traveling delay is $\sqrt{n}d$ minutes with $d$ taking values 0, 0.1, 0.5, 1, 2, and 10, respectively---that is, 0, 1, 5, 10, 20, 100 minutes for $\rho = 0.90$, and 0, 10, 50, 100, 200, 1000 minutes for $\rho = 0.99$.

We first test the traveling delays that are either on a lower order of magnitude than or on the same order as $\sqrt{n} = 1/(1-\rho)$. In Figure~\ref{fig:QL-vs-rho}, all the curves descend rapidly as $\chi$ begins to grow from zero; this implies sending a small fraction of customers to the shorter queue suffices to balance the workload, regardless of the traveling delay. When the traveling delay is zero, the MTCC continues to decrease as $\chi$ increases further, but at a much slower rate, until all the customers are sent to the shorter queue ($\chi = 0.5$). When the delay is on a lower order than $1/(1-\rho)$ (i.e., 1 minute for $\rho = 0.90$ and 10 minutes for $\rho = 0.99$), the curves exhibit a similar pattern where the JSQ policy is either optimal or near-optimal. When the traveling delay is on the same order as $1/(1-\rho)$ (i.e., 5, 10, or 20 minutes for $\rho = 0.90$ and 50, 100, or 200 minutes for $\rho = 0.99$), the curves each form a relatively ``flat'' bottom, which corresponds to the range specified by \eqref{eq:epsilon-1} and \eqref{eq:epsilon-2}, before ascending linearly to a large value. The JSQ policy is plagued when the traveling delay is considerable.

We also test traveling delays that are on a higher order than $1/(1-\rho)$---that is, 100 minutes for $\rho = 0.90$ and 1000 minutes for $\rho = 0.99$. Both curves in Figure~\ref{fig:QL-vs-rho-long} slightly descend as $\chi$ grows from zero, then ascending to an enormous value. The JSQ policy ($\chi = 0.5$) is far worse than \emph{routing at random} ($\chi = 0$). Within a very long traveling delay, load imbalance caused by random variations is often too large for the RJSQ policy to compensate in time. Even with the optimal balancing fraction, the RJSQ policy cannot offer much improvement over routing at random.
\begin{figure}
	\caption{Mean Total Customer Count with Long Traveling Delays}
	\centering
	\begin{subfigure}[b]{0.495\textwidth}
		\includegraphics[width=3.2in]{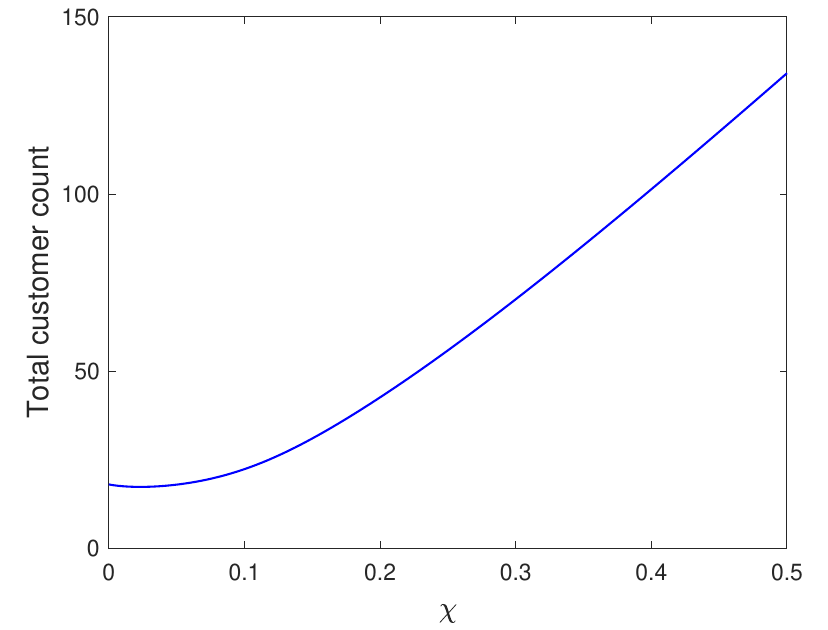}
		\caption{{\footnotesize $\rho = 0.90$, delay 100 minutes}}\label{fig:rho09-long}
	\end{subfigure}
	\begin{subfigure}[b]{0.495\textwidth}
		\includegraphics[width=3.2in]{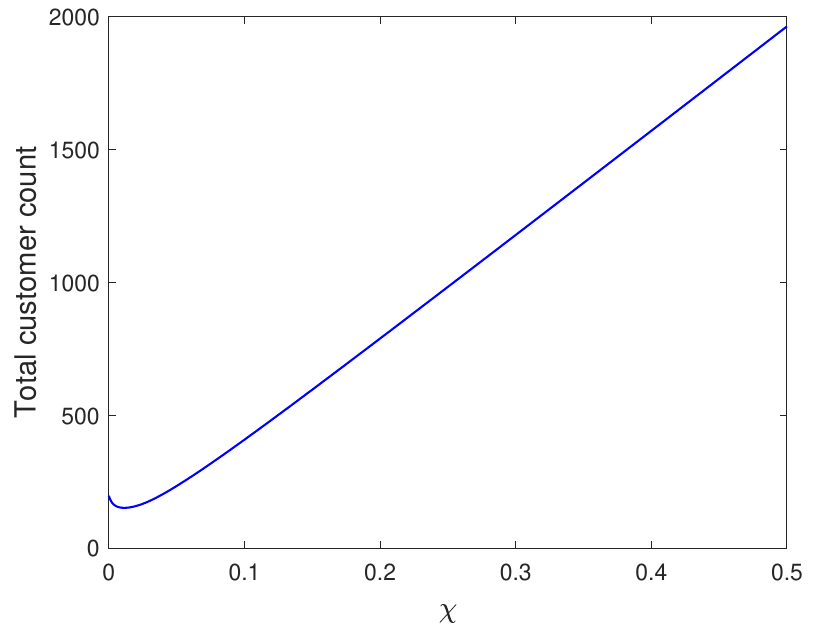}
		\caption{{\footnotesize $\rho = 0.99$, delay 1000 minutes}}\label{fig:rho099-long}
	\end{subfigure}\label{fig:QL-vs-rho-long}
\end{figure}

The optimal values of $\chi$ from the simulation results, denoted by $\chi_{\operatorname{opt}}$, are reported in Table~\ref{table:two-stations}, where the MTCC with balancing fraction $\chi$ is denoted by $\operatorname{CC}(\chi)$. As discussed in Remark~\ref{remark:trade-off}, load imbalance between the stations is attributed to both queue length oscillations brought by the RJSQ policy, and random variations caused by the arrival and service completion processes. By Theorem~\ref{theorem:load-balancing-gap}, the magnitude of oscillations is $O(\sqrt{n}\chi_{n})$ and that of random variations is $O((\log \log n + \log (n \chi_{n}^{2}))/\chi_{n})$. The latter involves logarithmic terms, because the law of the iterated logarithm is used to obtain a uniform bound for the load imbalance over the time interval $[0, nT]$. At an arbitrary moment, the magnitude of oscillations is still $O(\sqrt{n}\chi_{n})$, but that of random variations would be $O(1/\chi_{n})$. Hence, the mean load imbalance, as opposed to the supremum over a time interval in \eqref{eq:load-imbalance}, could be minimized if the balancing fraction is set to
\begin{equation}\label{eq:root-excess}
	\chi^{*} = C^{*} n^{-1/4} = C^{*}\sqrt{1-\rho},
\end{equation}
where $C^{*}$ is a positive constant. The MTCC should also be minimized if the mean load imbalance is minimized. Then, $\chi^{*}$ should be close to the optimal balancing fraction. To check this heuristic, the values of $\chi_{\operatorname{opt}}/\sqrt{1-\rho}$, as estimates of $C^{*}$, are plotted for different primitive traveling delays in Figure~\ref{fig:root-of-excess}: except for $d = 0.1$ (i.e., 1 minute for $\rho = 0.90$ and 10 minutes for $\rho = 0.99$), the optimal balancing fractions, scaled by $1/\sqrt{1-\rho}$, agree well when they correspond to the same value of $d$.
\begin{table}[b]
	\centering
	\caption{Mean Total Customer Counts in Two-Station System}
	\label{table:two-stations}
	\footnotesize
	\begin{tabular}{cccccccc}
		\toprule
		$\rho$ & Delay (min) & $\chi_{\operatorname{opt}}$ & $\operatorname{CC}(\chi_{\operatorname{opt}})$ & $\operatorname{CC}(\chi^{*})$ & $\operatorname{CC}(\chi^{**})$ & $\operatorname{CC}(0)$ & $\operatorname{CC}(0.5)$ \\ \midrule
		\multirow{6}{*}{0.90} & \phantom{000}0 & 0.500 & \phantom{00}9.86 & \phantom{0}11.63 & \phantom{00}9.86 & \phantom{0}18.01 & \phantom{000}9.86\\
		& \phantom{000}1 & 0.500 & \phantom{0}10.75 & \phantom{0}11.93 & \phantom{0}10.79 & \phantom{0}18.01 & \phantom{00}10.75\\
		& \phantom{000}5 & 0.184 & \phantom{0}12.70 & \phantom{0}12.87 & \phantom{0}12.71 & \phantom{0}18.01 & \phantom{00}14.75\\
		& \phantom{00}10 & 0.121 & \phantom{0}13.82 & \phantom{0}13.82 & \phantom{0}13.82 & \phantom{0}18.01 & \phantom{00}20.38 \\
		& \phantom{00}20 & 0.078 & \phantom{0}15.01 & \phantom{0}15.47 & \phantom{0}15.04 & \phantom{0}18.01 & \phantom{00}32.31\\
		& \phantom{0}100 & 0.022 & \phantom{0}17.32 & \phantom{0}26.52 & \phantom{0}17.57 & \phantom{0}18.01 & \phantom{0}134.12\\
		\midrule
		\multirow{6}{*}{0.99} & \phantom{000}0 & 0.500 & 100.02 & 110.05 & 100.02 & 197.12 & \phantom{0}100.02\\
		& \phantom{00}10 & 0.138 & 106.07 & 111.31 & 106.12 & 197.12 & \phantom{0}113.99\\
		& \phantom{00}50 & 0.059 & 114.38 & 115.47 & 114.40 & 197.12 & \phantom{0}182.10\\
		& \phantom{0}100 & 0.042 & 120.11 & 120.13 & 120.13 & 197.12 & \phantom{0}271.89\\
		& \phantom{0}200 & 0.029 & 127.58 & 129.07 & 127.59 & 197.12 & \phantom{0}455.74\\
		& 1000 & 0.012 & 152.06 & 204.06 & 152.34 & 197.31 & 1964.64\\
		\bottomrule
	\end{tabular}
\end{table}

In Figure~\ref{fig:root-of-excess}, $\chi_{\operatorname{opt}}/\sqrt{1-\rho}$ decreases as $d$ gets larger. All the traveling delays are $\sqrt{n}d$ minutes in this example. If we use the squared delay $nd^{2}$, instead of $n$, as the scaling factor to define the scaled queue length processes, the main results in Sections~\ref{sec:Load-Balancing}--\ref{sec:Routing-Capacity-Planning} still hold, except the constant coefficients (e.g., $C$ and $C'$ in Theorem~\ref{theorem:load-balancing-gap} and $C''$ in Corollary~\ref{corollary:optimal-gap}) need to be modified. By \eqref{eq:root-excess}, the MTCC could be minimized if the balancing fraction is set to
\begin{equation}\label{eq:root-of-excess-d}
	\chi^{**} = C^{**} (nd^{2})^{-1/4} = C^{**}\sqrt{\frac{1-\rho}{d}},
\end{equation}
where $C^{**}$ is a positive constant not depending on $d$. The values of $ \sqrt{d/(1-\rho)}\chi_{\operatorname{opt}} $, as estimates of $C^{**}$, are compared in Figure~\ref{fig:root-of-excess-d}: for $\rho = 0.99$, the values are all around 0.4; for $\rho = 0.90$, the value decreases with $d$, but staying around 0.4 for $d = 0.5$, 1, and 2---that is, for the traveling delays on the same order as $1/(1-\rho)$. Taking $C^{**} = 0.4$ in \eqref{eq:root-of-excess-d} and setting $\chi^{**} = 0.5$ for $d = 0$ (i.e., the JSQ policy), we list the MTCCs, denoted by $\operatorname{CC}(\chi^{**})$, in Table~\ref{table:two-stations}; in all the instances, these MTCCs are almost identical to the minimum values denoted by $\operatorname{CC}(\chi_{\operatorname{opt}})$.
\begin{figure}
	\caption{Optimal Values of Balancing Fraction}
	\centering
	\begin{subfigure}[b]{0.495\textwidth}
		\includegraphics[width=3.2in]{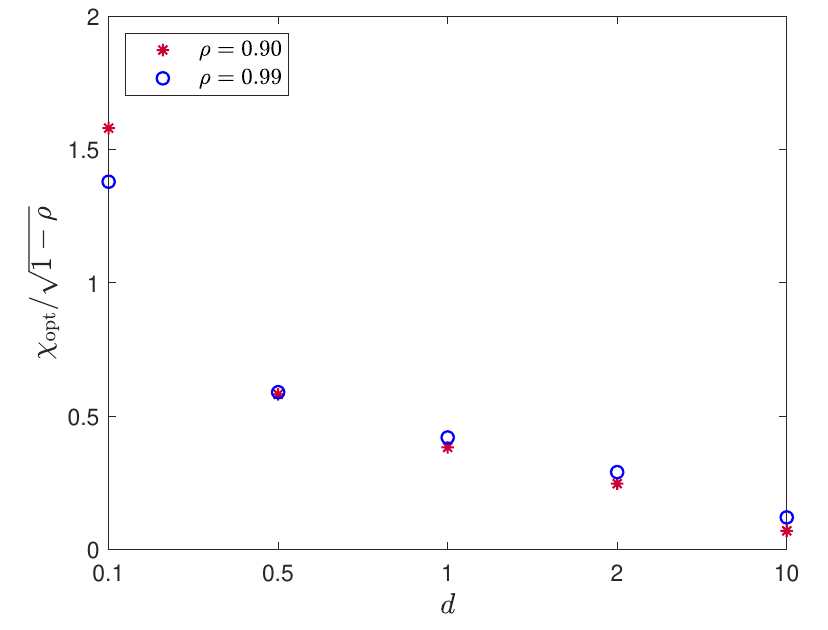}
		\caption{{\footnotesize Comparison of $ \chi_{\operatorname{opt}}/\sqrt{1-\rho} $}}\label{fig:root-of-excess}
	\end{subfigure}
	\begin{subfigure}[b]{0.495\textwidth}
		\includegraphics[width=3.2in]{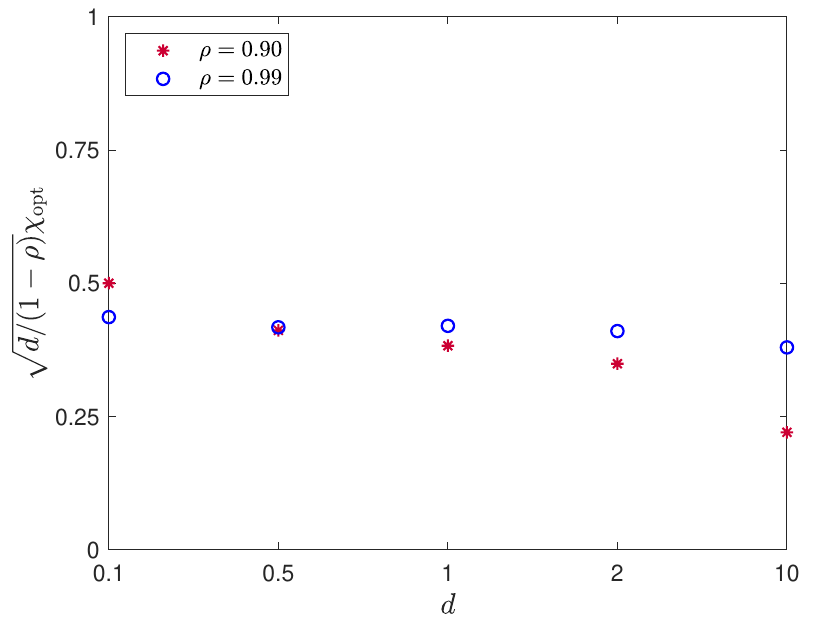}
		\caption{{\footnotesize Comparison of $ \sqrt{d/(1-\rho)}\chi_{\operatorname{opt}} $}}\label{fig:root-of-excess-d}
	\end{subfigure}
\end{figure}

We may use these heuristics to determine the balancing fraction under more general delay assumptions. Consider the primitive delay distribution given by \eqref{eq:delay-distribution}. Because the weighted queue lengths are approximately equal, the probability of an arbitrary station having the shortest queue should be around $1/s$. If customers can be sent to the shortest queue regardless of the origins (e.g., when the dispatcher is unaware of their origins), the mean primitive delay of additional customers sent to the shortest queue can be estimated by 
\[  
\bar{\gamma} = \frac{1}{s}\sum_{m = 1}^{b} \sum_{k = 1}^{s} p_{m} d_{m,k},
\]
which could replace $d$ in \eqref{eq:root-of-excess-d} to determine the balancing fraction. If the dispatcher is aware of the customers' origins and sends additional customers to the shortest queue only from specific origins (as discussed in Remark~\ref{remark:geographic}), the above mean primitive delay can be estimated by
\[  
\bar{\gamma} = \frac{1}{s}\sum_{k = 1}^{s}  \frac{\sum_{m \in \mathcal{N}'_{k}}p_{m} d_{m,k}}{\sum_{m \in \mathcal{N}'_{k}}p_{m}},
\]
where $\mathcal{N}'_{k}$ is the set of origins from which additional customers can be sent to the shortest queue in station~$k$. With the estimate $\bar{\gamma}$, we may modify \eqref{eq:root-of-excess-d} into
\begin{equation}\label{eq:reciprocal-root}
	\chi^{**} = C^{**}\hat{\gamma}^{-1/2},
\end{equation}
where $\hat{\gamma} = \bar{\gamma}/(1-\rho)$ is the estimate of the mean traveling delay of additional customers sent to the shortest queue. We refer to the heuristic formula \eqref{eq:reciprocal-root} as the \emph{reciprocal-root-of-delay rule}.

It would be difficult to obtain $\hat{\gamma}$ if the traveling delay distribution is unknown. In this case, we could use \eqref{eq:root-excess} instead, as long as traveling delays are on the same order as $1/(1 - \rho)$. Since $1 - \rho$ is the fraction of excess service capacity, we refer to the heuristic formula \eqref{eq:root-excess} as the \emph{root-of-excess rule}. By Theorem~\ref{theorem:load-balancing}, the magnitude of load imbalance is $o(\sqrt{n})$ when the balancing fraction is within the range specified by \eqref{eq:epsilon-1} and \eqref{eq:epsilon-2}; this range corresponds to the curves' flat bottoms in Figure~\ref{fig:QL-vs-rho}, at which the MTCC is insensitive to small changes in $\chi$. Taking $C^{*} = 0.4$ in \eqref{eq:root-excess}, we list the MTCCs in Table~\ref{table:two-stations}. Denoted by $\operatorname{CC}(\chi^{*})$, these values are close to the minimum when the traveling delays are on the same order as $1/(1-\rho)$ (i.e., 5, 10, 20 minutes for $\rho = 0.90$ and 50, 100, 200 minutes for $\rho = 0.99$). The table also includes the MTCCs under routing at random ($\chi = 0$) and the JSQ policy ($\chi = 0.5$): with traveling delays on a lower or higher order than $1/(1-\rho)$, the JSQ policy or routing at random could be an acceptable option, respectively.

Are these rules of thumb applicable to more than two stations? We test the system with three to five stations and keep other settings unchanged. The MTCC curves for $\rho = 0.99$ and traveling delays of 100 minutes are plotted in Figure~\ref{fig:QL-vs-ST}: all the curves have a steep drop followed by a flat bottom, before ascending linearly to a large value; the optimal balancing fractions are relatively stable, taking values at $0.042$, $0.046$, $0.046$, and $0.043$ for two to five stations, respectively. Table~\ref{table:five-stations} reports MTCCs with specific balancing fractions for the five-station system, where $\chi^{*}$ and $\chi^{**}$ are given by \eqref{eq:root-excess} and \eqref{eq:reciprocal-root} with $C^{*} = C^{**} = 0.4$ and the JSQ policy corresponds to $\chi = 0.8$. Our heuristics are confirmed again: the values of $\chi_{\operatorname{opt}}$ in Table~\ref{table:five-stations} are generally close to those in Table~\ref{table:two-stations}; the reciprocal-root-of-delay rule produces near-optimal performance; the root-of-excess rule produces satisfactory performance when the traveling delays are on the same order as $1/(1-\rho)$.
\begin{figure}[t]
		\caption{Mean Total Customer Count with More Stations}
		\centering
		\includegraphics[width=3.2in]{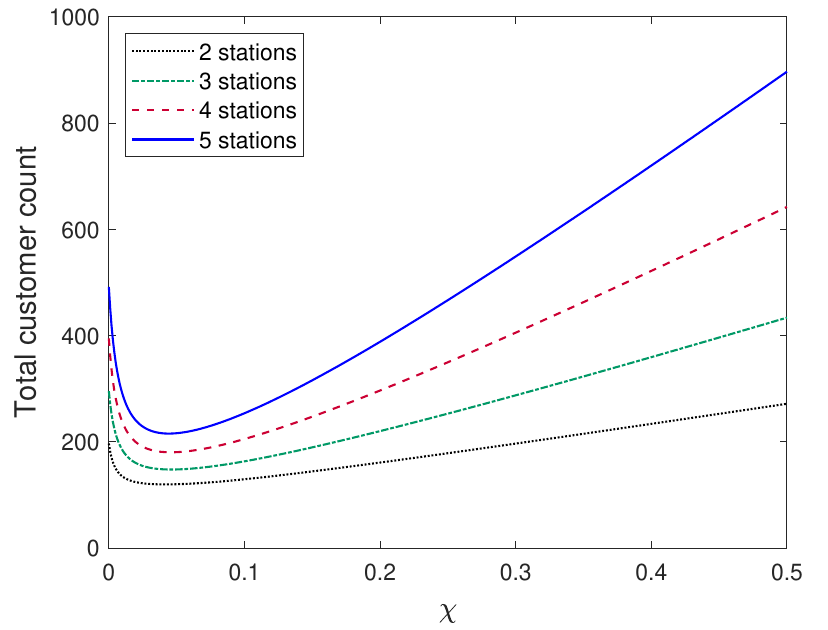}
		\label{fig:QL-vs-ST}
\end{figure}

\begin{table}[b]
	\centering
	\caption{Mean Total Customer Counts in Five-Station System}
	\label{table:five-stations}
	\footnotesize
	\begin{tabular}{cccccccc}
		\toprule
		$\rho$ & Delay (min) & $\chi_{\operatorname{opt}}$ & $\operatorname{CC}(\chi_{\operatorname{opt}})$ & $\operatorname{CC}(\chi^{*})$ & $\operatorname{CC}(\chi^{**})$ & $\operatorname{CC}(0)$ & $\operatorname{CC}(0.8)$ \\ \midrule
		\multirow{6}{*}{0.90} & \phantom{000}0 & 0.800 & \phantom{0}12.59 & \phantom{0}23.42 & \phantom{0}12.59 & \phantom{0}45.01 & \phantom{000}12.59\\
		& \phantom{000}1 & 0.462 & \phantom{0}19.73 & \phantom{0}25.04 & \phantom{0}19.80 & \phantom{0}45.01 & \phantom{000}21.18\\
		& \phantom{000}5 & 0.184 & \phantom{0}28.42 & \phantom{0}29.12 & \phantom{0}28.42 & \phantom{0}45.01 & \phantom{000}54.60\\
		& \phantom{00}10 & 0.122 & \phantom{0}32.77 & \phantom{0}32.78 & \phantom{0}32.78 & \phantom{0}45.01 & \phantom{000}98.10 \\
		& \phantom{00}20 & 0.078 & \phantom{0}36.93 & \phantom{0}38.45 & \phantom{0}36.99 & \phantom{0}45.01 & \phantom{00}186.52\\
		& \phantom{0}100 & 0.024 & \phantom{0}43.46 & \phantom{0}65.69 & \phantom{0}44.03 & \phantom{0}45.01 & \phantom{00}902.46\\
		\midrule
		\multirow{6}{*}{0.99} & \phantom{000}0 & 0.800 & 103.16 & 165.27 & 103.16 & 492.05 & \phantom{00}103.16\\
		& \phantom{00}10 & 0.149 & 140.50 & 172.28 & 141.13 & 492.05 & \phantom{00}222.40\\
		& \phantom{00}50 & 0.064 & 185.97 & 193.55 & 186.49 & 492.05 & \phantom{00}758.95\\
		& \phantom{0}100 & 0.043 & 215.80 & 216.22 & 216.22 & 492.04 & \phantom{0}1445.10\\
		& \phantom{0}200 & 0.030 & 252.46 & 257.29 & 252.87 & 492.03 & \phantom{0}2821.06\\
		& 1000 & 0.012 & 357.40 & 548.94 & 358.15 & 491.95 & 13467.16\\
		\bottomrule
	\end{tabular}
\end{table}

In the next example, we consider joint capacity planning and load balancing for three stations in a square district, each side of which is 20 km. As shown in Figure~\ref{fig:region}, the coordinate of the southwest corner is $(0,0)$ and the stations are at $(5, 5)$, $(5, 15)$, and $(15, 5)$. Following a Poisson process with rate $\lambda = 3.96$ per minute, customers appear at locations uniformly distributed in the district. Once the destination is determined, a customer will walk in a straight line to the station, at a speed of 0.1 km per minute; the traveling delay is thus a \emph{continuous} random variable. By the discussion in Remark~\ref{remark:general-distribution}, our heuristics should also work with a continuous delay distribution. The total service capacity is $\mu = 4$, which yields $\rho = 0.99$.
\begin{figure}[t]
	\caption{Geographically Separated Stations}
	\centering
	\includegraphics[width=2.3in]{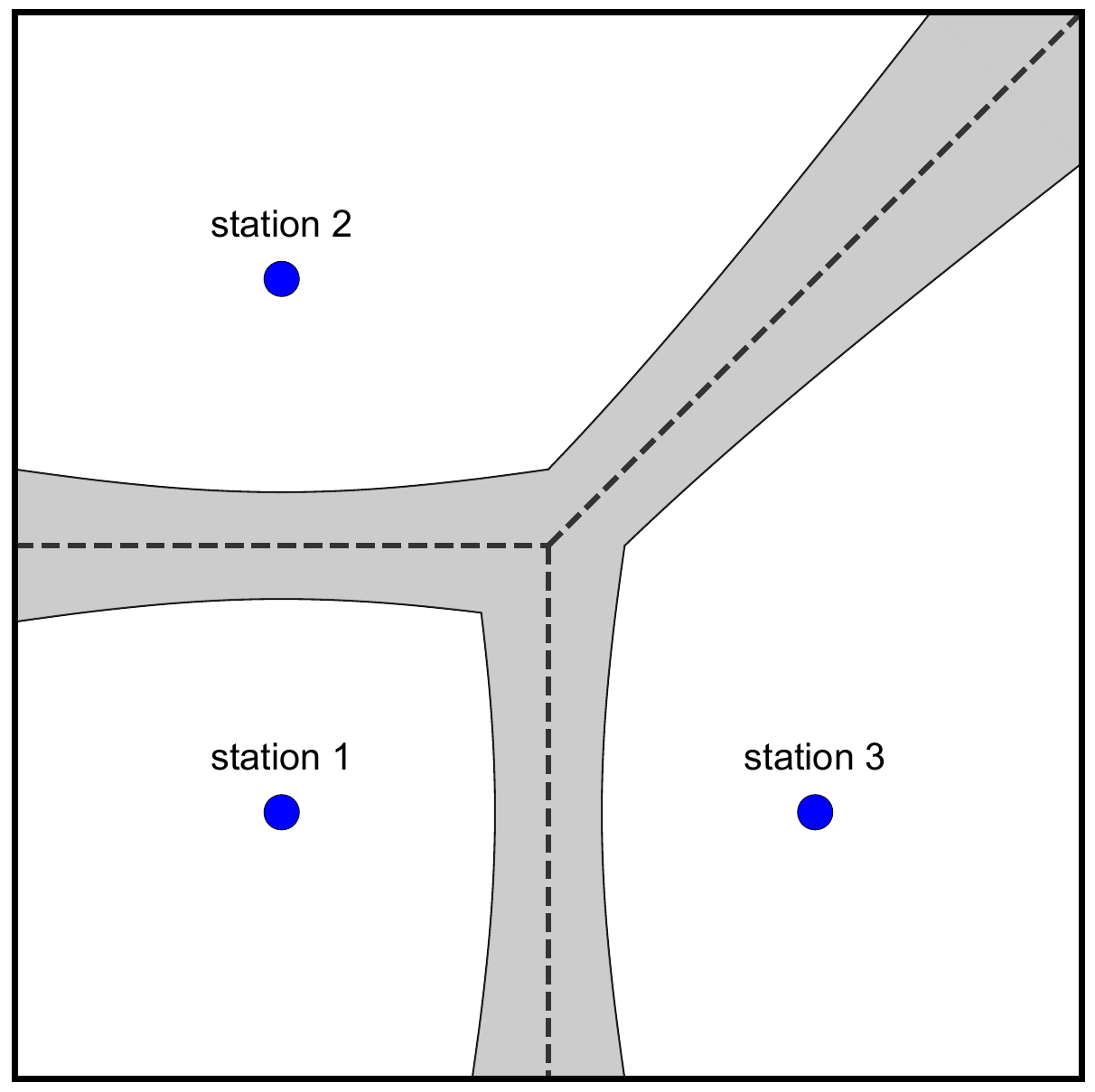}
	\label{fig:region}
\end{figure}

Let us first determine the capacity of each station. By the discussion in Remark~\ref{remark:geographic}, we may divide the district into three zones such that for $k = 1, 2, 3$, station~$k$ is the nearest to customers appearing in zone~$k$; these zones are separated by the dashed lines in Figure~\ref{fig:region}. Since the customers' origins are uniformly distributed, the service capacity of station~$k$ should be proportional to the area of zone~$k$. By setting $\mu_{1} = 1$ and $\mu_{2} = \mu_{3} = 1.5$, we make the system GBC; the resultant routing plan suggests sending customers from zone~$k$ to station~$k$, which is the nearest station. 

We next specify $\bar{\tau}$, the tolerance for delays, as discussed in Remark~\ref{remark:geographic}. We assume $\bar{\tau}$ is common to all the stations. Upon appearance from zone~$k$, a customer is sent to station~$\ell$ in an adjacent zone, if and only if station~$\ell$ has the shortest queue and the traveling delay to station~$\ell$ exceeds that to station~$k$ by at most $\bar{\tau}$; otherwise, the customer is sent to station~$k$. With a small $\bar{\tau}$, a customer can be sent to an adjacent zone only if the origin is close to the border---that is, within the shaded \emph{border region} in Figure~\ref{fig:region} (which is plotted for $\bar{\tau} = 20$ minutes).

How to choose $\bar{\tau}$? It can also be determined by the reciprocal-root-of-delay rule: Because the additional customers sent to the shortest queue are from the border region, we may estimate their mean distance to the destinations using the mean distance from the border lines by
\[
\begin{aligned}
	\frac{1}{20+10\sqrt{2}}\bigg(2 \int_{0}^{10} \sqrt{(x-5)^{2}+ 25}\,\mathrm{d}x + \sqrt{2} \int_{10}^{20} \sqrt{(x-5)^{2} + (x-15)^{2}}\,\mathrm{d}x \bigg) = 7.694 \ \mbox{km.}
\end{aligned}
\]
Their mean traveling delay to the shortest queue is around $\hat{\gamma} = 76.94 $ minutes. Taking $C^{**} = 0.4$ in \eqref{eq:reciprocal-root}, we obtain $\chi^{**} = 0.0456$, which is the additional fraction of customers sent to the shortest queue. A customer from the border region is sent to the second nearest station if it has the shortest queue; this probability should be around 1/3, since the weighted queue lengths are approximately equal. Therefore, the probability of a customer appearing from the border region should be around $3 \chi^{**}$, the area of the border region should be around $400 \times 3\chi^{**} = 54.72 $ $\mbox{km}^2$, and the mean ``width'' of the border region should be around $54.72/(20+10\sqrt{2}) = 1.603$ km. Hence, we may take $\bar{\tau} = 16 $ minutes. Although the above calculation is quite crude, the estimate may still produce near-optimal performance, since the RJSQ policy is insensitive to $\chi$ around the optimal value.

We assume the service requirements follow either an exponential distribution with mean 1 or a lognormal distribution with mean 1 and variance 3. Table~\ref{table:geographic} shows the system's performance with different values of $\bar{\tau}$. The results are produced by averaging $5 \times 10^{3}$ runs; in each run, the system continuously operates $2 \times 10^{5}$ minutes, including a burn-in period of $6 \times 10^{4}$ minutes. Since each $\bar{\tau}$ specifies a particular border region, with the same $\bar{\tau}$, both the balancing fractions (i.e., the fractions of customers sent to the second nearest stations) and the mean traveling delays should be identical under the two service distributions. All the customers are sent to their nearest stations when $\bar{\tau} = 0$, and to the shortest queue when $\bar{\tau} = \infty$. Although the former policy has the least mean traveling delay, it suffers from long waiting times in the absence of load balancing; the JSQ policy is even worse because of queue length oscillations. A relatively small $\bar{\tau}$ (as compared with the mean traveling delay and the mean waiting time) suffices to induce a satisfactory balancing fraction; this fraction of customers experience extra traveling delays of at most $\bar{\tau}$. When $\bar{\tau}$ is increased from 0 to 20 minutes, the mean traveling delay grows by less than 1.5\%, whereas the mean waiting time is reduced by 52.5\% and 56.9\% under the exponential and lognormal service distributions, respectively. Since the mean time to service is greatly shortened, the customers sent to the second nearest stations also benefit from the overall improvement, as everyone else does.
\begin{table}[t]
	\centering
	\caption{Mean Traveling Delay, Waiting, and Time to Service in Geographically Separated Stations (in Minutes)}
	\label{table:geographic}
	\footnotesize
	\begin{tabular}{c|cccc|cccc}
		\toprule
		& \multicolumn{4}{c|}{Exponential Service} & \multicolumn{4}{c}{Lognormal Service} \\ 
		$\bar{\tau}$ & $\chi$ & Delay & Waiting & Time to Service & $\chi$ & Delay & Waiting & Time to Service \\ \midrule
		\phantom{0}0 & 0\phantom{.0000} & 50.73 & \phantom{0}73.97 & 124.70 & 0\phantom{.0000} & 50.73 & 145.85 & 196.58 \\
		\phantom{0}5 & 0.0182 & 50.77 & \phantom{0}39.56 & \phantom{0}90.33 & 0.0182 & 50.77 & \phantom{0}76.23 & 127.00 \\
		10 & 0.0361 & 50.91 & \phantom{0}35.49 & \phantom{0}86.40 & 0.0362 & 50.91 & \phantom{0}66.81 & 117.72 \\
		15 & 0.0541 & 51.13 & \phantom{0}34.72 & \phantom{0}85.85 & 0.0542 & 51.13 & \phantom{0}63.76 & 114.89 \\
		\emph{16} & \emph{0.0577} & \emph{51.19} & \emph{\phantom{0}34.72} & \emph{\phantom{0}85.90} & \emph{0.0578} & \emph{51.19} & \emph{\phantom{0}63.49} & \emph{114.68} \\
		20 & 0.0722 & 51.45 & \phantom{0}35.11 & \phantom{0}86.56 & 0.0724 & 51.45 & \phantom{0}62.92 & 114.37 \\
		25 & 0.0906 & 51.86 & \phantom{0}36.03 & \phantom{0}87.89 & 0.0908 & 51.86 & \phantom{0}63.08 & 114.94 \\
		30 & 0.1094 & 52.38 & \phantom{0}37.39 & \phantom{0}89.77 & 0.1095 & 52.38 & \phantom{0}63.85 & 116.22 \\
		35 & 0.1287 & 53.00 & \phantom{0}39.13 & \phantom{0}92.13 & 0.1289 & 53.00 & \phantom{0}65.10 & 118.10 \\
		40 & 0.1489 & 53.75 & \phantom{0}41.21 & \phantom{0}94.96 & 0.1490 & 53.75 & \phantom{0}66.80 & 120.56 \\
		45 & 0.1702 & 54.65 & \phantom{0}43.78 & \phantom{0}98.43 & 0.1703 & 54.66 & \phantom{0}69.01 & 123.67 \\
		50 & 0.1934 & 55.75 & \phantom{0}47.00 & 102.75 & 0.1934 & 55.75 & \phantom{0}71.84 & 127.58 \\
		$\infty$ & 0.6557 & 97.78 & 181.11 & 278.89 & 0.6557 & 97.78 & 200.38 & 298.16 \\
		\bottomrule
	\end{tabular}
\end{table}

Owing to a long computing time, it is difficult to carry out an exhaustive search for the optimal $\bar{\tau}$ in this example. However, we can still tell from Table~\ref{table:geographic} that both the value that minimizes the mean waiting time and the value that minimizes the mean time to service are between 10 and 20 minutes under the exponential distribution, and between 15 and 25 minutes under the lognormal distribution. Thus, $\bar{\tau} = 16$ minutes should produce near-optimal performance, as seen in Table~\ref{table:geographic} (in \emph{italics}). The above observation also implies the optimal $\bar{\tau}$ (either minimizing the mean waiting time or minimizing the mean time to service) under the lognormal distribution is greater than that under the exponential distribution. This is because with a larger variance, the lognormal distribution causes more substantial random variations in the service completion processes, thus resulting in greater load imbalance; accordingly, more customers should be sent to the shortest queue to offset this effect. When either the inter-appearance time distribution or the service distribution has a large variance, one may slightly increases $C^{*}$ and $C^{**}$ in \eqref{eq:root-excess} and \eqref{eq:reciprocal-root} in order for $\chi^{*}$ and $\chi^{**}$ to be closer to the optimal value. This would be helpful but may not be necessary, since the RJSQ policy is insensitive to $\chi$ around the optimal value. As long as the variance is not too large, we may still use $C^{*} = C^{**} = 0.4$ in the heuristic formulas and expect satisfactory performance.

\section{Proof Sketches}
\label{sec:Sketches}

We sketch out the proofs of Theorems~\ref{theorem:load-balancing}--\ref{theorem:workload-optimality}. The complete proofs are collected in the appendices.

\subsection{Convergence Lemmas for Delayed Arrival Processes}
\label{sec:Lemmas}

With traveling delays, the customer arrival process at a station is in general not a renewal process. In the $n$th distributed system, the number of customers that have arrived at station~$ k $ by time $t$ is
\begin{equation}\label{eq:station-arrivals} 
	A_{n,k}(t) = \sum_{j = 1}^{E_{n}(t)} \mathbb{1}_{ \{ \xi_{n}(j) = k, a_{n}(j) + \gamma_{n,k}(j) \leq t \} } \quad\mbox{for $ k = 1, \ldots, s $,}
\end{equation}
where $ a_{n}(j) $ is the time the $ j $th customer joins the system. To analyze this arrival process, consider
\begin{equation}\label{eq:Rmnk}
	R^{m}_{n,k}(j) = \sum_{i = 1}^{j} \mathbb{1}_{\{ \xi_{n}(i) = k, \boldsymbol{\gamma}(i) = \boldsymbol{d}_{m} \}} \quad\mbox{for $ m = 1, \ldots, b $ and  $ j \in \mathbb{N} $,}
\end{equation}
which is the number of customers sent from origin~$ m $ to station~$ k $ among the first $ j $ customers. The sequence $ \{R^{m}_{n,k}(j) : j \in \mathbb{N} \}$ is a basic component of the arrival process at station~$ k $, whose stochastic variations are caused by both the random origins and the randomized routing policy. If the customers' origins are known and the destinations are given by \eqref{eq:destination-n-m}, it can be decomposed into
\begin{equation}\label{eq:Rmnk-decomposition}
	R^{m}_{n,k}(j) = j p_{m} r_{m,k} + G^{m}_{k}(j) + \mathcal{E}^{m}_{n,k}(j) + M^{m}_{n,k}(j),
\end{equation}
where
\begin{align}
	G_{k}^{m}(j) & = \sum_{i = 1}^{j} \big(\mathbb{1}_{ \{ \kappa^{m}_{k-1} \leq u(i) < \kappa^{m}_{k}, \boldsymbol{\gamma}(i) = \boldsymbol{d}_{m}  \}} - p_{m}r_{m,k}\big), \label{eq:Gmk}\\
	\mathcal{E}^{m}_{n,k}(j) & =  \sum_{i = 1}^{j} p_{m} \varepsilon^{m}_{n,\pi_{n,k}(i)}(i), \label{eq:cali-Emnk} \\
	M^{m}_{n,k}(j) & = \sum_{i = 1}^{j} \big( \mathbb{1}_{\{ \xi_{n}(i) = k, \boldsymbol{\gamma}(i) = \boldsymbol{d}_{m} \}} - \mathbb{1}_{ \{ \kappa^{m}_{k-1} \leq u(i) < \kappa^{m}_{k}, \boldsymbol{\gamma}(i) = \boldsymbol{d}_{m}  \}} - p_{m}\varepsilon^{m}_{n,\pi_{n,k}(i)}(i)\big), \label{eq:Mmnk}
\end{align}
with $ \kappa^{m}_{0} = 0 $ and $ \kappa^{m}_{k} = \sum_{\ell = 1}^{k} r_{m,\ell} $. If the origins are unknown and the destinations are given by \eqref{eq:destination-n}, we take $ r_{m,k} = \mu_{k}/\mu $ and $ \varepsilon^{m}_{n,\pi_{n,k}(i)}(i) = \varepsilon_{n,\pi_{n,k}(i)} $ in \eqref{eq:Rmnk-decomposition}--\eqref{eq:Mmnk}. The first two terms on the right side of \eqref{eq:Rmnk-decomposition} are dictated by the routing plan: they are the expected number of customers sent from origin~$ m $ to station~$ k $ and the variation caused by randomized routing. The other two terms come from perturbation, representing the cumulative adjustment and the induced additional random variation. This decomposition enables us to analyze random variations from different sources separately.

The sequence $ \{G^{m}_{k}(j) : j \in \mathbb{N} \}$ is a random walk. The lemma below follows from the law of the iterated logarithm (Lemma~\ref{lemma:LIL}) and the strong approximation of random walks (Lemma~\ref{lemma:strong-approximations}).
\begin{lemma}\label{lemma:random-walk}
	Assume condition~\eqref{eq:rmk-inequality} holds. Then with probability 1,
	\begin{equation}\label{eq:Gmk-LIL}
		\limsup_{n \to \infty} \frac{\sup_{1\leq j \leq n^{\alpha}T}G^{m}_{k}(j)}{\sqrt{n^{\alpha}\log\log n}} = - \liminf_{n \to \infty} \frac{\inf_{1\leq j \leq n^{\alpha}T}G^{m}_{k}(j)}{\sqrt{n^{\alpha}\log\log n}} = \sqrt{2 T p_{m}r_{m,k}(1 - p_{m}r_{m,k})}
	\end{equation}
	for $ \alpha > 0 $, $ T > 0 $, $ m = 1, \ldots, b $, and $ k = 1, \ldots, s $. Moreover, there exists a positive number $ C^{m}_{k,T} $ and a standard Brownian motion $ \check{B}^{m}_{k} $ such that with probability 1,
	\begin{equation}\label{eq:Gmk-strong-approximation}
		\limsup_{n \to \infty} \frac{1}{\log n} \sup_{1 \leq j \leq nT} | G^{m}_{k}(j) - \sqrt{p_{m}r_{m,k}(1 - p_{m}r_{m,k})}\check{B}^{m}_{k}(j)| \leq C^{m}_{k,T}.
	\end{equation}
\end{lemma}

The sequence $ \{M^{m}_{n,k}(j) : j \in \mathbb{N} \}$ is a square-integrable martingale adapted to a certain filtration. Using Freedman's inequality (Lemma~\ref{lemma:Freedman}) and the martingale functional central limit theorem (Theorem~2.1 in \citealp{Whitt.2007}), we can estimate the magnitude of its variations and increments, as stated in the following lemma. Please refer to Appendix~\ref{sec:Proof-Mnkm} for the proof.
\begin{lemma}\label{lemma:martingale}
	Assume the conditions of Theorem~\ref{theorem:load-balancing} hold. Then with probability 1,
	\begin{equation}\label{eq:Mmnk-fluctuations}
		\limsup_{n \to \infty} \frac{\max_{1\leq j \leq n^{\alpha}T}|M^{m}_{n,k}(j)|}{\sqrt{n^{\alpha}(\chi_{n}\log n \vee 1)}} \leq \sqrt{10 T}
	\end{equation}
	for $ \alpha > 0 $, $ T > 0 $, $ m = 1, \ldots, b $, and $ k = 1, \ldots, s $. Let $ \{ q_{n} : n\in \mathbb{N} \} $ be a sequence of positive integers such that $ \lim_{n \to \infty} q_{n}/(\log n)^{2} = \infty $. Then  with probability 1,
	\begin{equation}\label{eq:Mmnk-increments}
		\limsup_{n \to \infty} \max_{1 \leq j \leq nT}\frac{\max_{1 \leq i \leq q_{n}}|M^{m}_{n,k}(j+i) - M^{m}_{n,k}(j)|}{\sqrt{q_{n}(\chi_{n}\log n \vee 1)}} \leq \sqrt{20}.
	\end{equation}
	Write $ \hat{M}^{m}_{n,k}(t) = M^{m}_{n,k}(\lfloor nt \rfloor)/\sqrt{n} $ for $ n \in \mathbb{N} $ and $ t \geq 0 $. Then, $ \sup_{0 \leq t \leq T} \hat{M}^{m}_{n,k}(t) \Rightarrow 0 $ as $ n \to \infty $.
\end{lemma}

By \eqref{eq:station-arrivals} and \eqref{eq:Rmnk}, the arrival process at station~$ k $ can be written as 
\begin{equation}\label{eq:station-arrivals-analysis} 
	A_{n,k}(t) = \sum_{m = 1}^{b} R^{m}_{n,k}\big(E_{n}((t-\sqrt{n} d_{m,k})^{+})\big),
\end{equation} 
where $E_{n}$ also contributes to random variations. We define the fluid-scaled renewal processes by $ \bar{E}_{n}(t) = E_{n}(nt)/n $ and $ \bar{S}_{n,k}(t) = S_{k}(nt)/n $, and define the diffusion-scaled processes by $ \tilde{E}_{n}(t) = \sqrt{n} (\bar{E}_{n}(t) - \lambda_{n} t) $ and $ \tilde{S}_{n,k}(t) = \sqrt{n} (\bar{S}_{n,k}(t) - \mu_{k} t ) $. By \eqref{eq:heavy-traffic} and the functional strong law of large numbers for renewal processes (Theorem~5.10 in \citealp{Chen.Yao.2001}), with probability 1,
\begin{equation}\label{eq:FSLLN}
	(\bar{E}_{n}, \bar{S}_{n,1}, \ldots, \bar{S}_{n,s}) \rightarrow (\mu\mathcal{I}, \mu_{1}\mathcal{I}, \ldots, \mu_{s}\mathcal{I} )  \quad \mbox{as $ n \to \infty $,}
\end{equation}
where $\mathcal{I}(t) = t$ for $t\geq 0$. By the functional central limit theorem for renewal processes (Theorem~5.11 in \citealp{Chen.Yao.2001}),
\begin{equation}\label{eq:FCLT}
	(\tilde{E}_{n}, \tilde{S}_{n,1}, \ldots, \tilde{S}_{n,s}) \Rightarrow (\tilde{E}, \tilde{S}_{1}, \ldots, \tilde{S}_{s}) \quad \mbox{as $ n \to \infty $,}
\end{equation}
where $ \tilde{E} $ and $ \tilde{S}_{1}, \ldots, \tilde{S}_{s} $ are mutually independent driftless Brownian motions, starting from $ 0 $ and having variances $ \mu c^{2}_{0} $ and $ \mu_{1} c^{2}_{1}, \ldots, \mu_{s} c^{2}_{s}$, respectively. If the normalized inter-appearance time distribution and the normalized service time distributions have finite fourth moments, we can obtain the following strong approximations (see Lemma~\ref{lemma:strong-approximations}).
\begin{lemma}\label{lemma:renewal-strong-approximations}
	If $ \mathbb{E}[z(1)^{4}] < \infty $, there exists a standard Brownian motion $ \check{B}_{0} $ such that with probability~1,
	\begin{equation}\label{eq:E-strong-approximation}
		\lim_{n \to \infty} \frac{1}{n^{1/4}} \sup_{0 \leq t \leq nT} | E_{n}(t) - \lambda_{n} t - c_{0} \check{B}_{0}(\lambda_{n} t ) | = 0 \quad \mbox{for $ T > 0 $.}
	\end{equation}
	If $ \mathbb{E}[w_{k}(1)^{4}] < \infty $ for $ k = 1, \ldots, s $, there exists a standard Brownian motion $ \check{B}_{k} $ such that with probability 1,
	\begin{equation}\label{eq:Sk-strong-approximation}
		\lim_{n \to \infty} \frac{1}{n^{1/4}} \sup_{0 \leq t \leq nT} | S_{k}(t) - \mu_{k} t - c_{k} \check{B}_{k}(\mu_{k} t ) | = 0 \quad \mbox{for $ T > 0 $.}
	\end{equation}
\end{lemma}

Lemma~\ref{lemma:Wiener-intervals} is part of Theorem~1 in \citet{Csorgo.Revesz.1979}, concerned with the increments of a standard Brownian motion over subintervals of $ [0, T] $. It is used with Lemma~\ref{lemma:renewal-strong-approximations} to estimate the increments of renewal processes in our proofs.
\begin{lemma}\label{lemma:Wiener-intervals}
	Let $ \{\check{B}(t) : t \geq 0 \} $ be a standard Brownian motion and $ f: \mathbb{R}_{+} \rightarrow \mathbb{R}_{+} $ a nondecreasing function such that $ 0 < f(t) \leq t $ and $ f(t)/t $ is nonincreasing for $ t > 0 $. Then with probability 1,
	\[
	\limsup_{T \to \infty} \sup_{0 \leq t \leq T - f(T)} \sup_{0 < u \leq f(T)} \frac{|\check{B}(t+u) - \check{B}(t)|}{\sqrt{2f(T)(\log T - \log f(T) + \log\log T)}} = 1.
	\]
\end{lemma}

\subsection{Theorems~\ref{theorem:load-balancing} and \ref{theorem:load-balancing-gap}}

We prove Theorems~\ref{theorem:load-balancing} and~\ref{theorem:load-balancing-gap} following similar procedures. Consider the event
\[
\Theta_{n} = \Big\{ \sup_{0 \leq t \leq T} \max_{k,\ell = 1, \ldots, s} | \tilde{L}_{n,k}(t) - \tilde{L}_{n,\ell}(t) | \geq \phi_{n} \Big\}.
\]
With certain choices of $ \phi_{n} $, we prove $ \lim_{n \to \infty} \mathbb{P}[\Theta_{n}] = 0 $ for Theorem~\ref{theorem:load-balancing} and $ \mathbb{P} [ \limsup_{n \to \infty} \Theta_{n} ] = 0 $ for Theorem~\ref{theorem:load-balancing-gap} under the respective conditions.

Rewrite the above event as $ \Theta_{n} = \{ \tau_{1,n} \leq T \} $, where
\[
\tau_{1,n} = \inf \Big\{ t \geq 0 : \max_{k,\ell = 1, \ldots, s} | \tilde{L}_{n,k}(t) - \tilde{L}_{n,\ell}(t) | \geq \phi_{n} \Big\}. 
\] 
Consider an outcome in $ \Theta_{n} $. Let stations~$ \bar{\ell}(n) $ and~$ \underline{\ell}(n) $ be the stations that have the longest and shortest queues at time $ \tau_{1,n} $, respectively. Then, $ \tilde{L}_{n,\bar{\ell}(n)}(\tau_{1,n}) - \tilde{L}_{n,\underline{\ell}(n)}(\tau_{1,n}) \geq \phi_{n} $. Writing $ \tilde{H}_{n}(t) = \mu \tilde{L}_{n,\bar{\ell}(n)}(t) - \sum_{\ell = 1}^{s} \tilde{Q}_{n,\ell}(t) $, we can prove $\tilde{H}_{n}(\tau_{1,n}) \geq \mu_{1}\phi_{n}$. Let
\[
\tau_{2,n} = \sup\Big\{ t \in [0, \tau_{1,n}] :  \tilde{H}_{n}(t-) < \frac{\mu_{1}\phi_{n}}{2}  \Big\} \quad \mbox{and} \quad \tau_{3,n} = \inf\big\{ t \in [\tau_{2,n}, \tau_{1,n}] :  \tilde{H}_{n}(t) \geq \mu_{1}\phi_{n}  \big\}.
\] 
Since the system is initially empty, $ 0 \leq \tau_{2,n} \leq \tau_{3,n} \leq \tau_{1,n} $ and $\tilde{H}_{n}(\tau_{3,n}) - \tilde{H}_{n}(\tau_{2,n}-) \geq \mu_{1}\phi_{n}/2$.

Let $\psi_{n}$ be a positive number and $ d_{0} = \max \{ d_{m,k} : m = 1,\ldots, b; k = 1, \ldots, s \} $. We decompose $ \Theta_{n} $ into three disjoint events: $ \Theta_{1,n} = \{ \tau_{1,n} \leq T, \tau_{2,n} \leq n^{-1/2} d_{0} \} $, $ \Theta_{2,n} = \{ \tau_{1,n} \leq T, \tau_{2,n} > n^{-1/2} d_{0}, \tau_{3,n} - \tau_{2,n} < \psi_{n} \} $, and $ \Theta_{3,n} = \{ \tau_{1,n} \leq T, \tau_{2,n} > n^{-1/2} d_{0}, \tau_{3,n} - \tau_{2,n} \geq \psi_{n} \} $. We prove these events are asymptotically negligible by selecting some appropriate $ \psi_{n}$.

It is simple to show
\begin{align*} 
	\mathbb{P}\Big[ \limsup_{n \to \infty} \Theta_{1,n} \Big] & \leq \mathbb{P}\Big[\limsup_{n \to \infty} \{\tau_{2,n} \leq n^{-1/2} d_{0}\}\Big] \\
	& \leq \mathbb{P} \Big[ \limsup_{n \to \infty} \Big\{\sup_{0 \leq t \leq d_{0}} \max_{k = 1, \ldots, s} \tilde{L}_{n,k} (n^{-1/2}t) \geq \frac{\mu_{1}\phi_{n}}{2(\mu - \mu_{1})} \Big\}\Big].
\end{align*}
Since the system is initially empty, the queue lengths under the RJSQ policy, which sends customers to each station with a probability approximately proportional to the service capacity, cannot be too long at the beginning. We can thus prove $ \mathbb{P}[\limsup_{n \to \infty} \Theta_{1,n}] = 0 $ when $\phi_{n}$ is sufficiently large. 

Consider $\Theta_{2,n}$ and suppose $\psi_{n}$ is given. Since $\tau_{3,n} - \tau_{2,n} < \psi_{n}$, we render the event $\{ \tilde{H}_{n}(\tau_{3,n}) - \tilde{H}_{n}(\tau_{2,n}-) \geq \mu_{1}\phi_{n}/2 \}$ asymptotically negligible by making $\phi_{n}$ (which depends on $\psi_{n}$) sufficiently large. Then, $\Theta_{2,n}$ will be asymptotically negligible. We determine $ \psi_{n} $ by analyzing $ \Theta_{3,n} $ as follows.

Assume $\psi_{n} \geq n^{-1/2} (2d_{0} \vee 1) $ and write $ \tau_{4,n} = \tau_{3,n} -  \psi_{n}/2 - n^{-1/2} d_{0} $. Then, $ \tau_{2,n} \leq \tau_{4,n} < \tau_{3,n} $ and $ \tilde{H}_{n} (\tau_{3,n}) - \tilde{H}_{n} (\tau_{4,n}-) > 0 $ on $ \Theta_{3,n} $. Since $ \tilde{H}_{n}(t) \geq \mu_{1}\phi_{n}/2 $ for $ t \in [\tau_{2,n}, \tau_{3,n}] $, station~$\bar{\ell}(n)$ cannot have the shortest queue within this period. If the $ j $th customer joins the system during $[\tau_{2,n}, \tau_{3,n}] $, then $ \varepsilon_{n,\pi_{n,\bar{\ell}(n)}(j)} \leq - \delta_{0} \chi_{n}$ by  condition~\eqref{item:cond-1-n}. Because the consecutive perturbation coefficients for routing customers to station~$\bar{\ell}(n)$ are all negative, $\mu \tilde{L}_{n,\bar{\ell}(n)}(t)$ tends to decrease faster than $ \sum_{\ell = 1}^{s} \tilde{Q}_{n,\ell}(t) $ for $t \in [\tau_{4,n}+n^{-1/2}d_{0}, \tau_{3,n}]$. Then, we can prove $ \tilde{H}_{n} (\tau_{3,n}) - \tilde{H}_{n} (\tau_{4,n}-) \leq 0 $ when $\psi_{n}$ is sufficiently large. Because $ \tilde{H}_{n} (\tau_{3,n}) - \tilde{H}_{n} (\tau_{4,n}-) > 0 $ on $ \Theta_{3,n} $, the event must be asymptotically negligible.

According to the above discussion, we need to determine $ \psi_{n} $ and $ \phi_{n} $ in order for $ \Theta_{n} $ to be asymptotically negligible. Under the respective conditions of Theorems~\ref{theorem:load-balancing} and~\ref{theorem:load-balancing-gap}, Lemmas~\ref{lemma:random-walk}--\ref{lemma:Wiener-intervals} are used to first specify the order of magnitude for $ \psi_{n} $ and then for $ \phi_{n} $. Please refer to Appendix~\ref{sec:Proofs-Load-Balancing}.

\subsection{Theorems~\ref{theorem:workload-difference} and~\ref{theorem:CRP}}

Let $ B_{n,k}^{\star}(t) $ be the amount of time that the server in the SSP has spent serving class-$ k $ customers by time $ t $. The server's cumulative busy time is $ B_{n}^{\star}(t) = \sum_{k = 1}^{s} B_{n,k}^{\star}(t) $. The fluid-scaled versions are defined by $ \bar{B}_{n,k}^{\star}(t) = B_{n,k}^{\star}(nt)/n $ and $ \bar{B}_{n}^{\star}(t) = B_{n}^{\star}(nt)/n $. 

Since the corresponding customers have identical service requirements, the cumulative stationed workloads must be equal in the two systems. Then by \eqref{eq:workload-difference}, $ \tilde{\Gamma}^{\star}_{n}(t) = \sqrt{n} (\mu \bar{B}_{n}^{\star}(t) - \sum_{k = 1}^{s} \mu_{k} \bar{B}_{n,k}(t)) $. Because the sample paths of $\tilde{\Gamma}^{\star}_{n}$ are continuous and piecewise linear,
\[
\sup_{0 \leq t \leq T} \tilde{\Gamma}^{\star}_{n}(t) = \sup_{t \in \mathcal{S}^{\star}_{n} } \tilde{\Gamma}^{\star}_{n}(t) = \sup_{t \in \mathcal{S}^{\star}_{n} } \tilde{\Gamma}^{\star}_{n}(t-),
\]
where $ \mathcal{S}^{\star}_{n} = \{ t \in (0,T] : \mbox{$ \partial_{-} \tilde{\Gamma}^{\star}_{n}(t) > 0 $} \} \cup \{0\} $, with $\partial_{-}$ denoting the left derivative. At any time, the cumulative busy time of a server either increases at rate $ 1 $ or stays unchanged. Then for $t \in \mathcal{S}^{\star}_{n}\setminus\{0\}$, $ \partial_{-} \bar{B}_{n}^{\star}(t) = 1 $ and at least one of $\partial_{-}\bar{B}_{n,1}(t), \ldots, \partial_{-}\bar{B}_{n,s}(t) $ is zero. Hence, at least one of $\tilde{Q}_{n,1}(t-), \ldots, \tilde{Q}_{n,s}(t-)$ is zero. By Theorems~\ref{theorem:load-balancing} and~\ref{theorem:load-balancing-gap} and under the respective conditions,
\[
\max_{k = 1, \ldots, s} \sup_{t \in \mathcal{S}^{\star}_{n}} \tilde{Q}_{n,k}(t-) \Rightarrow 0 \quad \mbox{as $ n \to \infty $,} 
\]
or with probability 1,
\[
\sup_{t \in \mathcal{S}^{\star}_{n}} \tilde{Q}_{n,k}(t-) < \max \Big\{ C \mu_{k} \chi_{n}, C' \mu_{k} \frac{\log \log n + \log (n \chi_{n}^{2})}{\sqrt{n} \chi_{n}} \Big\} \quad \mbox{for $n$ sufficiently large.}
\]

Write $\tilde{W}_{n,k}(t) = W_{n,k}(nt)/\sqrt{n}$. Then, $\tilde{W}_{n,k}(t) \approx \tilde{Q}_{n,k}(t)$ because $\{w_{k}(i) : i\in\mathbb{N} \} $ is a sequence of i.i.d.\ random variables with mean 1. We complete the proof of Theorem~\ref{theorem:workload-difference} using
\[
\sup_{t \in \mathcal{S}^{\star}_{n} } \tilde{\Gamma}^{\star}_{n}(t-) \leq \sum_{k = 1}^{s} \sup_{t \in \mathcal{S}^{\star}_{n} } \tilde{W}_{n,k}(t-) \approx \sum_{k = 1}^{s} \sup_{t \in \mathcal{S}^{\star}_{n} } \tilde{Q}_{n,k}(t-).
\] 

Let $Q^{\star}_{n}(t)$ be the customer count in the station of the SSP at time $t$, with the scaled version $\tilde{Q}_{n}^{\star}(t) = Q_{n}^{\star}(nt)/\sqrt{n}$. Since the workloads in the two systems are close, their customer counts should also be close.

\begin{proposition}\label{prop:asymptotic-equivalence}
	Assume the conditions of Theorem~\ref{theorem:load-balancing} hold and all the SSPs are initially empty. Then for $ T > 0 $,
	\[  
	\sup_{0 \leq t \leq T} \Big\vert \sum_{k = 1}^{s} \tilde{Q}_{n,k}(t) - \tilde{Q}^{\star}_{n}(t) \Big\vert \Rightarrow 0 \quad \mbox{as $ n \to \infty $.}
	\]
\end{proposition}

In the SSP, the scaled customer count process converges to a reflected Brownian motion.
\begin{proposition}\label{prop:star-limit}
	Assume the conditions of Theorem~\ref{theorem:load-balancing} hold and all the SSPs are initially empty. Then, $ \tilde{Q}_{n}^{\star} \Rightarrow \tilde{Q} $ as $ n \to \infty $, where $ \tilde{Q} $ is a one-dimensional reflected Brownian motion starting from $ 0 $ with drift $ -\beta \mu $ and variance $ \mu c^{2}_{0} + \sum_{k = 1}^{s} \mu_{k} c^{2}_{k} $.
\end{proposition}

Using Propositions~\ref{prop:asymptotic-equivalence} and~\ref{prop:star-limit}, we follow a continuous mapping approach and prove Theorem~\ref{theorem:CRP}. Please refer to Appendix~\ref{sec:Proofs-CRP}.

\subsection{Theorem~\ref{theorem:workload-optimality}}

Let $A^{\dagger}_{n,k}(t)$ be the number of class-$k$ customers that have arrived at the station of the MDSP by time~$t$, with the fluid-scaled version $\bar{A}^{\dagger}_{n,k}(t) = A^{\dagger}_{n,k}(nt)/n$. Clearly, $A^{\dagger}_{n,k}(t) \geq A_{n,k}(t)$. Since the cumulative workloads are equal in the two systems,
\[
\sum_{k = 1}^{s} \sum_{i = 1}^{A_{n,k}(t)}w_{k}(i) + U_{n}(t) = \sum_{k = 1}^{s} \sum_{i = 1}^{A^{\dagger}_{n,k}(t)} w_{k}(i) + U^{\dagger}_{n}(t).
\] 
Write $\tilde{U}_{n}(t) = U_{n}(nt)/\sqrt{n}$ and $\tilde{U}^{\dagger}_{n}(t) = U^{\dagger}_{n}(nt)/\sqrt{n}$. Because $\{ w_{k}(i) : i\in\mathbb{N} \}$ is a sequence of i.i.d.\ random variables with mean 1, $ \tilde{U}_{n}(t) - \tilde{U}^{\dagger}_{n}(t) \approx \sum_{k = 1}^{s} \sqrt{n} ( \bar{A}^{\dagger}_{n,k}(t) - \bar{A}_{n,k}(t) ) $. Assume the conditions of Theorem~\ref{theorem:load-balancing} hold. Using \eqref{eq:Mmnk-increments} and \eqref{eq:FSLLN}, we can prove 
\[ 
\sup_{0 \leq t \leq T} \sqrt{n}\big( \bar{A}^{\dagger}_{n,k}(t) - \bar{A}_{n,k}(t) \big) < \max\big\{ C^{\dagger}_{k} \chi_{n}, C^{\ddagger}_{k} n^{-1/4}\sqrt{\chi_{n}\log n \vee 1} \big\} \quad \mbox{for $n$ sufficiently large,}
\]
where $C^{\dagger}_{k}$ and $C^{\ddagger}_{k}$ are two positive numbers. Then, an upper bound for the difference between the en route workloads can be specified.

Following the approach in the proof of Theorem~\ref{theorem:workload-difference}, we estimate the difference between the total stationed workloads. It follows from Proposition~\ref{prop:JNQ-workload-comparision} that $ \sup_{0 \leq t \leq T} \tilde{\Gamma}^{\dagger}_{n}(t) = \sup_{t \in \mathcal{S}^{\dagger}_{n} } \tilde{\Gamma}^{\dagger}_{n}(t-) $, where $ \mathcal{S}^{\dagger}_{n} = \{ t \in (0,T] : \mbox{$ \partial_{-} \tilde{\Gamma}^{\dagger}_{n}(t) > 0 $} \} \cup \{0\} $. By Theorems~\ref{theorem:load-balancing} and~\ref{theorem:load-balancing-gap} and under the respective conditions,
\[
\max_{k = 1, \ldots, s} \sup_{t \in \mathcal{S}^{\dagger}_{n} } \tilde{W}_{n,k}(t-) \Rightarrow 0 \quad \mbox{as $ n \to \infty $,} 
\]
or with probability 1,
\[
\sup_{t \in \mathcal{S}^{\dagger}_{n} } \tilde{W}_{n,k}(t-)  <  \max \Big\{ 2C \mu \chi_{n}, 2C' \mu \frac{\log \log n + \log (n \chi_{n}^{2})}{\sqrt{n} \chi_{n}} \Big\} \quad \mbox{for $n$ sufficiently large.}
\]
Then, we can conclude the proof using the fact that
\[
\sup_{t \in \mathcal{S}^{\dagger}_{n} } \tilde{\Gamma}^{\dagger}_{n}(t-) \leq \sum_{k = 1}^{s} \sup_{t \in \mathcal{S}^{\dagger}_{n} } \tilde{W}_{n,k}(t-) + \sup_{0 \leq t \leq T} \big(\tilde{U}_{n}(t) - \tilde{U}^{\dagger}_{n}(t)\big).
\] 
Please refer to Appendix~\ref{sec:Proof-Capacity}.

\section{Concluding Remarks}
\label{sec:Conclusion}

We devised and analyzed the RJSQ policy for a queueing system with parallel stations distant from customers. In contrast to the JSQ policy, which induces queue length oscillations in the presence of information delays, the RJSQ policy retains complete resource pooling when customers' traveling delays are on the same order of magnitude as their potential waiting times. A strong approximation approach was developed to analyze load imbalance between the stations, enabling us to fine-tune the balancing fraction and optimize the RJSQ policy.

This study is the first step toward our research on \emph{stochastic spatial networks}---an example of which is the service system with geographically separated stations in Section~\ref{sec:Routing-Capacity-Planning}. In such a network, if customers may visit other stations upon service completion in one station, can we still rely on the insights gleaned here to design a routing policy and retain complete resource pooling? A stochastic network model with traveling delays between the stations may have applications in transportation, logistics, sharing economy services, etc. For instance, the stations could represent multiple regions in a ride-hailing system, between which passengers are transported by cars; empty-car routing is often needed to balance supply and demand across the regions \citep{Braverman.et.al.2019}.

It is known that routing policies of the \emph{backpressure} type are throughput-optimal and may asymptotically minimize workload in queueing networks \citep{Tassiulas.Ephremides.1992,Stolyar.2004,Dai.Lin.2005,Dai.Lin.2008}. Motivated by ride-hailing applications, \citet{Kanoria.Qian.2024} proposed a policy of this type for closed queueing networks; assuming customers to be transferred from one station to another \emph{instantaneously}, they proved the policy to be near-optimal and obtained the optimality gap. The backpressure policies, however, also suffer from queue length oscillations when there are information delays. Our study suggests a simple randomization technique would suffice to mitigate the oscillation phenomenon. Under certain conditions, a randomized version of the backpressure policy may still attain throughput optimality and workload minimization. We believe randomized policies will play a crucial role in load balancing for stochastic spatial networks.

Numerous recent studies are concerned with asymptotic analysis of the JSQ policy in many-server regimes \citep{Mukherjee.et.al.2016,Eschenfeldt.Gamarnik.2018,Gupta.Walton.2019,Banerjee.Mukherjee.2019,Banerjee.Mukherjee.2020,Braverman.2020,Braverman.2023,Cao.et.al.2021,Budhiraja.et.al.2021,Hurtado-Lange.Maguluri.2022}. It is not straightforward to extend our approach to a many-server setting; we would also leave this topic for future research.

\appendix

\section*{Appendices}

The complete proofs of all theoretical results are given here. Fundamental lemmas are summarized in Appendix~\ref{sec:additional-lemmas}. Appendix~\ref{sec:Proof-Mnkm} gives the proof of Lemma~\ref{lemma:martingale}. Additional convergence lemmas are collected in Appendix~\ref{sec:Preliminary-Results}.  We prove Theorem~\ref{theorem:load-balancing}, Theorem~\ref{theorem:load-balancing-gap}, and Corollary~\ref{corollary:optimal-gap} in  Appendix~\ref{sec:Proofs-Load-Balancing}, Theorems~\ref{theorem:workload-difference} and~\ref{theorem:CRP} in Appendix~\ref{sec:Proofs-CRP}, and Theorem~\ref{theorem:workload-optimality} in Appendix~\ref{sec:Proof-Capacity}. We prove all the propositions in Appendix~\ref{sec:Proofs-Propositions} and the lemmas given in the appendices in Appendix~\ref{sec:Proofs-Additional-Lemmas}.

\subsubsection*{Notation}
The symbols $\mathbb{N}$, $\mathbb{N}_{0}$, $\mathbb{R}$, and $\mathbb{R}_{+}$ are used for the sets of positive integers, nonnegative integers, real numbers, and nonnegative real numbers, respectively. For $s \in \mathbb{N}$, $\mathbb{R}_{+}^{s}$ is the set of $s$-dimensional nonnegative vectors. The space of functions from $\mathbb{R}_{+} $ to $ \mathbb{R}$ that are right-continuous on $\mathbb{R}_{+}$ and have left limits on $(0, \infty)$ is denoted by $\mathbb{D}$. For $f \in \mathbb{D}$, $t \geq 0$, and $0 \leq t_{1} \leq t_{2}$, we write $\Delta f(t) = f(t) - f(t-)$, where $f(t-) = \lim_{u \uparrow t} f(u)$ with the convention $f(0-) = 0$, and write $ f(t_{1}, t_{2}] = f(t_{2}) - f(t_{1})$ and $ f[t_{1}, t_{2}] = f(t_{2}) - f(t_{1}-) $. The identity function on $\mathbb{R}_{+}$ is denoted by $\mathcal{I}$, with $\mathcal{I}(t) = t$ for $t \geq 0$. For $g, g' : \mathbb{N} \rightarrow \mathbb{R}$, we write $\Delta g(n) = g(n) - g(n-1)$, with the convention $\Delta g(1) = g(1)$, write $g(n) = O(g'(n))$ if $\limsup_{n \to \infty} |g(n)/g'(n)| < \infty$, and write $g(n) = o(g'(n))$ if $\lim_{n \to \infty} g(n)/g'(n) = 0$. For $a, a' \in \mathbb{R}$, $a \vee a' = \max \{a, a'\}$ and $a \wedge a' = \min \{ a, a'\}$.

All random variables and stochastic processes are defined on a common probability space $( \Omega, \mathcal{F}, \mathbb{P})$. We reserve $\mathbb{E}[\cdot]$ for expectation and $\mathbb{1}_{\Theta}$ for the indicator random variable of $\Theta \in \mathcal{F}$. For two random variables $X$ and $Y$, we write $X \asl Y$, $X \asleq Y$, and $X \aseq Y$ to denote $\mathbb{P}[X < Y]=1$, $\mathbb{P}[X \leq Y]=1$, and $\mathbb{P}[X = Y]=1$, respectively. For a collection of random variables $\{ X_{n} : n \in \mathbb{I} \}$, where $\mathbb{I}$ is an index set, we use $\sigma\{ X_{n} : n \in \mathbb{I} \}$ to denote the $\sigma$-field generated by $\{ X_{n} : n \in \mathbb{I} \}$. All continuous-time stochastic processes are assumed to have sample paths in $\mathbb{D}$, which is endowed with the Skorokhod $J_{1}$-topology \citep{Billingsley.1999}. For a sequence of random variables or stochastic processes $\{ X_{n} : n \in \mathbb{N} \}$, we write $X_{n} \Rightarrow X$ for the convergence of $X_{n}$ to $X$ in distribution.

\section{Preliminaries}
\label{sec:additional-lemmas}

The proofs of the main theorems rely on Freedman's inequality, the law of the iterated logarithm, and strong approximations for random walks and renewal processes. These results are given in this section as a series of lemmas for reference.

The first lemma is identical to Theorem~1.6 in \citet{Freedman.1975}, known as Freedman's inequality. We use this inequality, along with the Borel--Cantelli lemma, to estimate the increments of martingales that have uniformly bounded differences. 
\begin{lemma}\label{lemma:Freedman}
	Let $ \{ M(j) : j \in \mathbb{N}_{0} \} $ be a locally square-integrable martingale adapted to a filtration $ \{ \mathcal{F}(j) : j\in\mathbb{N}_{0} \} $, with predictable quadratic variation $ \langle M \rangle (j) = \sum_{i = 1}^{j} \mathbb{E}[\Delta M(i)^2 \,|\, \mathcal{F}(i-1)] $. If $ M(0) = 0 $ and $ \sup_{j \in \mathbb{N}} | \Delta M(j) | \leq 1 $, then for $ a_{1},a_{2} > 0 $,
	\[
	\mathbb{P}[ \mbox{$ M(j) \geq a_{1} $ and $ \langle M \rangle(j) \leq a_{2} $ for some $ j \in \mathbb{N} $} ] \leq \exp\Big( -\frac{a_{1}^{2}}{2(a_{1}+a_{2})} \Big).
	\]
\end{lemma}

We next present the law of the iterated logarithm for random walks and renewal processes (see, e.g., Theorems~1.9.1 and~3.11.1 in \citealp{Gut.2009}).
\begin{lemma}\label{lemma:LIL}
	Let $ \{ x(j) : j \in \mathbb{N} \} $ be a sequence of i.i.d.\ nonnegative random variables with mean~$ 1 $ and coefficient of variation $ c_{x} < \infty $. Then,
	\[  
	\limsup_{n \to \infty} \frac{\sup_{1\leq j \leq n} \big( \sum_{i = 1}^{j}x(i) - j \big)}{\sqrt{2 n \log\log n}} \aseq - \liminf_{n \to \infty} \frac{\inf_{1\leq j \leq n} \big( \sum_{i = 1}^{j}x(i) - j \big)}{\sqrt{2 n \log\log n}} \aseq c_{x}.
	\]
	Let $ \{ N(t) : t \geq 0 \} $ be the associated renewal process, given by $ N(t) = \max\{ j \in \mathbb{N}_{0} : \sum_{i = 1}^{j} x(j) \leq t \} $. Then,
	\[  
	\limsup_{T \to \infty} \frac{\sup_{0 \leq t \leq T} (N(t) - t)}{\sqrt{2 T \log\log T}} \aseq - \liminf_{T \to \infty} \frac{\inf_{0 \leq t \leq T} (N(t) - t)}{\sqrt{2 T \log\log T}} \aseq c_{x}.
	\]
\end{lemma}

If the above sequence of random variables fulfills an additional moment condition, we can obtain refined sample-path approximations. The third lemma corresponds to Theorem~5.14 in \citet{Chen.Yao.2001} and Corollary~3.1 in \citet{Csorgo.et.al.1987a}.

\begin{lemma}\label{lemma:strong-approximations}
	Assume the conditions of Lemma~\ref{lemma:LIL} hold. If $ \mathbb{E}[x(1)^{\alpha}] < \infty $ for some $ \alpha > 2 $, then there exist two standard Brownian motions $ \check{B}' $ and $ \check{B}'' $ such that
	\[
	\lim_{T \to \infty} \frac{1}{T^{1/\alpha}} \sup_{0 \leq t \leq T} \Big\vert \sum_{i = 1}^{\lfloor t \rfloor} x(i) - t - c_{x} \check{B}'(t)\Big\vert \aseq \lim_{T \to \infty} \frac{1}{T^{1/\alpha}} \sup_{0 \leq t \leq T} \vert N(t) - t - c_{x} \check{B}''(t) \vert \aseq 0.
	\]
	If $ \mathbb{E}[\exp(t x(1))] < \infty $ in a neighborhood of $ t = 0 $, then there exists a positive number $ C_{x} $ such that
	\begin{align*}
		\limsup_{T \to \infty} \frac{1}{\log T} \sup_{0 \leq t \leq T} \Big\vert \sum_{i = 1}^{\lfloor t \rfloor} x(i) - t - c_{x} \check{B}'(t)\Big\vert & \asl C_{x},\\
		\limsup_{T \to \infty} \frac{1}{\log T} \sup_{0 \leq t \leq T} \vert N(t) - t - c_{x} \check{B}''(t) \vert & \asl C_{x}.
	\end{align*}
\end{lemma}

\section{Proof of Lemma~\ref{lemma:martingale}}
\label{sec:Proof-Mnkm}

For $ n \in \mathbb{N} $, we define a filtration $ \mathbb{F}_{n} = \{ \mathcal{F}_{n}(j) : j \in \mathbb{N}_{0} \} $ by $ \mathcal{F}_{n}(0) = \sigma\{ \pi_{n,k}(1) : k = 1, \ldots, s \} $ and $ \mathcal{F}_{n}(j) = \sigma\{ \pi_{n,k}(1), \pi_{n,k}(i+1), u(i), \boldsymbol{\gamma}(i) : i = 1, \ldots, j; k = 1, \ldots, s \} $ for $ j \in \mathbb{N} $. By \eqref{eq:kappa}--\eqref{eq:kappa-m}, $ \{ \xi_{n}(j) : j \in \mathbb{N} \} $ is adapted to $ \mathbb{F}_{n} $. We use $ \mathcal{F}_{n}(j) $ to keep track of the ranking history of queue lengths upon the $ (j+1) $st customer's appearance, along with the previous customers' origins and destinations. Since $ u(j) $ and $ \boldsymbol{\gamma}(j) $ are independent of $ \mathcal{F}_{n}(j-1) $, $ \{ M_{n,k}^{m}(j) : j \in \mathbb{N} \} $ is a square-integrable martingale adapted to $ \mathbb{F}_{n} $, with $ |\Delta M_{n,k}^{m}(j)| \leq 1 + \chi_{n} $. By \eqref{eq:kappa-m} and condition \eqref{item:cond-1-m}, $ \vert r_{m,k} - ((\kappa^{m}_{n,k}(j) \wedge \kappa^{m}_{k}) - (\kappa^{m}_{n,k-1}(j)\vee \kappa^{m}_{k-1}))\vert \leq 2 \varepsilon^{m}_{n,1}(j) $. Then by \eqref{eq:pm-epsilon-chi},
\begin{align*}
	\mathbb{E}[\Delta M_{n,k}^{m}(j)^{2}\,|\, \mathcal{F}_{n}(j-1)]
	& = 2p_{m}r_{m,k} - 2p_{m}\mathbb{E}\big[ \mathbb{1}_{ \{ \kappa^{m}_{n,k-1}(j)\vee \kappa^{m}_{k-1} \leq u(j) < \kappa^{m}_{n,k}(j) \wedge \kappa^{m}_{k}\}} \,|\, \mathcal{F}_{n}(j-1)\big] \\ & \quad + p_{m}\varepsilon^{m}_{n,\pi_{n,k}(j)}(j) - p^{2}_{m} \varepsilon^{m}_{n,\pi_{n,k}(j)}(j)^{2}\\
	& \leq 5 \chi_{n}.
\end{align*}
Hence, $ \langle M_{n,k}^{m} \rangle (j) \leq 5 j \chi_{n}$ for $ j \in \mathbb{N} $.

We first prove $ \sup_{0 \leq t \leq T} \hat{M}^{m}_{n,k}(t) \Rightarrow 0 $ as $ n \to \infty $. Put $ \mathcal{G}_{n}(t) = \mathcal{F}_{n}(\lfloor n t \rfloor) $. Then, $ \hat{M}_{n,k}^{m} $ is a martingale adapted to $ \{ \mathcal{G}_{n}(t) : t \geq 0 \} $ with $ \langle \hat{M}_{n,k}^{m} \rangle (t) \leq 5 \chi_{n} t $. It follows from \eqref{eq:epsilon-1} that $ \lim_{n \to \infty} \langle \hat{M}_{n,k}^{m} \rangle (t) = 0 $. Because $ | \Delta \hat{M}^{m}_{n,k} (t) | \leq (1 + \chi_{n})/\sqrt{n} $ and $ |\Delta \langle \hat{M}^{m}_{n,k} \rangle (t) | \leq 5\chi_{n}/n $, 
\[  
\lim_{n \to \infty} \mathbb{E}\Big[ \sup_{0 < u \leq t} | \Delta \hat{M}^{m}_{n,k} (u)|^{2} \Big] = \lim_{n \to \infty} \mathbb{E}\Big[ \sup_{0 < u \leq t} | \Delta \langle \hat{M}^{m}_{n,k} \rangle (u)| \Big] = 0.
\]
Then, the assertion follows from the martingale functional central limit theorem (Theorem~2.1 in \citealp{Whitt.2007}).

To prove \eqref{eq:Mmnk-fluctuations}, write $ \eta_{n}(\delta) = (1+\delta)\sqrt{10  n^{\alpha} T (\chi_{n}\log n \vee 1) } $ for $ \delta > 0 $. If $ \chi_{n} \log n \leq 1 $, then for $ n $ sufficiently large, $ (1+\chi_{n})\eta_{n}(\delta)\log n = (1+\chi_{n})(1+\delta)\sqrt{10  n^{\alpha} T } \log n \leq 5 \delta n^{\alpha} T $, in which case
\begin{align*}
	\frac{\eta_{n}(\delta)^{2} }{2((1+\chi_{n})\eta_{n}(\delta)+5 n^{\alpha} T \chi_{n})} & = \frac{\eta_{n}(\delta)^{2} \log n }{2((1 + \chi_{n}) \eta_{n}(\delta)\log n + 5n^{\alpha} T \chi_{n}\log n)}\\
	& \geq \frac{10(1+\delta)^{2} n^{\alpha} T \log n }{10 \delta n^{\alpha} T  + 10 n^{\alpha} T  } \\
	& = (1 + \delta) \log n.
\end{align*}
If $ \chi_{n} \log n > 1 $, $(1+\chi_{n})\eta_{n}(\delta)  < (1+\chi_{n})(1+\delta)\sqrt{10  n^{\alpha} T } \chi_{n} \log n \leq 5 \delta n^{\alpha} T \chi_{n}$ for $ n $ sufficiently large, in which case
\[
\frac{\eta_{n}(\delta)^{2} }{2((1+\chi_{n})\eta_{n}(\delta)+5 n^{\alpha} T \chi_{n})} \geq \frac{10 (1+\delta)^{2} n^{\alpha} T \chi_{n}\log n }{10 \delta n^{\alpha} T \chi_{n} + 10 n^{\alpha} T \chi_{n} } = (1 + \delta) \log n.
\]
Hence, by Lemma~\ref{lemma:Freedman},
\begin{align*}
	\mathbb{P}\Big[ \max_{1\leq j \leq n^{\alpha}T} M_{n,k}^{m}(j) \geq \eta_{n}(\delta) \Big]
	& \leq \mathbb{P}\big[ \mbox{$ M_{n,k}^{m}(j) \geq \eta_{n}(\delta) $ and $ \langle M_{n,k}^{m} \rangle(j) \leq 5 n^{\alpha} T \chi_{n} $ for some $ j \in \mathbb{N} $} \big]\\
	& \leq \exp\Big( -\frac{\eta_{n}(\delta)^{2} }{2((1+\chi_{n})\eta_{n}(\delta)+5 n^{\alpha} T \chi_{n})} \Big) \\
	& \leq \exp( -(1+\delta) \log n )\\
	& = n^{-(1+\delta)}.
\end{align*}
Since $ \sum_{n = 1}^{\infty} n^{-(1+\delta)} < \infty $, by the Borel--Cantelli lemma, 
\[
\limsup_{n \to \infty} \frac{\max_{1\leq j \leq n^{\alpha}T} M_{n,k}^{m}(j)}{\eta_{n}(\delta)} \asl 1 \quad \mbox{for all $ \delta > 0 $.}
\]
Taking $ \delta \rightarrow 0 $, we obtain
\[
\limsup_{n \to \infty} \frac{\max_{1\leq j \leq n^{\alpha}T} M_{n,k}^{m}(j)}{\sqrt{n^{\alpha}(\chi_{n}\log n \vee 1) }} \asleq \sqrt{10T},
\]
and by symmetry, 
\[
\limsup_{n \to \infty} \frac{\max_{1\leq j \leq n^{\alpha}T} \big(- M_{n,k}^{m}(j)\big)}{\sqrt{n^{\alpha}(\chi_{n}\log n \vee 1)}} \asleq \sqrt{10T}.
\]
Combining these two inequalities, we complete the proof of \eqref{eq:Mmnk-fluctuations}.

The proof of \eqref{eq:Mmnk-increments} is similar. Put $ \mathbb{F}^{j}_{n} = \{ \mathcal{F}_{n}(j+i) : i \in \mathbb{N}_{0} \} $ and $ M^{m,j}_{n,k}(i) = M^{m}_{n,k}(j+i) - M^{m}_{n,k}(j) $ for $ j \in \mathbb{N} $. Then, $ \{ M^{m,j}_{n,k}(i) : i \in \mathbb{N}_{0} \} $ is a square-integrable martingale adapted to $ \mathbb{F}^{j}_{n} $, with $ |\Delta M^{m,j}_{n,k}(i)| \leq 1 + \chi_{n} $ and $ \langle M_{n,k}^{m,j} \rangle(i) = \langle M_{n,k}^{m} \rangle(j+i) - \langle M_{n,k}^{m} \rangle(j) \leq 5i\chi_{n} $ for $ i \in \mathbb{N} $. 

Write $\eta'_{n}(\delta) =  \sqrt{ 10 (1 + \delta)(2 + \delta)q_{n}(\chi_{n}\log n \vee 1)}$ for $ \delta > 0 $. If $ \chi_{n} \log n \leq 1 $, then for $n$ sufficiently large, $(1+\chi_{n})\eta'_{n}(\delta)\log n = (1+\chi_{n})\sqrt{ 10 (1 + \delta)(2 + \delta) q_{n}} \log n \leq 5 \delta q_{n}$, since $ \lim_{n \to \infty} q_{n}/(\log n)^{2} = \infty $. In this case,
\begin{align*}
	\frac{\eta'_{n}(\delta)^{2} }{2((1+\chi_{n})\eta'_{n}(\delta)+5 q_{n} \chi_{n})} & = \frac{\eta'_{n}(\delta)^{2} \log n }{2((1 + \chi_{n}) \eta'_{n}(\delta)\log n + 5 q_{n} \chi_{n}\log n)}\\
	& \geq \frac{10(1+\delta)(2+\delta) q_{n} \log n }{10 \delta q_{n}  + 10 q_{n}  } \\
	& = (2 + \delta) \log n.
\end{align*}
If $ \chi_{n} \log n > 1 $, $(1+\chi_{n})\eta'_{n}(\delta)  < (1+\chi_{n})\sqrt{ 10 (1 + \delta)(2 + \delta)  q_{n}} \chi_{n}\log n \leq 5 \delta q_{n} \chi_{n}$ for $ n $ sufficiently large, in which case
\[
\frac{\eta'_{n}(\delta)^{2} }{2((1+\chi_{n})\eta'_{n}(\delta)+5 q_{n} \chi_{n})} \geq \frac{10 (1+\delta) (2+\delta) q_{n} \chi_{n}\log n }{10 \delta q_{n} \chi_{n} + 10 q_{n} \chi_{n} } = (2 + \delta) \log n.
\]
Hence, by Lemma~\ref{lemma:Freedman},
\begin{align*}
	\mathbb{P}\Big[ \max_{1\leq i \leq q_{n}} M_{n,k}^{m,j}(i) \geq \eta'_{n}(\delta) \Big]
	& \leq \mathbb{P}\big[ \mbox{$ M_{n,k}^{m,j}(i) \geq \eta'_{n}(\delta) $ and $ \langle M_{n,k}^{m,j} \rangle(i) \leq 5 q_{n} \chi_{n} $ for some $ i \in \mathbb{N} $} \big]\\
	& \leq \exp\Big( -\frac{\eta'_{n}(\delta)^{2} }{2((1+\chi_{n})\eta'_{n}(\delta)+5 q_{n} \chi_{n})} \Big) \\
	& \leq \exp( -(2+\delta) \log n )\\
	& = n^{-(2+\delta)}.
\end{align*}
It follows that
\[
\mathbb{P}\Big[ \max_{1\leq j \leq nT} \max_{1\leq i \leq q_{n}} M_{n,k}^{m,j}(i) \geq \eta'_{n}(\delta) \Big] \leq \sum_{j = 1}^{\lfloor nT \rfloor} \mathbb{P}\Big[ \max_{1\leq i \leq q_{n}} M_{n,k}^{m,j}(i) \geq \eta'_{n}(\delta) \Big] \leq n^{-(1+\delta)}T.
\]
Since $ \sum_{n = 1}^{\infty} n^{-(1+\delta)}T < \infty $, by the Borel--Cantelli lemma, 
\[
\limsup_{n \to \infty} \max_{1\leq j \leq nT} \frac{\max_{1\leq i \leq q_{n}} M_{n,k}^{m,j}(i)}{\eta'_{n}(\delta)} \asl 1 \quad \mbox{for all $ \delta > 0 $.}
\]
Taking $ \delta \rightarrow 0 $, we obtain
\[
\limsup_{n \to \infty} \max_{1\leq j \leq nT} \frac{\max_{1\leq i \leq q_{n}} M_{n,k}^{m,j}(i)}{\sqrt{q_{n}(\chi_{n}\log n \vee 1)}} \asleq \sqrt{20},
\]
and by symmetry, 
\[
\limsup_{n \to \infty} \max_{1\leq j \leq nT} \frac{\max_{1\leq i \leq q_{n}} \big(- M_{n,k}^{m,j}(i)\big)}{\sqrt{q_{n}(\chi_{n}\log n \vee 1)}} \asleq \sqrt{20}.
\]
Combining these two inequalities, we complete the proof of \eqref{eq:Mmnk-increments}.

\section{More Convergence Results}
\label{sec:Preliminary-Results}

In this section, we present additional convergence results to be used in subsequent proofs. Let
\begin{equation}\label{eq:V-nk}
	V_{k}(j) =  \sum_{i = 1}^{j} w_{k}(i) \quad \mbox{for $ j \in \mathbb{N} $,} 
\end{equation}
which is the workload of the first $ j $ customers arriving at station~$ k $. Then,
\begin{equation}\label{eq:B-V}
	B_{n,k}(t) \leq \frac{1}{\mu_{k}}V_{k}(A_{n,k}(t)).
\end{equation}
We define the fluid- and diffusion-scaled workload processes by $ \bar{V}_{n,k}(t) = V_{n,k}(\lfloor nt \rfloor)/n $ and $ \tilde{V}_{n,k}(t) = \sqrt{n} ( \bar{V}_{n,k}(t) -t ) $. By the functional strong law of large numbers, with probability 1,
\begin{equation}\label{eq:FSLLN-V}
	\bar{V}_{n,k} \rightarrow \mathcal{I} \quad \mbox{as $ n \to \infty $,}
\end{equation}
and by Donsker's theorem, 
\begin{equation}\label{eq:FCLT-V}
	\tilde{V}_{n,k} \Rightarrow \tilde{V}_{k} \quad \mbox{ as $ n \to \infty $,}
\end{equation}
where $ \tilde{V}_{k} $ is a driftless Brownian motion starting from $ 0 $ and having variance $ c_{k}^{2} $. 

We rewrite the decomposition \eqref{eq:Rmnk-decomposition} under the fluid scaling as follows:
\begin{equation}\label{eq:Rmnk-fluid}
	\bar{R}^{m}_{n,k}(t) = p_{m}r_{m,k} \bar{E}_{n}(t) + \frac{1}{\sqrt{n}} \big(\tilde{G}^{m}_{n,k}(t) + \tilde{\mathcal{E}}^{m}_{n,k}(t) + \tilde{M}^{m}_{n,k}(t)\big),
\end{equation}
where 
\begin{align*}
	\bar{R}^{m}_{n,k}(t) & = \frac{1}{n} R^{m}_{n,k}(E_{n}(nt)), \quad \tilde{G}^{m}_{n,k}(t) = \frac{1}{\sqrt{n}} G^{m}_{k}(E_{n}(nt)),\\
	\tilde{\mathcal{E}}^{m}_{n,k}(t) & = \frac{1}{\sqrt{n}} \mathcal{E}^{m}_{n,k}(E_{n}(nt)), \quad \tilde{M}^{m}_{n,k}(t) = \frac{1}{\sqrt{n}} M^{m}_{n,k}(E_{n}(nt)).
\end{align*}
The fluid-scaled arrival process at station~$ k $ is defined by $ \bar{A}_{n,k}(t) = A_{n,k}(nt)/n $, which, by \eqref{eq:station-arrivals-analysis}, can be written into
\begin{equation}\label{eq:bar-Ank}
	\bar{A}_{n,k}(t) = \sum_{m = 1}^{b} \bar{R}^{m}_{n,k}\big((t - n^{-1/2}d_{m,k})^{+}\big).
\end{equation}
By \eqref{eq:pm-epsilon-chi}, \eqref{eq:epsilon-1}, \eqref{eq:Gmk-LIL}, and \eqref{eq:Mmnk-fluctuations},
\[
\lim_{n \to \infty}\frac{1}{n} \max_{0 \leq j \leq nT} |G^{m}_{k}(j) + \mathcal{E}^{m}_{n,k}(j) + M^{m}_{n,k}(j)| \aseq 0 \quad \mbox{for $ T > 0 $.}
\]
Then by \eqref{eq:rmk-heavy-traffic}, \eqref{eq:FSLLN}, \eqref{eq:Rmnk-fluid}, and \eqref{eq:bar-Ank}, with probability 1,
\begin{equation}\label{eq:A-bar-convergence}
	(\bar{A}_{n,1}, \ldots, \bar{A}_{n,s}) \rightarrow (\mu_{1}\mathcal{I}, \ldots, \mu_{s}\mathcal{I})  \quad \mbox{as $n \to \infty$.}
\end{equation}
The diffusion-scaled arrival process at station~$k$ is defined by
\begin{equation}\label{eq:tilde-A}
	\tilde{A}_{n,k}(t) = \sqrt{n} \Big(\bar{A}_{n,k}(t) - \rho_{n} \mu \sum_{m = 1}^{b} p_{m}r_{m,k} (t - n^{-1/2}d_{m,k})^{+} \Big).
\end{equation}
It can be written into
\begin{equation}\label{eq:tilde-Ank-decomposition}
	\tilde{A}_{n,k}(t) = \sum_{m = 1}^{b} \Big( p_{m}r_{m,k} \tilde{E}^{m,\Delta}_{n,k}(t) + \tilde{G}^{m,\Delta}_{n,k} (t) + \tilde{\mathcal{E}}^{m,\Delta}_{n,k}(t) + \tilde{M}^{m,\Delta}_{n,k} (t) \Big),
\end{equation}
where
\begin{align*}
	\tilde{E}^{m,\Delta}_{n,k}(t) & = \tilde{E}_{n}\big( (t - n^{-1/2}d_{m,k})^{+}  \big), \quad \tilde{G}^{m,\Delta}_{n,k}(t) = \tilde{G}^{m}_{n,k} \big( (t - n^{-1/2}d_{m,k})^{+} \big),\\
	\tilde{\mathcal{E}}^{m,\Delta}_{n,k}(t) & = \tilde{\mathcal{E}}^{m}_{n,k}\big( (t - n^{-1/2}d_{m,k})^{+} \big), \quad \tilde{M}^{m,\Delta}_{n,k}(t) = \tilde{M}^{m}_{n,k} \big( (t - n^{-1/2}d_{m,k})^{+} \big).
\end{align*}

We define the fluid-scaled busy time process by $ \bar{B}_{n,k}(t) = B_{n,k}(nt)/n $. Clearly,
\begin{equation}\label{eq:busy-inequality}
	\bar{B}_{n,k}(t) \leq t.
\end{equation}
By \eqref{eq:queue-length} and \eqref{eq:diffusion-scaled}, the scaled customer count process satisfies
\begin{equation}\label{eq:diffusion-scaled-queue}
	\tilde{Q}_{n,k}(t) = \sqrt{n}\big(\bar{A}_{n,k}(t) - \bar{S}_{n,k}(\bar{B}_{n,k}(t))\big).
\end{equation}

The following lemma is used for proving Theorems~\ref{theorem:load-balancing} and~\ref{theorem:load-balancing-gap}. It implies the queue lengths under the RJSQ policy cannot be long at the beginning.
\begin{lemma}\label{lemma:after-initial}
	Assume the conditions of Theorem~\ref{theorem:load-balancing} hold. Then for $T_{0} > 0$ and $ k = 1, \ldots, s $, there exist two positive numbers $ C_{k} $ and $ C'_{k} $, both depending on $ T_{0} $, such that with probability~1,
	\[  
	\sup_{0 \leq t \leq T_{0}} \tilde{Q}_{n,k}(n^{-1/2} t) < C_{k}\chi_{n} + C'_{k} n^{-1/4}\sqrt{\log\log n} \quad \mbox{for $ n $ sufficiently large.}
	\]
\end{lemma}

The next three lemmas are used for proving Theorem~\ref{theorem:CRP}. Let $ \tilde{A}_{n}(t) = \sum_{k = 1}^{s} \tilde{A}_{n,k}(t) $. Subject to traveling delays, $\tilde{A}_{n}$ is a randomly perturbed version of $\tilde{E}_{n}$. Because these traveling delays are ``short,'' the two processes are asymptotically close.
\begin{lemma}\label{lemma:A-star}
	Assume the conditions of Theorem~\ref{theorem:load-balancing} hold. Then, $ \sup_{0 \leq t \leq T} \vert \tilde{E}_{n}(t) - \tilde{A}_{n}(t) \vert \Rightarrow 0 $ as $ n \to \infty $ for $T>0$.
\end{lemma}
Using Theorem~\ref{theorem:load-balancing}, we can prove the scaled total customer count process is stochastically bounded.
\begin{lemma}\label{lemma:stochastic-boundedness}
	Assume the conditions of Theorem~\ref{theorem:load-balancing} hold. Then,
	\[  
	\lim_{c \to \infty} \limsup_{n \to \infty} \mathbb{P}\Big[ \sup_{0 \leq t \leq T} \sum_{k = 1}^{s}\tilde{Q}_{n,k}(t) > c \Big] = 0 \quad \mbox{for $ T > 0 $.}
	\]
\end{lemma}
By Lemma~\ref{lemma:stochastic-boundedness}, we can further establish the limit of the fluid-scaled busy time processes, which implies the servers are almost always busy.
\begin{lemma}\label{lemma:busy-time}
	Assume the conditions of Theorem~\ref{theorem:load-balancing} hold. Then, $ (\bar{B}_{n,1}, \ldots, \bar{B}_{n,s}) \Rightarrow (\mathcal{I}, \ldots, \mathcal{I}) $ as $ n \to \infty $.
\end{lemma}

\section{Proofs of Theorem~\ref{theorem:load-balancing}, Theorem~\ref{theorem:load-balancing-gap}, and Corollary~\ref{corollary:optimal-gap}}
\label{sec:Proofs-Load-Balancing}

We prove Theorems~\ref{theorem:load-balancing} and~\ref{theorem:load-balancing-gap} following similar procedures. To avoid unnecessary repetition, we present the common part of the proofs before the specific parts are given. 

Consider the event 
\[
\Theta_{n} = \Big\{ \sup_{0 \leq t \leq T} \max_{k,\ell = 1, \ldots, s} | \tilde{L}_{n,k}(t) - \tilde{L}_{n,\ell}(t) | \geq \phi_{n} \Big\},
\]
where $ \{\phi_{n} : n \in \mathbb{N} \} $ is a sequence of positive numbers. We will prove $ \lim_{n \to \infty} \mathbb{P}[\Theta_{n}] = 0 $ for Theorem~\ref{theorem:load-balancing} and $ \mathbb{P} [ \limsup_{n \to \infty} \Theta_{n} ] = 0 $ for Theorem~\ref{theorem:load-balancing-gap}, with certain choices of $ \{\phi_{n} : n \in \mathbb{N} \} $.

The event $ \Theta_{n} $ can be written as $ \Theta_{n} = \{ \tau_{1,n} \leq T \} $, where
\[
\tau_{1,n} = \inf \Big\{ t \geq 0 : \max_{k,\ell = 1, \ldots, s} | \tilde{L}_{n,k}(t) - \tilde{L}_{n,\ell}(t) | \geq \phi_{n} \Big\}.
\] 
As a convention, we take $ \tau_{1,n} = \infty $ if $ \max_{k,\ell = 1, \ldots, s} | \tilde{L}_{n,k}(t) - \tilde{L}_{n,\ell}(t) | < \phi_{n} $ for all $ t \geq 0 $. Consider an arbitrary outcome in $ \Theta_{n} $. Let stations~$ \bar{\ell}(n) $ and~$ \underline{\ell}(n) $ be the stations having the longest and shortest queue lengths at time $ \tau_{1,n} $, respectively. Then, $ \tilde{L}_{n,\underline{\ell}(n)}(\tau_{1,n}) \leq \tilde{L}_{n,\ell}(\tau_{1,n}) \leq \tilde{L}_{n,\bar{\ell}(n)}(\tau_{1,n}) $ for $ \ell = 1, \ldots, s $, and $ \tilde{L}_{n,\bar{\ell}(n)}(\tau_{1,n}) - \tilde{L}_{n,\underline{\ell}(n)}(\tau_{1,n}) \geq \phi_{n} $. By \eqref{eq:capacity}, \eqref{eq:diffusion-scaled}, and the assumption $\mu_{1} \leq \ldots \leq \mu_{s}$,
\begin{equation}\label{eq:tilde-L-Q}
	\mu \tilde{L}_{n,\bar{\ell}(n)}(\tau_{1,n}) - \sum_{\ell = 1}^{s} \tilde{Q}_{n,\ell}(\tau_{1,n}) \geq \mu_{1} \phi_{n}.
\end{equation}
Write $ \tilde{H}_{n}(t) = \mu \tilde{L}_{n,\bar{\ell}(n)}(t) - \sum_{\ell = 1}^{s} \tilde{Q}_{n,\ell}(t) $ with $\tilde{H}_{n}(0-) = 0$. Consider
\[
\tau_{2,n} = \sup\Big\{ t \in [0, \tau_{1,n}] :  \tilde{H}_{n}(t-) < \frac{\mu_{1}\phi_{n}}{2}  \Big\} \quad \mbox{and} \quad \tau_{3,n} = \inf\big\{ t \in [\tau_{2,n}, \tau_{1,n}] :  \tilde{H}_{n}(t) \geq \mu_{1}\phi_{n}  \big\}.
\] 
Since the system is initially empty, $ 0 \leq \tau_{2,n} \leq \tau_{3,n} \leq \tau_{1,n} $.

Let $ \{ \psi_{n} : n \in \mathbb{N} \} $ be a sequence of positive numbers such that $ \psi_{n} < T $ for all $ n \in \mathbb{N} $. The event $ \Theta_{n} $ can be decomposed into three disjoint events:
\begin{align*}
	\Theta_{1,n} & = \{ \tau_{1,n} \leq T, \tau_{2,n} \leq n^{-1/2} d_{0} \},\\
	\Theta_{2,n} & = \{ \tau_{1,n} \leq T, \tau_{2,n} > n^{-1/2} d_{0}, \tau_{3,n} - \tau_{2,n} < \psi_{n} \},\\
	\Theta_{3,n} & = \{ \tau_{1,n} \leq T, \tau_{2,n} > n^{-1/2} d_{0}, \tau_{3,n} - \tau_{2,n} \geq \psi_{n} \},
\end{align*}
where $ d_{0} = \max \{ d_{m,k} : m = 1,\ldots, b; k = 1, \ldots, s \} $ is the largest value of a primitive delay. We will prove these events are asymptotically negligible by finding an appropriate sequence $ \{ \psi_{n} : n \in \mathbb{N} \} $ for each theorem. To this end, the increments of each $ \tilde{Q}_{n,\ell} $ need to be analyzed in detail. 

By \eqref{eq:tilde-A}, \eqref{eq:tilde-Ank-decomposition}, and \eqref{eq:diffusion-scaled-queue}, $ \tilde{Q}_{n,\ell}(t) = \sum_{i = 1}^{4}\tilde{Q}^{(i)}_{n,\ell} (t) $ where
\begin{align*}
	\tilde{Q}^{(1)}_{n,\ell} (t) & = \sum_{m = 1}^{b} \tilde{\mathcal{E}}^{m,\Delta}_{n,\ell}(t), \quad \tilde{Q}^{(2)}_{n,\ell} (t) = - \sqrt{n} \mu_{\ell}\bar{B}_{n,\ell}(t), \\
	\tilde{Q}^{(3)}_{n,\ell} (t) & = \sum_{m = 1}^{b} \big( p_{m} r_{m,\ell} \tilde{E}^{m,\Delta}_{n,\ell}(t) + \tilde{G}^{m,\Delta}_{n,\ell} (t) + \tilde{M}^{m,\Delta}_{n,\ell} (t) \big) - \tilde{S}_{n,\ell}\big( \bar{B}_{n,\ell}(t)\big),\\
	\tilde{Q}^{(4)}_{n,\ell} (t) & = \sqrt{n} \rho_{n} \mu \sum_{m = 1}^{b} p_{m}r_{m,\ell} (t - n^{-1/2}d_{m,\ell})^{+} .
\end{align*}
It follows that $ \tilde{H}_{n}(t) = \sum_{i = 1}^{4} \tilde{H}^{(i)}_{n}(t) $ where $ \tilde{H}^{(i)}_{n}(t) = \mu\tilde{Q}^{(i)}_{n,\bar{\ell}(n)}(t)/\mu_{\bar{\ell}(n)} - \sum_{\ell = 1}^{s} \tilde{Q}^{(i)}_{n,\ell}(t)$.

The sequence of events $ \{ \Theta_{1,n} : n \in \mathbb{N} \} $ satisfies
\begin{align*} 
	\mathbb{P}\Big[ \limsup_{n \to \infty} \Theta_{1,n} \Big] & \leq \mathbb{P}\Big[\limsup_{n \to \infty} \{\tau_{2,n} \leq n^{-1/2} d_{0}\}\Big] \\
	& = \mathbb{P} \Big[ \limsup_{n \to \infty} \Big\{\sup_{0 \leq t \leq d_{0}} \tilde{H}_{n} (n^{-1/2}t) \geq \frac{\mu_{1}\phi_{n}}{2} \Big\} \Big] \\
	& \leq \mathbb{P} \Big[ \limsup_{n \to \infty} \Big\{\sup_{0 \leq t \leq d_{0}} \big(\mu \tilde{L}_{n,\bar{\ell}(n)} (n^{-1/2}t) - \tilde{Q}_{n,\bar{\ell}(n)} (n^{-1/2}t)\big) \geq \frac{\mu_{1}\phi_{n}}{2} \Big\} \Big] \\
	& \leq \mathbb{P} \Big[ \limsup_{n \to \infty} \Big\{\sup_{0 \leq t \leq d_{0}} \max_{k = 1, \ldots, s} (\mu - \mu_{\bar{\ell}(n)})\tilde{L}_{n,k} (n^{-1/2}t) \geq \frac{\mu_{1}\phi_{n}}{2} \Big\}\Big] \\
	& \leq \mathbb{P} \Big[ \limsup_{n \to \infty} \Big\{\sup_{0 \leq t \leq d_{0}} \max_{k = 1, \ldots, s} \tilde{L}_{n,k} (n^{-1/2}t) \geq \frac{\mu_{1}\phi_{n}}{2(\mu - \mu_{1})} \Big\}\Big],
\end{align*}
where we obtain the second inequality by \eqref{eq:tilde-L-Q}, the third inequality by \eqref{eq:diffusion-scaled}, and the fourth inequality by the assumption $\mu_{1} \leq \mu_{\bar{\ell}(n)}$. By Lemma~\ref{lemma:after-initial}, $ \mathbb{P}[\limsup_{n \to \infty} \Theta_{1,n}] = 0 $ if 
\begin{equation}\label{eq:phi-n-condition-1}
	\phi_{n} \geq \max_{k = 1, \ldots, s}  \frac{2(\mu - \mu_{1})}{\mu_{1}\mu_{k}} \big(C_{k}\chi_{n} + C'_{k} n^{-1/4}\sqrt{\log\log n}\big) \quad \mbox{for $ n $ sufficiently large,}
\end{equation}
where $ C_{k} $ and $ C'_{k} $ are two positive numbers depending on $ d_{0} $.

Consider $ \Theta_{2,n} $. Since $ n^{-1/2} d_{0} < \tau_{2,n} \leq \tau_{3,n} $, $ \tilde{Q}^{(4)}_{n,\ell} [\tau_{2,n}, \tau_{3,n}] = \sqrt{n} \rho_{n} \mu_{\ell} (\tau_{3,n} - \tau_{2,n}) $ by \eqref{eq:rmk-heavy-traffic}, and thus $ \tilde{H}^{(4)}_{n}[\tau_{2,n},\tau_{3,n}] = 0 $. Because $ \tilde{Q}_{n,\bar{\ell}(n)}(t) > 0 $ for $ t \in [\tau_{2,n}, \tau_{3,n}] $, the server in station~$ \bar{\ell}(n) $ must be busy and $ \bar{B}_{n,\bar{\ell}(n)}[\tau_{2,n},\tau_{3,n}] = \tau_{3,n} - \tau_{2,n} $. Since $\bar{B}_{n,\ell} [\tau_{2,n},\tau_{3,n}] \leq \tau_{3,n} - \tau_{2,n} $ for $\ell \neq \bar{\ell}(n)$, $ \tilde{H}^{(2)}_{n}[\tau_{2,n},\tau_{3,n}] \leq 0$. By \eqref{eq:pm-epsilon-chi}, $ |\tilde{Q}^{(1)}_{n,\ell} [\tau_{2,n}, \tau_{3,n}]| \leq b \mu \sqrt{n}\rho_{n}\chi_{n}(\tau_{3,n} - \tau_{2,n}) + \sum_{m = 1}^{b} \chi_{n} \tilde{E}^{m,\Delta}_{n,\ell}[\tau_{2,n}, \tau_{3,n}] $. Then by \eqref{eq:tilde-L-Q},
\begin{equation}\label{eq:bound-Theta-2}
	\tilde{H}_{n}[\tau_{2,n}, \tau_{3,n}] \leq b \mu\Big(\frac{\mu}{\mu_{\bar{\ell}(n)}} + s - 2\Big) \sqrt{n}\rho_{n}\chi_{n}\psi_{n} + \tilde{H}^{(3)}_{n}[\tau_{2,n}, \tau_{3,n}] + \tilde{H}^{(5)}_{n}[\tau_{2,n}, \tau_{3,n}],
\end{equation}
where 
\[
\tilde{H}^{(5)}_{n} [\tau_{2,n}, \tau_{3,n}] = \Big(\frac{\mu}{\mu_{\bar{\ell}(n)}} - 1\Big) \sum_{m = 1}^{b} \chi_{n}\tilde{E}^{m,\Delta}_{n,\bar{\ell}(n)}[\tau_{2,n}, \tau_{3,n}] + \sum_{\ell \neq \bar{\ell}(n)} \sum_{m = 1}^{b}\chi_{n} \tilde{E}^{m,\Delta}_{n,\ell}[\tau_{2,n}, \tau_{3,n}].
\]
If $ \psi_{n} $ is given, we can choose $ \phi_{n} $ to make the right side of \eqref{eq:bound-Theta-2} less than $ \mu_{1}\phi_{n}/2 $ when $ n $ is large. Then, $ \Theta_{2,n} $ will be asymptotically negligible, since $ \tilde{H}_{n} [\tau_{2,n}, \tau_{3,n}] \geq \mu_{1}\phi_{n}/2 $ on  $ \Theta_{2,n} $. We will determine $ \psi_{n} $ by analyzing $ \Theta_{3,n} $ as follows.

Let us assume
\begin{equation}\label{eq:psi-condition-1}
	\psi_{n} \geq n^{-1/2} (2d_{0} \vee 1) \quad \mbox{for $ n $ sufficiently large.}
\end{equation}
Write $ \tau_{4,n} = \tau_{3,n} -  \psi_{n}/2 - n^{-1/2} d_{0} $. Then, $ \tau_{2,n} \leq \tau_{4,n} < \tau_{3,n} $ and $ \tilde{H}_{n} [\tau_{4,n}, \tau_{3,n}] > 0 $ on $ \Theta_{3,n} $. By the previous arguments, $ \tilde{H}^{(4)}_{n} [\tau_{4,n}, \tau_{3,n}] = 0 $ and $ \tilde{H}^{(2)}_{n} [\tau_{4,n}, \tau_{3,n}] \leq 0 $. Put $\tilde{\mathcal{E}}_{n,\ell}(t) = \sum_{j = 1}^{E_{n}(nt)} \varepsilon_{n,\pi_{n,\ell}(j)}/\sqrt{n} $. Then, $\sum_{\ell = 1}^{s} \tilde{\mathcal{E}}_{n,\ell}(t) = 0$ because $\sum_{\ell = 1}^{s} \varepsilon_{n,\ell} = 0$. By \eqref{eq:perturbation-parameters} and \eqref{eq:pm-epsilon-chi}, 
\begin{align*}
	& | \tilde{Q}^{(1)}_{n,\ell}[\tau_{4,n}, \tau_{3,n}] -  \tilde{\mathcal{E}}_{n,\ell}[ \tau_{4,n}, \tau_{3,n}  - n^{-1/2}d_{0}] | \\
	& \quad \leq b \sqrt{n}\chi_{n} \big(\bar{E}_{n}[\tau_{4,n} - n^{-1/2}d_{0}, \tau_{4,n}] + \bar{E}_{n}[\tau_{3,n} - n^{-1/2}d_{0}, \tau_{3,n}] \big).
\end{align*}
Then, $ \tilde{H}^{(1)}_{n} [\tau_{4,n}, \tau_{3,n}] \leq \tilde{H}^{(6)}_{n} [\tau_{4,n}, \tau_{3,n}] + \tilde{H}^{(7)}_{n} [\tau_{4,n}, \tau_{3,n}] + \tilde{H}^{(8)}_{n} [\tau_{4,n}, \tau_{3,n}] $, where 
\begin{align*}  
	\tilde{H}^{(6)}_{n} [\tau_{4,n}, \tau_{3,n}] & = 2 b \mu d_{0} \Big( \frac{\mu}{\mu_{\bar{\ell}(n)}} + s - 2\Big) \rho_{n} \chi_{n}  ,\\
	\tilde{H}^{(7)}_{n} [\tau_{4,n}, \tau_{3,n}] & =  b \Big( \frac{\mu}{\mu_{\bar{\ell}(n)}} + s - 2 \Big)\chi_{n} \big(\tilde{E}_{n}[\tau_{4,n} - n^{-1/2}d_{0}, \tau_{4,n}] + \tilde{E}_{n}[\tau_{3,n} - n^{-1/2}d_{0}, \tau_{3,n}] \big),\\
	\tilde{H}^{(8)}_{n} [\tau_{4,n}, \tau_{3,n}] & = \frac{\mu}{\mu_{\bar{\ell}(n)}} \tilde{\mathcal{E}}_{n,\bar{\ell}(n)}[\tau_{4,n}, \tau_{3,n}  - n^{-1/2}d_{0}].
\end{align*}
Since $ \tilde{H}_{n}(t) \geq \mu_{1}\phi_{n}/2 $ for $ t \in [\tau_{2,n}, \tau_{3,n}] $, station~$\bar{\ell}(n)$ cannot have the shortest queue. If the $ j $th customer joins the system during this time interval, $ \varepsilon_{n,\pi_{n,\bar{\ell}(n)}(j)} \leq - \delta_{0} \chi_{n}$ by  condition~\eqref{item:cond-1-n}. Hence, 
\[
\tilde{H}^{(8)}_{n} [\tau_{4,n}, \tau_{3,n}] \leq - \frac{\mu}{\mu_{\bar{\ell}(n)}}\delta_{0} \sqrt{n} \chi_{n} \bar{E}_{n}[\tau_{4,n}, \tau_{3,n}  - n^{-1/2}d_{0}] =  \tilde{H}^{(9)}_{n} [\tau_{4,n}, \tau_{3,n}] - \frac{\delta_{0} \mu^{2}}{2\mu_{\bar{\ell}(n)}} \sqrt{n}\rho_{n}\chi_{n}\psi_{n},
\]
where $ \tilde{H}^{(9)}_{n} [\tau_{4,n}, \tau_{3,n}] = - \mu\delta_{0}\chi_{n} \tilde{E}_{n}[\tau_{4,n}, \tau_{3,n}- n^{-1/2}d_{0}] /\mu_{\bar{\ell}(n)}$. Combing the above results, we obtain
\begin{equation}\label{eq:bound-Theta-3}
	\tilde{H}_{n} [\tau_{4,n}, \tau_{3,n}] \leq - \frac{\delta_{0} \mu^{2}}{2\mu_{\bar{\ell}(n)}} \sqrt{n}\rho_{n}\chi_{n}\psi_{n} + \sum_{i = 3,6,7,9}\tilde{H}^{(i)}_{n} [\tau_{4,n}, \tau_{3,n}].
\end{equation}
We choose $ \psi_{n} $ so that the first term will dominate the right side of \eqref{eq:bound-Theta-3} as $ n $ gets large; then, $ \tilde{H}_{n} [\tau_{4,n}, \tau_{3,n}] \leq 0 $. Because $ \tilde{H}_{n} [\tau_{4,n}, \tau_{3,n}] > 0 $ on $ \Theta_{3,n} $, this event must be asymptotically negligible. 

\begin{proof}[Proof of Theorem~\ref{theorem:load-balancing}.]
	Take $ \phi_{n} = \delta > 0 $ for all $ n \in \mathbb{N} $. By \eqref{eq:epsilon-1}, condition~\eqref{eq:phi-n-condition-1} holds and thus
	\[
	\lim_{n \to \infty} \mathbb{P} [ \Theta_{1,n} ] \leq \lim_{n \to \infty} \mathbb{P} \Big[ \bigcup_{i = n}^{\infty} \Theta_{1,i} \Big] = \mathbb{P}\Big[ \limsup_{n \to \infty} \Theta_{1,n} \Big] = 0,
	\]
	where the first equality follows from the continuity of probability measures. Take
	\[
	\psi_{n} = \frac{\delta \mu_{1}^{2}}{4 b \mu (\mu + (s - 2)\mu_{1}) \sqrt{n}\rho_{n}\chi_{n} }.
	\]
	Then, $ \lim_{n \to \infty} \psi_{n} = 0 $ by \eqref{eq:heavy-traffic} and \eqref{eq:epsilon-2}. Since $\mu_{1} \leq \mu_{\bar{\ell}(n)}$, $ \tilde{H}_{n} [\tau_{2,n}, \tau_{3,n}]  < \delta\mu_{1}/4 + \tilde{H}^{(3)}_{n} [\tau_{2,n}, \tau_{3,n}] + \tilde{H}^{(5)}_{n} [\tau_{2,n}, \tau_{3,n}] $ by \eqref{eq:bound-Theta-2}. If we can prove 
	\begin{equation}\label{eq:H3H5}
		\lim_{n \to \infty} \mathbb{P}\Big[ \tilde{H}^{(3)}_{n} [\tau_{2,n}, \tau_{3,n}] + \tilde{H}^{(5)}_{n} [\tau_{2,n}, \tau_{3,n}]  \geq \frac{\delta\mu_{1}}{4}, \tau_{1,n} \leq T,  \tau_{3,n} - \tau_{2,n} < \psi_{n} \Big] = 0,
	\end{equation}
	it follows that $ \lim_{n \to \infty} \mathbb{P}[ \tilde{H}_{n} [\tau_{2,n}, \tau_{3,n}] \geq \delta\mu_{1}/2, \tau_{1,n} \leq T, \tau_{3,n} - \tau_{2,n} < \psi_{n} ] = 0 $, by which we will obtain $ \lim_{n \to \infty} \mathbb{P} [ \Theta_{2,n} ] = 0 $. 
	
	By \eqref{eq:Gmk-strong-approximation}, \eqref{eq:FSLLN}, and the random time-change theorem (Theorem~5.3 in \citealp{Chen.Yao.2001}), $ \tilde{G}^{m}_{n,k} \Rightarrow \tilde{G}^{m}_{k} $ as $ n \to \infty $, where $ \tilde{G}^{m}_{k} $ is a driftless Brownian motion starting from $ 0 $ with variance $ p_{m}r_{m,k}(1 - p_{m}r_{m,k}) \mu $. It follows from Lemma~3.2 in \citet{Iglehart.Whitt.1970a} that $ \{ \tilde{G}^{m}_{n,k} : n \in \mathbb{N} \} $ is C-tight. Because $ \lim_{n \to \infty} \psi_{n} = 0 $, $ \sup_{0 \leq t_{2} \leq t_{1} \leq T, t_{1} - t_{2} \leq \psi_{n}}  |\tilde{G}^{m}_{n,k} [t_{2}, t_{1}]| \Rightarrow 0 $ as $ n \to \infty $, by which we obtain $ \tilde{G}^{m,\Delta}_{n,k} [\tau_{2,n}, \tau_{3,n}] \mathbb{1}_{\{ \tau_{1,n} \leq T, \tau_{3,n} - \tau_{2,n} < \psi_{n} \}} \Rightarrow 0 $. Following this C-tightness argument, we may use \eqref{eq:FCLT} to obtain $ \tilde{E}^{m,\Delta}_{n,k} [\tau_{2,n}, \tau_{3,n}] \mathbb{1}_{\{ \tau_{1,n} \leq T, \tau_{3,n} - \tau_{2,n} < \psi_{n} \}} \Rightarrow 0 $, use the last assertion of Lemma~\ref{lemma:martingale} to obtain $ \tilde{M}^{m,\Delta}_{n,k} [\tau_{2,n}, \tau_{3,n}] \mathbb{1}_{\{ \tau_{1,n} \leq T, \tau_{3,n} - \tau_{2,n} < \psi_{n} \}} \Rightarrow 0 $, and use \eqref{eq:FCLT} and \eqref{eq:busy-inequality} to obtain $ \tilde{S}_{n,k} [ \bar{B}_{n,k}(\tau_{2,n}), \bar{B}_{n,k}(\tau_{3,n}) ] \mathbb{1}_{\{ \tau_{1,n} \leq T, \tau_{3,n} - \tau_{2,n} < \psi_{n} \}} \Rightarrow 0 $. These results lead to $(\tilde{H}^{(3)}_{n} [\tau_{2,n}, \tau_{3,n}] + \tilde{H}^{(5)}_{n} [\tau_{2,n}, \tau_{3,n}]) \mathbb{1}_{\{ \tau_{1,n} \leq T, \tau_{3,n} - \tau_{2,n} < \psi_{n} \}} \Rightarrow 0 $, which proves \eqref{eq:H3H5}.
	
	Now consider $ \Theta_{3,n} $. By \eqref{eq:epsilon-1}, condition \eqref{eq:psi-condition-1} holds for $ n $ sufficiently large, and thus $ \tau_{3,n} - \tau_{4,n} \leq \psi_{n} $. Since $\mu_{\bar{\ell}(n)} \leq \mu_{s}$, it follows from \eqref{eq:bound-Theta-3} that
	\[
	\tilde{H}_{n} [\tau_{4,n}, \tau_{3,n}] \leq -  \frac{\delta \delta_{0} \mu \mu_{1}^{2}}{8 b \mu_{s}(\mu + (s-2) \mu_{1}) } + \sum_{i = 3,6,7,9}\tilde{H}^{(i)}_{n} [\tau_{4,n}, \tau_{3,n}].
	\]
	Using the C-tightness argument, we obtain $ \tilde{H}^{(3)}_{n} [\tau_{4,n}, \tau_{3,n}] \mathbb{1}_{\{ \tau_{1,n} \leq T, \tau_{3,n} - \tau_{4,n} \leq \psi_{n} \}} \Rightarrow 0 $, and by \eqref{eq:heavy-traffic} and \eqref{eq:epsilon-1}, $ (\tilde{H}^{(6)}_{n} [\tau_{4,n}, \tau_{3,n}] + \tilde{H}^{(7)}_{n} [\tau_{4,n}, \tau_{3,n}] + \tilde{H}^{(9)}_{n} [\tau_{4,n}, \tau_{3,n}] ) \mathbb{1}_{\{ \tau_{1,n} \leq T, \tau_{3,n} - \tau_{4,n} \leq \psi_{n} \}} \Rightarrow 0 $. Therefore, 
	\[
	\lim_{n \to \infty} \mathbb{P}\Big[ \sum_{i = 3,6,7,9}\tilde{H}^{(i)}_{n} [\tau_{4,n}, \tau_{3,n}] \geq \frac{\delta \delta_{0} \mu \mu_{1}^{2}}{16 b \mu_{s} (\mu + (s-2) \mu_{1}) }, \tau_{1,n} \leq T, \tau_{3,n} - \tau_{4,n} \leq \psi_{n} \Big] = 0,
	\]
	by which we obtain $ \lim_{n \to \infty} \mathbb{P}\big[ \tilde{H}_{n} [\tau_{4,n}, \tau_{3,n}] > 0, \tau_{1,n} \leq T, \tau_{3,n} - \tau_{4,n} \leq \psi_{n} \big] = 0$. Then, we conclude the proof with $ \lim_{n \to \infty} \mathbb{P} [ \Theta_{3,n} ] = 0 $. 
\end{proof}

\begin{proof}[Proof of Theorem~\ref{theorem:load-balancing-gap}.]
	Assume \eqref{eq:phi-n-condition-1} holds. Then, $ \mathbb{P}[\limsup_{n \to \infty} \Theta_{1,n}] = 0 $.
	
	Without loss of generality, we assume the sequence $ \{ \psi_{n} : n \in \mathbb{N} \} $ is nonincreasing. By \eqref{eq:psi-condition-1} and the assumption $ \psi_{n} < T $, $ \limsup_{n \to \infty} \sqrt{n^{1/2} \psi_{n} (\log \log n - \log \psi_{n})} = \infty $. By Lemmas~\ref{lemma:renewal-strong-approximations} and \ref{lemma:Wiener-intervals},
	\begin{align}
		\limsup_{n \to \infty} \sup_{0 \leq t \leq nT } \sup_{0 \leq u \leq n \psi_{n}} \frac{|E_{n}(t, t+u] - \lambda_{n} u |}{\sqrt{n \psi_{n} (\log \log n - \log \psi_{n})}} & \aseq c_{0}\sqrt{2\mu}, \label{eq:E-increments}\\
		\limsup_{n \to \infty} \sup_{0 \leq t \leq nT } \sup_{0 \leq u \leq n \psi_{n}} \frac{|S_{k}(t, t+u] - \mu_{k} u |}{\sqrt{n \psi_{n} (\log \log n - \log \psi_{n})}} & \aseq c_{k}\sqrt{2\mu_{k}} \label{eq:S-increments}
	\end{align}
	for $ k = 1, \ldots, s $. By \eqref{eq:Gmk-strong-approximation}, \eqref{eq:FSLLN}, and Lemma~\ref{lemma:Wiener-intervals},
	\[
	\limsup_{n \to \infty} \sup_{0 \leq t \leq nT } \sup_{0 \leq u \leq n \psi_{n}} \frac{|G^{m}_{k}(E_{n}(t+u)) - G^{m}_{k}(E_{n}(t)) |}{\sqrt{n \psi_{n} (\log \log n - \log \psi_{n})}} \aseq \sqrt{2 \mu p_{m}r_{m,k}(1 - p_{m}r_{m,k})} \leq \sqrt{0.5\mu},
	\]
	and by \eqref{eq:Mmnk-increments} and \eqref{eq:FSLLN},
	\[
	\limsup_{n \to \infty} \sup_{0 \leq t \leq nT } \sup_{0 \leq u \leq n \psi_{n}} \frac{|M^{m}_{n,k}(E_{n}(t+u)) - M^{m}_{n,k}(E_{n}(t)) |}{\sqrt{n \psi_{n} (\chi_{n} \log n \vee 1)}} \asleq \sqrt{20},
	\]
	for $ m = 1, \ldots, b $ and $ k = 1, \ldots, s $. These results will be used to analyze $ \Theta_{2,n} $ and $ \Theta_{3,n} $. 
	
	Take 
	\[
	\bar{C} > \max_{k = 1, \ldots, s} \Big(\frac{\mu}{\mu_{1}} + s - 2\Big)\Big( c_{0}\mu_{k} \sqrt{2\mu^{-1}} +  c_{k}\sqrt{2\mu_{k}} + b\sqrt{0.5\mu} \Big) \quad\mbox{and}\quad \bar{C}' > \sqrt{20}b\Big( \frac{\mu}{\mu_{1}} + s - 2\Big).
	\] 
	Then by \eqref{eq:rmk-heavy-traffic}, with probability 1, $ |\tilde{H}^{(3)}_{n} [\tau_{2,n}, \tau_{3,n}] \mathbb{1}_{\{ \tau_{1,n} \leq T, \tau_{3,n} - \tau_{2,n} < \psi_{n} \}} | < \bar{C} \sqrt{\psi_{n} (\log \log n - \log \psi_{n})} + \bar{C}' \sqrt{\psi_{n} (\chi_{n} \log n \vee 1)} $ for $ n $ sufficiently large. Moreover,
	\[  
	\limsup_{n \to \infty} \frac{|\tilde{H}^{(5)}_{n} [\tau_{2,n}, \tau_{3,n}]\mathbb{1}_{\{ \tau_{1,n} \leq T, \tau_{3,n} - \tau_{2,n} < \psi_{n} \}}|}{\chi_{n}\sqrt{\psi_{n} (\log \log n - \log \psi_{n})}} \asleq c_{0} b \sqrt{2 \mu }  \Big(\frac{\mu}{\mu_{1}} + s - 2\Big).
	\]
	By \eqref{eq:epsilon-1} and \eqref{eq:bound-Theta-2}, with probability 1,
	\begin{equation}\label{eq:psi-Theta-2}
		\begin{aligned}
			\tilde{H}_{n} [\tau_{2,n}, \tau_{3,n}] \mathbb{1}_{\{ \tau_{1,n} \leq T, \tau_{3,n} - \tau_{2,n} < \psi_{n} \}} & < b \mu\Big(\frac{\mu}{\mu_{1}}+s-2\Big) \sqrt{n}\rho_{n}\chi_{n}\psi_{n} \\
			& \quad + \bar{C} \sqrt{\psi_{n} (\log \log n - \log \psi_{n})} + \bar{C}' \sqrt{\psi_{n} (\chi_{n} \log n \vee 1)}
		\end{aligned}
	\end{equation}
	for $ n $ sufficiently large.
	
	Similarly, with probability 1, $ |\tilde{H}^{(3)}_{n} [\tau_{4,n}, \tau_{3,n}] \mathbb{1}_{\{ \tau_{1,n} \leq T, \tau_{3,n} - \tau_{4,n} \leq \psi_{n} \}} | < \bar{C} \sqrt{\psi_{n} (\log \log n - \log \psi_{n})} + \bar{C}' \sqrt{\psi_{n} (\chi_{n} \log n \vee 1)} $ for $ n $ sufficiently large. By Lemmas~\ref{lemma:renewal-strong-approximations} and~\ref{lemma:Wiener-intervals},
	\begin{align*}
		\limsup_{n \to \infty} \frac{|\tilde{H}^{(7)}_{n} [\tau_{4,n}, \tau_{3,n}]\mathbb{1}_{\{ \tau_{1,n} \leq T, \tau_{3,n} - \tau_{4,n} \leq \psi_{n} \}}|}{n^{-1/4}\chi_{n}\sqrt{\log n}} & \asleq 4 c_{0} b \sqrt{\mu d_{0}} \Big(\frac{\mu}{\mu_{1}}+s-2\Big), \\
		\limsup_{n \to \infty} \frac{|\tilde{H}^{(9)}_{n} [\tau_{4,n}, \tau_{3,n}]\mathbb{1}_{\{ \tau_{1,n} \leq T, \tau_{3,n} - \tau_{4,n} \leq \psi_{n} \}}|}{\chi_{n}\sqrt{\psi_{n}(\log \log n - \log \psi_{n}})} & \asleq  \frac{2 \delta_{0} c_{0}\mu^{3/2}}{\mu_{1}}.
	\end{align*}
	Note that 
	\[
	\lim_{n \to \infty} \frac{n^{-1/4}\chi_{n}\sqrt{\log n}}{\sqrt{\psi_{n} (\chi_{n} \log n \vee 1)}} =  \lim_{n \to \infty} \frac{\sqrt{\chi_{n}}}{\sqrt{n^{1/2}\psi_{n}}} \cdot \frac{\sqrt{\chi_{n}\log n}}{\sqrt{ \chi_{n} \log n \vee 1}} \leq \lim_{n \to \infty}  \frac{\sqrt{\chi_{n}}}{\sqrt{n^{1/2}\psi_{n}}} = 0,
	\]
	where the last equality follows from \eqref{eq:epsilon-1} and \eqref{eq:psi-condition-1}. By \eqref{eq:epsilon-1} and \eqref{eq:bound-Theta-3}, with probability 1,
	\begin{equation}\label{eq:Theta-3-choice}
		\begin{aligned}
			\tilde{H}_{n} [\tau_{4,n}, \tau_{3,n}] \mathbb{1}_{\{ \tau_{1,n} \leq T, \tau_{3,n} - \tau_{4,n} \leq \psi_{n} \}} & < - \frac{\delta_{0} \mu^{2}}{2\mu_{s}} \sqrt{n}\rho_{n}\chi_{n}\psi_{n} + 2 b \mu d_{0}\Big(\frac{\mu}{\mu_{1}} + s - 2\Big) \rho_{n} \chi_{n}\\
			& \quad + \bar{C} \sqrt{\psi_{n} (\log \log n - \log \psi_{n})} + \bar{C}' \sqrt{\psi_{n} (\chi_{n} \log n \vee 1)}
		\end{aligned}
	\end{equation}
	for $ n $ sufficiently large.
	
	We choose $ \psi_{n} $ in order for the first term to dominate the right side of \eqref{eq:Theta-3-choice}. To this end, we set
	\begin{equation}\label{eq:psi-Theta-3}
		\begin{aligned}
			\frac{\delta_{0}\mu^{2}}{6\mu_{s}}  \sqrt{n}\rho_{n}\chi_{n}\psi_{n} \geq \max & \Big\{ 2 b \mu d_{0}\Big(\frac{\mu}{\mu_{1}} + s - 2\Big) \rho_{n} \chi_{n},\\
			& \quad \bar{C} \sqrt{\psi_{n} (\log \log n - \log \psi_{n})}, \bar{C}'\sqrt{\psi_{n} (\chi_{n} \log n \vee 1)} \Big\}.
		\end{aligned}
	\end{equation}
	Solving the inequality with the first term on the right side, we obtain
	\begin{equation}\label{eq:psi-bound-1}
		\psi_{n} \geq \Big(\frac{\mu}{\mu_{1}} + s - 2\Big) \frac{12 b d_{0}\mu_{s}}{\delta_{0} \mu\sqrt{n}}.
	\end{equation}
	The inequality involving the second term can be reduced to
	\begin{equation}\label{eq:psi-bound-2}
		\psi_{n} \geq h_{n} (\log \log n - \log \psi_{n}),
	\end{equation}
	where $ h_{n} = 36 \bar{C}^{2}\mu_{s}^{2}/(\delta_{0}^{2} \mu^{4} n \rho_{n}^{2} \chi_{n}^{2}) $. Take
	\begin{equation}\label{eq:psi-bound-2-1}
		\psi_{n} \geq h_{n} (\log \log n - \log h_{n}).
	\end{equation}
	Since $ \lim_{n \to \infty} h_{n} = 0 $ by \eqref{eq:heavy-traffic} and \eqref{eq:log-log-n-bound}, $ h_{n} (\log \log n - \log \psi_{n}) \leq h_{n} (\log \log n - \log h_{n} - \log (\log \log n - \log h_{n})) < h_{n} (\log \log n - \log h_{n} ) $ for $ n $ sufficiently large, and thus inequality \eqref{eq:psi-bound-2} holds. Solving the inequality with the third term on the right side of \eqref{eq:psi-Theta-3}, we obtain
	\begin{equation}\label{eq:psi-bound-3}
		\psi_{n} \geq \frac{\bar{C}'^{2}}{\bar{C}^{2}}h_{n}(\chi_{n} \log n \vee 1).
	\end{equation}
	Note that when $ \chi_{n} \log n \geq 1 $,
	\[
	\frac{\chi_{n}\log n}{\log (n \chi_{n}^{2})} \leq \frac{\chi_{n}\log n}{\log (n (\log n)^{-2})} = \frac{\chi_{n}\log n}{\log n - 2 \log \log n},
	\]
	and then by \eqref{eq:epsilon-1} and \eqref{eq:log-log-n-bound}, $ \chi_{n}\log n \vee 1 = o(\log (n \chi_{n}^{2}))$. Hence, $ \chi_{n}\log n \vee 1 = o(- \log h_{n})$, and it follows that for $ n $ sufficiently large, inequality \eqref{eq:psi-bound-3} must hold if \eqref{eq:psi-bound-2-1} holds. If we take
	\begin{equation}\label{eq:psi-bound}
		\begin{aligned}
			\psi_{n} \geq \max & \Big\{ \frac{2d_{0}\vee 1}{\sqrt{n}},\Big(\frac{\mu}{\mu_{1}} + s - 2\Big) \frac{12 b d_{0}\mu_{s}}{\delta_{0}\mu \sqrt{n}}, \\
			& \quad \frac{36 \bar{C}^{2}\mu_{s}^{2}}{\delta_{0}^{2} \mu^{4} } \cdot \frac{\log \log n + \log (n \rho_{n}^{2} \chi_{n}^{2}) + \log (\delta_{0}^{2} \mu^{4}/(36 \bar{C}^{2} \mu_{s}^{2})) }{n \rho_{n}^{2} \chi_{n}^{2}}  \Big\}
		\end{aligned}
	\end{equation} 
	by \eqref{eq:psi-condition-1}, \eqref{eq:psi-bound-1}, and \eqref{eq:psi-bound-2-1}, then with probability 1, $ \tilde{H}_{n} [\tau_{4,n}, \tau_{3,n}] \mathbb{1}_{\{ \tau_{1,n} \leq T, \tau_{3,n} - \tau_{4,n} \leq \psi_{n} \}} < 0 $ for $ n $ sufficiently large. Therefore, $ \mathbb{P}[\limsup_{n \to \infty} \Theta_{3,n}] = 0 $ under this condition.
	
	Now we choose $ \phi_{n} $ to make $ \mathbb{P}[\limsup_{n \to \infty} \Theta_{2,n}] = 0 $. Assume $ \psi_{n} $ satisfies \eqref{eq:psi-bound}. By \eqref{eq:psi-Theta-2} and \eqref{eq:psi-Theta-3}, with probability 1,
	\[
	\tilde{H}_{n} [\tau_{2,n}, \tau_{3,n}] \mathbb{1}_{\{ \tau_{1,n} \leq T, \tau_{3,n} - \tau_{2,n} < \psi_{n} \}} < b \mu\Big(\frac{\mu}{\mu_{1}} +s-2+ \frac{\delta_{0}\mu}{3b\mu_{s}}\Big) \sqrt{n}\rho_{n}\chi_{n}\psi_{n} \quad \mbox{for $ n $ sufficiently large.}
	\]
	Since $ \tilde{H}_{n} [\tau_{2,n}, \tau_{3,n}] \geq \mu_{1}\phi_{n}/2 $ on  $ \Theta_{2,n} $, $ \mathbb{P}[\limsup_{n \to \infty} \Theta_{2,n}] = 0 $ holds if 
	\[
	\phi_{n} \geq \frac{2b \mu}{\mu_{1}}\Big(\frac{\mu}{\mu_{1}} + s - 2 + \frac{\delta_{0}\mu}{3b\mu_{s}}\Big) \sqrt{n}\rho_{n}\chi_{n}\psi_{n}.
	\]
	By \eqref{eq:psi-bound}, we can take
	\begin{equation}\label{eq:phi-Theta-2}
		\phi_{n} = \max \Big\{ C \chi_{n}, C' \frac{\log \log n + \log (n \chi_{n}^{2})}{\sqrt{n} \chi_{n}} \Big\},
	\end{equation}
	where 
	\begin{equation}\label{eq:C-C-prime-1}
		C \geq \frac{2b\mu}{\mu_{1}}\Big(\frac{\mu}{\mu_{1}}+s-2+ \frac{\delta_{0}\mu}{3b\mu_{s}}\Big) \bigg(2d_{0} \vee 1 \vee \Big(\frac{\mu}{\mu_{1}} + s - 2\Big) \frac{12 b d_{0}\mu_{s}}{\delta_{0}\mu}\bigg).
	\end{equation}
	and 
	\begin{equation}\label{eq:C-C-prime-2}
		C' \geq \Big(\frac{\mu}{\mu_{1}} +s-2+ \frac{\delta_{0}\mu}{3b\mu_{s}}\Big)\frac{72 b \bar{C}^{2}\mu\mu_{s}^{2}}{\delta_{0}^{2} \mu^{4} \mu_{1} } .
	\end{equation}
	
	With $ \phi_{n} $ given by \eqref{eq:phi-Theta-2}, we need to ensure that it also satisfies condition \eqref{eq:phi-n-condition-1}. By \eqref{eq:phi-Theta-2}, $ \phi_{n}/2 \geq \max_{k = 1, \ldots, s} 2(\mu - \mu_{1}) C_{k} \chi_{n} / (\mu_{1}\mu_{k})$ if we take
	\begin{equation}\label{eq:C}
		C \geq \max_{k = 1, \ldots, s} \frac{4(\mu - \mu_{1})}{ \mu_{1}\mu_{k}}C_{k} .
	\end{equation}
	By the inequality of arithmetic and geometric means, 
	\[
	\frac{\phi_{n}}{2} \geq \frac{1}{4}\Big( C \chi_{n} + C' \frac{\log \log n + \log (n \chi_{n}^{2})}{\sqrt{n} \chi_{n}} \Big) \geq  \frac{n^{-1/4}}{2}\sqrt{CC'(\log \log n + \log (n \chi_{n}^{2}))}.
	\]
	If we take
	\begin{equation}\label{eq:C-C-prime}
		C' \geq \max_{k = 1, \ldots, s} \frac{16(\mu - \mu_{1})^{2}}{C \mu_{1}^{2}\mu_{k}^{2}} C'^{2}_{k},
	\end{equation}
	then for $ n $ sufficiently large,
	\[
	\frac{\phi_{n}}{2} > \frac{\sqrt{CC'}}{2}n^{-1/4}\sqrt{\log \log n } \geq \max_{k = 1, \ldots, s}  \frac{2(\mu - \mu_{1})}{\mu_{1}\mu_{k}} C'_{k} n^{-1/4}\sqrt{\log\log n}.
	\]
	Therefore, condition \eqref{eq:phi-n-condition-1} holds if $ C $ and $ C' $ satisfy \eqref{eq:C} and \eqref{eq:C-C-prime}.
	
	Taking $ C $ and $ C' $ such that \eqref{eq:C-C-prime-1}--\eqref{eq:C-C-prime} hold, we obtain $ \mathbb{P}[\limsup_{n \to \infty} \Theta_{n}] = 0 $ and complete the proof of \eqref{eq:upper-bound}. The almost sure convergence in \eqref{eq:almost-sure-convergence} follows from \eqref{eq:epsilon-1} and \eqref{eq:upper-bound}. 
\end{proof}

\begin{remark}
	Both conditions \eqref{eq:fourth-moments} and \eqref{eq:log-log-n-bound} are essential for the proof of Theorem~\ref{theorem:load-balancing-gap}: Condition \eqref{eq:fourth-moments} ensures that the discrepancies between the renewal processes and their Brownian approximations are negligible, so we can use the Brownian approximations to estimate the increments of the renewal processes---as in \eqref{eq:E-increments} and \eqref{eq:S-increments}. Condition \eqref{eq:log-log-n-bound} ensures that we can find $ \psi_{n} < T $ for all $ T > 0 $, such that the first term will dominate the right side of \eqref{eq:Theta-3-choice}---see the third term on the right side of \eqref{eq:psi-bound}.
\end{remark}

\begin{proof}[Proof of Corollary~\ref{corollary:optimal-gap}.]
	Put $ g(x) = C x $ and $ g_{n}(x) = C' (\log \log n + \log (n x^{2}))/(\sqrt{n} x) $ for $ x > 0 $. Then, $ g $ is strictly increasing, $ g_{n} $ is strictly decreasing for $ x > \mathrm{e}/\sqrt{n} $, and both functions are continuous. When $n$ is large, we have $ \chi_{n} > \mathrm{e}/\sqrt{n} $ by \eqref{eq:epsilon-2}, and $g(\mathrm{e}/\sqrt{n}) < g_{n}(\mathrm{e}/\sqrt{n})$. It follows that $ g(\chi_{n}) \vee g_{n}(\chi_{n}) $ is minimized if and only if $ g(\chi_{n}) = g_{n}(\chi_{n}) $. This yields $ n \chi_{n}^{2} = C' \sqrt{n} (\log \log n + \log (n \chi_{n}^{2})) / C $, by which we obtain $ \lim_{n \to \infty} \log (n \chi_{n}^{2})/\log n = 1/2 $, and thus,
	\[
	\lim_{n \to \infty} \frac{g(\chi_{n}) \vee g_{n}(\chi_{n})}{n^{-1/4}\sqrt{\log n}} = \lim_{n \to \infty} \frac{C \chi_{n}}{n^{-1/4}\sqrt{\log n}} = \lim_{n \to \infty} \frac{\sqrt{CC'(\log \log n + \log (n \chi_{n}^{2}))}}{\sqrt{\log n}} = \sqrt{\frac{CC'}{2}}.
	\]
	The minimized upper bound is asymptotically equal to $ n^{-1/4	}\sqrt{CC'\log n/2} $.
	
	Plugging in $ \chi_{n} = C'' n^{-1/4}\sqrt{\log n} $, we obtain
	\[
	\frac{\log \log n + \log (n \chi_{n}^{2})}{\sqrt{n} \chi_{n}} = \frac{\log n + 4 \log \log n + 4 \log C''}{2 C'' n^{1/4}\sqrt{\log n}},
	\] 
	which, together with \eqref{eq:upper-bound}, yields \eqref{eq:optimal-bound}.
\end{proof}

\section{Proofs of Theorems~\ref{theorem:workload-difference} and~\ref{theorem:CRP}}
\label{sec:Proofs-CRP}

Since the customers join the distributed system and the SSP according to the same renewal process with identical traveling delays, $ A_{n,k} $ is also the arrival process of class-$k$ customers at the station of the SSP. Let $ B_{n,k}^{\star}(t) $ be the amount of time that the server in the SSP has spent serving class-$ k $ customers by time $ t $. Then, the server's cumulative busy time is $ B_{n}^{\star}(t) = \sum_{k = 1}^{s} B_{n,k}^{\star}(t) $, the number of class-$ k $ customers that have completed service is $ S_{k}(\mu B_{n,k}^{\star}(t)/\mu_{k}) $, and the customer count is
\begin{equation}\label{eq:Q-star}
	Q_{n}^{\star}(t) = \sum_{k = 1}^{s} \bigg(A_{n,k}(t) - S_{k}\Big(\frac{\mu}{\mu_{k}} B_{n,k}^{\star}(t)\Big)\bigg).
\end{equation}
We define the fluid-scaled cumulative busy time process in the SSP and that for class-$ k $ customers by $ \bar{B}_{n}^{\star}(t) = B_{n}^{\star}(nt)/n $ and $ \bar{B}_{n,k}^{\star}(t) = B_{n,k}^{\star}(nt)/n $, respectively. 

Because the corresponding customers have identical service requirements, the cumulative stationed workloads must be equal in both systems, and thus
\begin{equation}\label{eq:original-auxiliary}
	\sum_{k = 1}^{s} \mu_{k}B_{n,k}(t) + W_{n}(t) = \mu B_{n}^{\star}(t) + W_{n}^{\star}(t).
\end{equation}
We use this identity to establish asymptotic equivalence between the two systems. By \eqref{eq:workload-difference} and \eqref{eq:original-auxiliary}, the difference between the workloads can also be written as
\begin{equation}\label{eq:Gamma}
	\Gamma^{\star}_{n}(t) = \mu B_{n}^{\star}(t) - \sum_{k = 1}^{s} \mu_{k} B_{n,k}(t),
\end{equation}
or in the scaled version,
\begin{equation}\label{eq:tilde-Gamma}
	\tilde{\Gamma}^{\star}_{n}(t) = \sqrt{n} \Big(\mu \bar{B}_{n}^{\star}(t) - \sum_{k = 1}^{s} \mu_{k} \bar{B}_{n,k}(t)\Big).
\end{equation}
By \eqref{eq:V-nk}, the workload in station~$ k $ is bounded by
\begin{equation}\label{eq:workload-inequality}
	W_{n,k}(t) \leq V_{n,k}(A_{n,k}(t)) - V_{n,k}(A_{n,k}(t) - Q_{n,k}(t)).
\end{equation}
We use this inequality to prove Theorem~\ref{theorem:workload-difference} by considering the total stationed workload. 
\begin{proof}[Proof of Theorem~\ref{theorem:workload-difference}.]
	We first prove $ \sup_{0 < t \leq T} \tilde{\Gamma}^{\star}_{n}(t) \Rightarrow 0 $ as $ n \to \infty $ under the conditions of Theorem~\ref{theorem:load-balancing}. By Proposition~\ref{prop:stationed-workload-comparision}, the sample paths of $\tilde{\Gamma}^{\star}_{n}$ are continuous and piecewise linear. Then,
	\[
	\sup_{0 \leq t \leq T} \tilde{\Gamma}^{\star}_{n}(t) = \sup_{t \in \mathcal{S}^{\star}_{n} } \tilde{\Gamma}^{\star}_{n}(t) = \sup_{t \in \mathcal{S}^{\star}_{n} } \tilde{\Gamma}^{\star}_{n}(t-),
	\]
	where $ \mathcal{S}^{\star}_{n} = \{ t \in (0,T] : \mbox{$ \partial_{-} \tilde{\Gamma}^{\star}_{n}(t) > 0 $} \} \cup \{0\} $, with $\partial_{-}$ denoting the left derivative of a function from $\mathbb{R}_{+}$ to $\mathbb{R}$. Since $ \tilde{\Gamma}^{\star}_{n}(t) \geq 0 $ for $ t \geq 0 $, it suffices to prove 
	\begin{equation}\label{eq:tilde-Gamma-convergence}
		\sup_{t \in \mathcal{S}^{\star}_{n} } \tilde{\Gamma}^{\star}_{n}(t-) \Rightarrow 0 \quad \mbox{as $ n \to \infty $.}
	\end{equation}
	Note that at any moment the cumulative busy time of a server either increases at rate $ 1 $ or stays unchanged. For $t \in \mathcal{S}^{\star}_{n}\setminus\{0\}$, we deduce from \eqref{eq:capacity} and \eqref{eq:tilde-Gamma} that $ \partial_{-} \bar{B}_{n}^{\star}(t) = 1 $ and that at least one of $\partial_{-}\bar{B}_{n,1}(t), \ldots, \partial_{-}\bar{B}_{n,s}(t) $ is zero. Since the servers are work-conserving, at least one of $\tilde{Q}_{n,1}(t-), \ldots, \tilde{Q}_{n,s}(t-)$ is zero. By \eqref{eq:diffusion-scaled} and Theorem~\ref{theorem:load-balancing},
	\begin{equation}\label{eq:Q-t-minus}
		\max_{k = 1, \ldots, s} \sup_{t \in \mathcal{S}^{\star}_{n}} \tilde{Q}_{n,k}(t-) \Rightarrow 0 \quad \mbox{as $ n \to \infty $.} 
	\end{equation}
	Write $\tilde{W}_{n,k}(t) = W_{n,k}(nt)/\sqrt{n}$ and $ \bar{Q}_{n,k}(t) = Q_{n,k}(nt)/n $. By \eqref{eq:workload-inequality}, 
	\[
	\tilde{W}_{n,k}(t) \leq \sqrt{n} \bar{V}_{n,k}(\bar{A}_{n,k}(t) - \bar{Q}_{n,k}(t), \bar{A}_{n,k}(t)] = \tilde{Q}_{n,k}(t) + \tilde{V}_{n,k}(\bar{A}_{n,k}(t) - \bar{Q}_{n,k}(t), \bar{A}_{n,k}(t)] .
	\]
	We deduce from \eqref{eq:FCLT-V} that $ \{\tilde{V}_{n,k} : n\in\mathbb{N} \} $ is C-tight (Lemma~3.2 in \citealt{Iglehart.Whitt.1970a}), which, together with \eqref{eq:Q-t-minus}, implies
	\begin{equation}\label{eq:V-t-minus}
		\sup_{t \in \mathcal{S}^{\star}_{n}} \vert \tilde{V}_{n,k}(\bar{A}_{n,k}(t-) - \bar{Q}_{n,k}(t-), \bar{A}_{n,k}(t-)] \vert \Rightarrow 0 \quad \mbox{as $ n \to \infty $.}
	\end{equation}
	Because $\tilde{\Gamma}^{\star}_{n}(t) \leq \sum_{k = 1}^{s} \tilde{W}_{n,k}(t)$, we obtain \eqref{eq:tilde-Gamma-convergence} from \eqref{eq:Q-t-minus} and \eqref{eq:V-t-minus}.
	
	Assume the conditions of Theorem~\ref{theorem:load-balancing-gap} hold. Write
	\[
	\tilde{\phi}_{n,k} =  \max \Big\{ C \mu_{k} \frac{\chi_{n}}{\sqrt{n}}, C' \mu_{k} \frac{\log \log n + \log (n \chi_{n}^{2})}{n \chi_{n}} \Big\}.
	\]
	Then with probability 1, 
	\begin{equation}\label{eq:Q-bound}
		\sup_{t \in \mathcal{S}^{\star}_{n}} \tilde{Q}_{n,k}(t-) < \sqrt{n}\tilde{\phi}_{n,k} \quad \mbox{for $ n $ sufficiently large.}
	\end{equation} 
	By the inequality of arithmetic and geometric means, 
	\[
	n\tilde{\phi}_{n,k} \geq \frac{1}{2}\Big( C \mu_{k} \sqrt{n}\chi_{n} +  \frac{C' \mu_{k}}{\chi_{n}} ( \log \log n + \log (n \chi_{n}^{2}) ) \Big) \geq  \mu_{k} n^{1/4}\sqrt{CC'(\log \log n + \log (n \chi_{n}^{2}))}.
	\]
	By Lemma~\ref{lemma:strong-approximations} and \eqref{eq:A-bar-convergence}, there exists a standard Brownian motion $\check{B}'_{k}$ such that
	\[
	\lim_{n \to \infty} \frac{1}{n^{1/4}} \sup_{0 \leq t \leq nT} | V_{k}(A_{n,k}(t)) - A_{n,k}(t) - c_{k} \check{B}'_{k}( A_{n,k}(t) ) | \aseq 0.
	\]
	Then, it follows from \eqref{eq:epsilon-2} that
	\begin{equation}\label{eq:Vk-limit}
		\lim_{n \to \infty} \frac{1}{n\tilde{\phi}_{n,k}} \sup_{0 \leq t \leq nT} | V_{k}(A_{n,k}(t)) - A_{n,k}(t) - c_{k} \check{B}'_{k}( A_{n,k}(t) ) | \aseq 0.
	\end{equation}
	By Lemma~\ref{lemma:Wiener-intervals} and \eqref{eq:A-bar-convergence},
	\[
	\limsup_{n \to \infty} \sup_{0 \leq t \leq A_{n,k}(nT)} \sup_{0 < u \leq n\tilde{\phi}_{n,k}} \frac{|\check{B}'_{k}(t, t+u] |}{\sqrt{n\tilde{\phi}_{n,k}(\log \log n - \log  \tilde{\phi}_{n,k} )}} \aseq \sqrt{2}.
	\]
	Note that for $n$ sufficiently large,
	\begin{align*}
		\frac{|\log \log n - \log  \tilde{\phi}_{n,k}|}{n \tilde{\phi}_{n,k}} & = \frac{|\log \log n +\log n - \log  (n\tilde{\phi}_{n,k})|}{n \tilde{\phi}_{n,k}}\\
		& \leq \frac{\log \log n + \log n }{\mu_{k} n^{1/4} \sqrt{CC'(\log \log n + \log (n \chi_{n}^{2}))}} + \frac{\log (n \tilde{\phi}_{n,k})}{n \tilde{\phi}_{n,k}}.
	\end{align*}
	Then by \eqref{eq:epsilon-2}, $ |\log \log n - \log  \tilde{\phi}_{n,k}| = o(n \tilde{\phi}_{n,k}) $, which implies
	\begin{equation}\label{eq:B'-bound}
		\limsup_{n \to \infty} \sup_{0 \leq t \leq A_{n,k}(nT)} \sup_{0 < u \leq n\tilde{\phi}_{n,k}} \frac{|\check{B}'_{k}(t, t+u]|}{n\tilde{\phi}_{n,k}}  \aseq 0.
	\end{equation}
	By \eqref{eq:Q-bound}--\eqref{eq:B'-bound}, we obtain
	\[
	\lim_{n \to \infty}  \frac{1}{\sqrt{n}\tilde{\phi}_{n,k}} \sup_{0 \leq t \leq T} |\tilde{V}_{n,k}(\bar{A}_{n,k}(t) - \bar{Q}_{n,k}(t), \bar{A}_{n,k}(t)]| \aseq 0.
	\]
	We obtain the upper bound, because with probability 1,
	\[
	\sup_{t \in \mathcal{S}^{\star}_{n} } \tilde{\Gamma}^{\star}_{n}(t-) \leq  \sum_{k = 1}^{s} \sup_{t \in \mathcal{S}^{\star}_{n} } \tilde{Q}_{n,k}(t -) + \sum_{k = 1}^{s} \sup_{0 \leq t \leq T} |\tilde{V}_{n,k}(\bar{A}_{n,k}(t) - \bar{Q}_{n,k}(t), \bar{A}_{n,k}(t)]| < \sum_{k = 1}^{s} 2\sqrt{n}\tilde{\phi}_{n,k}
	\]
	for $n$ sufficiently large. The rest of the proof follows that of Corollary~\ref{corollary:optimal-gap}. 
\end{proof}

Using \eqref{eq:workload-inequality} and Proposition~\ref{prop:stationed-workload-comparision}, we can prove the following lemma that concerns the cumulative busy time process for customers in different classes in the SSP.
\begin{lemma}\label{lemma:B-star}
	Assume the conditions of Theorem~\ref{theorem:load-balancing} hold and all the SSPs are initially empty. Then, $(\bar{B}_{n,1}^{\star}, \ldots, \bar{B}_{n,s}^{\star} ) \Rightarrow (\mu_{1}\mathcal{I}/\mu, \ldots, \mu_{s}\mathcal{I}/\mu)$ as $ n \to \infty $.
\end{lemma} 
Lemma~\ref{lemma:B-star} is used in the proofs of Propositions~\ref{prop:asymptotic-equivalence} and~\ref{prop:star-limit}. With these results, we are ready to present the proof of Theorem~\ref{theorem:CRP}.

\begin{proof}[Proof of Theorem~\ref{theorem:CRP}.]
	Let $ \mathbf{A} $ be the coefficient matrix of the following system of linear equations:
	\begin{align*}
		\sum_{k = 1}^{s} \mu_{k} x_{k} = y_{1} \quad \mbox{and} \quad x_{\ell - 1} - x_{\ell} = y_{\ell} \quad \mbox{for $ \ell = 2, \ldots, s $.}
	\end{align*}
	The determinant of $ \mathbf{A} $ is $ \det (\mathbf{A}) = \sum_{k = 1}^{s} \mu_{k}(-1)^{s-k}(-1)^{k-1} = (-1)^{s-1} \mu $. Hence, $ \mathbf{A} $ is invertible. 
	
	We define three $ s $-dimensional column vector processes by $ \tilde{\boldsymbol{L}}_{n}(t) = (\tilde{L}_{n,1}(t), \ldots, \tilde{L}_{n,s}(t))^{\mathrm{T}}$, $\tilde{\boldsymbol{Q}}(t) = (\tilde{Q}(t), \ldots, \tilde{Q}(t))^{\mathrm{T}}$, and $ \tilde{\boldsymbol{Q}}_{0}(t) = (\tilde{Q}(t), 0, \ldots, 0)^{\mathrm{T}}$, where the superscript $ \mathrm{T} $ stands for transpose. By Propositions~\ref{prop:asymptotic-equivalence} and~\ref{prop:star-limit} and the convergence-together theorem (Theorem~5.4 in \citealp{Chen.Yao.2001}), $ \sum_{k = 1}^{s} \tilde{Q}_{n,k} \Rightarrow \tilde{Q} $ as $ n \to \infty $. Then by \eqref{eq:diffusion-scaled}, Theorem~\ref{theorem:load-balancing} in this paper, and Theorem~3.9 in \citet{Billingsley.1999}, $ \mathbf{A} \tilde{\boldsymbol{L}}_{n} \Rightarrow \tilde{\boldsymbol{Q}}_{0} $ as $ n \to \infty $. Since $ \mathbf{A} \tilde{\boldsymbol{Q}} = \mu \tilde{\boldsymbol{Q}}_{0} $, we deduce from the continuous mapping theorem that $ \tilde{\boldsymbol{L}}_{n} \Rightarrow \mathbf{A}^{-1} \tilde{\boldsymbol{Q}}_{0} = \tilde{\boldsymbol{Q}}/\mu $ as $ n \to \infty $, from which the assertion follows.
\end{proof}

\section{Proof of Theorem~\ref{theorem:workload-optimality}}
\label{sec:Proof-Capacity}

Let $A^{\dagger}_{n,k}(t)$ be the number of class-$k$ customers that have arrived at the station of the $n$th MDSP by time~$t$, with the fluid-scaled version $\bar{A}^{\dagger}_{n,k}(t) = A^{\dagger}_{n,k}(nt)/n$. Clearly, $A^{\dagger}_{n,k}(t) \geq A_{n,k}(t)$. The following lemma specifies an upper bound for the difference between the two arrival processes.

\begin{lemma}\label{lemma:comparison-MDSP-arrivals}
	Assume the conditions of Theorem~\ref{theorem:load-balancing} hold for a sequence of GBC systems and all the MDSPs are initially empty. Then, there exist two positive numbers $C^{\dagger}_{k}$ and $C^{\ddagger}_{k}$ such that with probability 1, $ \sup_{0 \leq t \leq T} \sqrt{n}( \bar{A}^{\dagger}_{n,k}(t) - \bar{A}_{n,k}(t) ) < \max\{ C^{\dagger}_{k} \chi_{n}, C^{\ddagger}_{k} n^{-1/4}\sqrt{\chi_{n}\log n \vee 1} \} $ for $T > 0$, $k = 1, \ldots, s$, and $n$ sufficiently large.
\end{lemma}

To prove Theorem~\ref{theorem:workload-optimality}, we employ the approach in the proof of Theorem~\ref{theorem:workload-difference} to estimate the difference between the total stationed workloads in the two systems, and rely on Lemma~\ref{lemma:comparison-MDSP-arrivals} to estimate the difference between the en route workloads.
\begin{proof}[Proof of Theorem~\ref{theorem:workload-optimality}.]
	Let us first prove $ \sup_{0 < t \leq T} \tilde{\Gamma}^{\dagger}_{n}(t) \Rightarrow 0 $ as $ n \to \infty $ under the conditions of Theorem~\ref{theorem:load-balancing}. By Proposition~\ref{prop:JNQ-workload-comparision}, 
	\begin{equation}\label{eq:sup-Gamma-dagger}
		\sup_{0 \leq t \leq T} \tilde{\Gamma}^{\dagger}_{n}(t) = \sup_{t \in \mathcal{S}^{\dagger}_{n} } \tilde{\Gamma}^{\dagger}_{n}(t-),
	\end{equation}
	where $ \mathcal{S}^{\dagger}_{n} = \{ t \in (0,T] : \mbox{$ \partial_{-} \tilde{\Gamma}^{\dagger}_{n}(t) > 0 $} \} \cup \{0\} $. Following the first part of the proof of Theorem~\ref{theorem:workload-difference}, we obtain
	\[
	\max_{k = 1, \ldots, s} \sup_{t \in \mathcal{S}^{\dagger}_{n}} \tilde{Q}_{n,k}(t-) \Rightarrow 0 \quad \mbox{as $ n \to \infty $,} 
	\]
	and then obtain
	\begin{equation}\label{eq:tilde-Wnk-limit}
		\max_{k = 1, \ldots, s} \sup_{t \in \mathcal{S}^{\dagger}_{n} } \tilde{W}_{n,k}(t-) \Rightarrow 0 \quad \mbox{as $ n \to \infty $,} 
	\end{equation}
	where $\tilde{W}_{n,k}(t) = W_{n,k}(nt)/\sqrt{n}$. Since the cumulative workloads are equal in the two systems, 
	\[
	\sum_{k = 1}^{s} V_{k}(A_{n,k}(t)) + U_{n}(t) = \sum_{k = 1}^{s} V_{k}(A^{\dagger}_{n,k}(t)) + U^{\dagger}_{n}(t).
	\] 
	Write $\tilde{U}_{n}(t) = U_{n}(nt)/\sqrt{n}$ and $\tilde{U}^{\dagger}_{n}(t) = U^{\dagger}_{n}(nt)/\sqrt{n}$. Then,
	\begin{equation}\label{eq:enroute-difference}
		\tilde{U}_{n}(t) - \tilde{U}^{\dagger}_{n}(t) = \sum_{k = 1}^{s} \big(\sqrt{n} \big( \bar{A}^{\dagger}_{n,k}(t) - \bar{A}_{n,k}(t) \big) + \tilde{V}_{n,k}(\bar{A}_{n,k}(t), \bar{A}^{\dagger}_{n,k}(t)]\big).
	\end{equation}
	By Lemma~\ref{lemma:comparison-MDSP-arrivals} and the C-tightness of $\{ \tilde{V}_{n,k} : n\in\mathbb{N} \}$, $\sup_{0 < t \leq T} (\tilde{U}_{n}(t) - \tilde{U}^{\dagger}_{n}(t)) \Rightarrow 0$ as $n \to \infty$. Because 
	\begin{equation}\label{eq:tilde-dagger-bound}
		\tilde{\Gamma}^{\dagger}_{n}(t) \leq \sum_{k = 1}^{s} \tilde{W}_{n,k}(t) + \tilde{U}_{n}(t) - \tilde{U}^{\dagger}_{n}(t),
	\end{equation}
	we deduce from \eqref{eq:sup-Gamma-dagger} and \eqref{eq:tilde-Wnk-limit} that $ \sup_{0 < t \leq T} \tilde{\Gamma}^{\dagger}_{n}(t) \Rightarrow 0 $ as $ n \to \infty $.
	
	Assume the conditions of Theorem~\ref{theorem:load-balancing-gap} hold. Following the second part of the proof of Theorem~\ref{theorem:workload-difference}, we obtain
	\begin{equation}\label{eq:bound-1}
		\sum_{k = 1}^{s} \sup_{t \in \mathcal{S}^{\dagger}_{n} } \tilde{W}_{n,k}(t-)  <  \max \Big\{ 2C \mu \chi_{n}, 2C' \mu \frac{\log \log n + \log (n \chi_{n}^{2})}{\sqrt{n} \chi_{n}} \Big\},
	\end{equation}
	where $C$ and $C'$ are the positive numbers specified in Theorem~\ref{theorem:load-balancing-gap}. Write
	\[
	\hat{\phi}_{n,k} =  \max \Big\{ C^{\dagger}_{k}\frac{\chi_{n}}{\sqrt{n}}, C^{\ddagger}_{k} \frac{\sqrt{\chi_{n}\log n \vee 1}}{n^{3/4}} \Big\},
	\]
	where $C^{\dagger}_{k}$ and $C^{\ddagger}_{k}$ are the positive numbers specified in Lemma~\ref{lemma:comparison-MDSP-arrivals}. Since $n \hat{\phi}_{n,k} \geq C^{\ddagger}_{k} n^{1/4} $, following the proof of Theorem~\ref{theorem:workload-difference}, we obtain
	\[
	\lim_{n \to \infty}  \frac{1}{\sqrt{n}\hat{\phi}_{n,k}} \sup_{0 \leq t \leq T} |\tilde{V}_{n,k}(\bar{A}_{n,k}(t), \bar{A}^{\dagger}_{n,k}(t)]| \aseq 0.
	\]
	It follows from Lemma~\ref{lemma:comparison-MDSP-arrivals} and \eqref{eq:enroute-difference} that with probability 1,
	\[
	\sup_{0 \leq t \leq T} \big( \tilde{U}_{n}(t) - \tilde{U}^{\dagger}_{n}(t) \big) < \sum_{k = 1}^{s} 2\sqrt{n}\hat{\phi}_{n,k} < \sum_{k = 1}^{s} \Big( 2C^{\dagger}_{k}\chi_{n} + 2C^{\ddagger}_{k} \frac{\sqrt{\chi_{n}\log n \vee 1}}{n^{1/4}} \Big) \quad \mbox{for $n$ sufficiently large.}
	\]
	Since $n^{-1/4}\sqrt{\chi_{n}\log n \vee 1} = o(n^{-1/4} \sqrt{\log n})$ by  \eqref{eq:epsilon-1}, it follows from Corollary~\ref{corollary:optimal-gap} that when $n$ is large, $n^{-1/4}\sqrt{\chi_{n}\log n \vee 1}$ is negligible in comparison with the right side of \eqref{eq:bound-1}. Then by \eqref{eq:sup-Gamma-dagger} and \eqref{eq:tilde-dagger-bound},
	\[
	\sup_{0 \leq t \leq T} \tilde{\Gamma}^{\dagger}_{n}(t) < \max \Big\{ 3C \mu \chi_{n}, 3C' \mu \frac{\log \log n + \log (n \chi_{n}^{2})}{\sqrt{n} \chi_{n}} \Big\} + \sum_{k = 1}^{s} 2C^{\dagger}_{k}\chi_{n} \quad \mbox{for $n$ sufficiently large,}
	\]
	which allows us to get the upper bound by taking
	\[
	C^{\dagger} = 3C \mu + \sum_{k = 1}^{s} 2C^{\dagger}_{k} \quad \mbox{and} \quad C^{\ddagger} = 3C' \mu + \frac{C'}{C}\sum_{k = 1}^{s} 2C^{\dagger}_{k}.
	\]
	Lastly, we complete the proof following that of Corollary~\ref{corollary:optimal-gap}.
\end{proof}

\section{Proofs of Propositions}
\label{sec:Proofs-Propositions}

All the propositions are proved in this section.

\subsection{Proposition~\ref{prop:lower-bound}}
\label{sec:two-station-example}

Let $\check{z} = \sum_{j=1}^{\nu} z(j)$, where $\nu$ is a geometric random variable that is independent of $\{ z(j) : j \in\mathbb{N} \}$ and has mean $\mu/\mu_{1}$. Then, $\mathbb{E}[\check{z}] = \mu/\mu_{1}$ and the coefficient of variation is $\check{c}_{0} = \sqrt{1 + \mu_{1}(c^{2}_{0}-1)/\mu}$ (see, e.g., Section~3.4.1 in \citealt{Ross.2010}). Since $\check{z}^{4} \leq \nu^{4} (\max_{1\leq j \leq \nu} z(j))^{4}  \leq \nu^{4} \sum_{j = 1}^{\nu} z(j)^{4} $, we obtain $\mathbb{E}[\check{z}^{4}] \leq \mathbb{E}[\nu^{5}]\mathbb{E}[z(1)^{4}] < \infty$.

For notational convenience, write $ E'_{n}(t) = \sum_{j = 1}^{E_{n}(t)} \mathbb{1}_{\{ 0 \leq u(j) < \mu_{1}/\mu  \}}$. Then, $E'_{n}$ is a renewal process and $\check{z}/\lambda_{n}$ is a generic inter-renewal time of $E'_{n}$. By \eqref{eq:E-strong-approximation}, there exists a standard Brownian motion $ \check{B}'_{0} $ such that
\begin{equation}\label{eq:E'-strong-approximation}
	\lim_{n \to \infty} \frac{1}{n^{1/4}} \sup_{0 \leq t \leq nT} | E'_{n}(t) - \mu_{1}\rho_{n} t - \check{c}_{0} \check{B}'_{0}(\mu_{1}\rho_{n} t ) | \aseq 0 \quad \mbox{for $ T > 0 $.}
\end{equation}
Let $ \tilde{E}'_{n}(t) = (E'_{n}(nt) - n\mu_{1}\rho_{n}t)/\sqrt{n} $. By \eqref{eq:Gmk}, $ \tilde{E}'_{n}(t - n^{-1/2}d_{1,1}) = \mu_{1} \tilde{E}^{1,\Delta}_{n,1}(t)/\mu + \tilde{G}^{1,\Delta}_{n,1}(t) $ for $ t \geq n^{-1/2}d_{1,1} $. Since $ \bar{B}_{n,1}(n^{-1/2}d_{1,1}) = 0 $, we deduce from \eqref{eq:diffusion-scaled-queue} that $ \tilde{Q}_{n,1}(t) \geq \sqrt{n}(\bar{A}_{n,1}(t) - \bar{S}_{n,1}(t - n^{-1/2} d_{1,1}))$ for $ t \geq n^{-1/2}d_{1,1} $, which, along with \eqref{eq:tilde-A} and \eqref{eq:tilde-Ank-decomposition}, yields
\[
\tilde{Q}_{n,1}(t) \geq \tilde{E}'_{n}(t - n^{-1/2}d_{1,1}) + \tilde{\mathcal{E}}^{1,\Delta}_{n,1}(t) + \tilde{M}^{1,\Delta}_{n,1}(t) - \tilde{S}_{n,1}(t - n^{-1/2} d_{1,1}) - \sqrt{n}(1 - \rho_{n})\mu_{1}(t - n^{-1/2} d_{1,1}).
\]

Since $E'_{n}$ and $S_{1}$ are independent, $\check{B}'_{0}$ and $\check{B}_{1}$ can be taken as independent standard Brownian motions. By the law of the iterated logarithm for Brownian motion (Lemma~1 in \citealp{Huggins.1985}),
\[
\limsup_{n \to \infty} \frac{\check{c}_{0} \check{B}'_{0}(\sqrt{n}\mu_{1}\rho_{n} (d'-d_{1,1}) ) - c_{1} \check{B}_{1}(\sqrt{n}\mu_{1}(d'-d_{1,1}) )}{n^{1/4}\sqrt{\log\log n}} \aseq \sqrt{2(\check{c}_{0}^{2} + c_{1}^{2})\mu_{1} (d' - d_{1,1})}.
\]
It follows from  \eqref{eq:Sk-strong-approximation} and \eqref{eq:E'-strong-approximation} that
\[
\limsup_{n \to \infty} \frac{\tilde{E}'_{n}(n^{-1/2}(d' - d_{1,1})) - \tilde{S}_{n,1}(n^{-1/2}(d' - d_{1,1}))}{n^{-1/4}\sqrt{\log\log n}} \aseq \sqrt{2(\check{c}_{0}^{2} + c_{1}^{2})\mu_{1} (d' - d_{1,1})}.
\]
By the strong law of large numbers for renewal processes, 
\[
\lim_{n \to \infty} \frac{\tilde{\mathcal{E}}^{1,\Delta}_{n,1}(n^{-1/2}d')}{\chi_{n}} = \lim_{n \to \infty} \frac{1}{\sqrt{n}} E_{n} (\sqrt{n}(d' - d_{1,1})) \aseq \mu (d' - d_{1,1}).
\]
As in the proof of Lemma~\ref{lemma:after-initial}, we can obtain
\[
\limsup_{n \to \infty}  \frac{\sup_{d_{1,1} \leq t \leq d'} |\tilde{M}^{1,\Delta}_{n,1}(n^{-1/2}t)|}{\chi_{n} \vee n^{-1/4}\sqrt{\log \log n}} \aseq 0.
\]
Taking $0 < \check{C} < \mu (d' - d_{1,1})/\mu_{1}$ and $0 < \check{C}' < \sqrt{2(\check{c}_{0}^{2} + c_{1}^{2})(d' - d_{1,1})/\mu_{1}}$, we complete the proof by the above results and the heavy-traffic condition \eqref{eq:heavy-traffic}.

\subsection{Proposition~\ref{prop:JSQ}}
\label{sec:proof-JSQ}

The proof of Proposition~\ref{prop:JSQ} is similar to that of Theorem~\ref{theorem:load-balancing-gap} but simpler. Consider the event
\[
\Theta_{n} = \Big\{ \sup_{0 \leq t \leq T} \max_{k,\ell = 1, \ldots, s} | \tilde{L}_{n,k}(t) - \tilde{L}_{n,\ell}(t) | \geq C_{\operatorname{JSQ}}\frac{\log n}{\sqrt{n}} \Big\},
\]
where $ C_{\operatorname{JSQ}} $ is a positive number to be determined. It can be written as $ \Theta_{n} = \{ \tau_{1,n} \leq T \} $, where 
\[
\tau_{1,n} = \inf \Big\{ t \geq 0: \max_{k,\ell = 1, \ldots, s} | \tilde{L}_{n,k}(t) - \tilde{L}_{n,\ell}(t) | \geq C_{\operatorname{JSQ}}\frac{\log n}{\sqrt{n}} \Big\}.
\] 
As a convention, $ \tau_{1,n} = \infty $ if $ \max_{k,\ell = 1, \ldots, s} | \tilde{L}_{n,k}(t) - \tilde{L}_{n,\ell}(t) | < C_{\operatorname{JSQ}} \log n/\sqrt{n} $ for all $ t \geq 0 $.

Consider an outcome in $ \Theta_{n} $. Let stations~$ \bar{\ell}(n) $ and~$ \underline{\ell}(n) $ be the stations having the longest and shortest queue lengths at time $ \tau_{1,n} $. Then, $ \tilde{L}_{n,\underline{\ell}(n)}(\tau_{1,n}) \leq \tilde{L}_{n,\ell}(\tau_{1,n}) \leq \tilde{L}_{n,\bar{\ell}(n)}(\tau_{1,n}) $ for $ \ell = 1, \ldots, s $, and $ \tilde{L}_{n,\bar{\ell}(n)}(\tau_{1,n}) - \tilde{L}_{n,\underline{\ell}(n)}(\tau_{1,n}) \geq C_{\operatorname{JSQ}} \log n/\sqrt{n} $. By \eqref{eq:capacity}, \eqref{eq:diffusion-scaled}, and the assumption $ \mu_{1} \leq \cdots \leq \mu_{s} $,
\begin{equation}\label{eq:Q-tau-1}
	\sum_{\ell = 1}^{s} \tilde{Q}_{n,\ell}(\tau_{1,n}) \leq \mu \tilde{L}_{n,\bar{\ell}(n)}(\tau_{1,n}) - \mu_{1} C_{\operatorname{JSQ}} \frac{\log n}{\sqrt{n}}.
\end{equation}
Put $ \tau_{2,n} = \sup\{ t \in [0, \tau_{1,n}] : \tilde{L}_{n,\bar{\ell}(n)}(t-) \leq \tilde{L}_{n,\ell}(t-)\mbox{ for all $ \ell \neq \bar{\ell}(n) $}  \} $ where we take $\tilde{L}_{n,k}(0-) = 0$ for $k = 1, \ldots, s$. Since the system is initially empty, $ 0 \leq \tau_{2,n} \leq \tau_{1,n} $, and by \eqref{eq:capacity} and \eqref{eq:diffusion-scaled}, 
\begin{equation}\label{eq:Q-tau-2}
	\mu \tilde{L}_{n,\bar{\ell}(n)}(\tau_{2,n}-) \leq \sum_{\ell = 1}^{s} \tilde{Q}_{n,\ell}(\tau_{2,n}-).
\end{equation}
Under the JSQ policy, there are no customers joining station $ \bar{\ell}(n) $ and the server is kept busy during $ (\tau_{2,n}, \tau_{1,n}] $. By \eqref{eq:diffusion-scaled-queue}, $ \tilde{Q}_{n,\bar{\ell}(n)}[\tau_{2,n}, \tau_{1,n}] \leq \Delta\tilde{E}_{n}(\tau_{2,n}) - \sqrt{n} \bar{S}_{n,\bar{\ell}(n)}[\bar{B}_{n,\bar{\ell}(n)}(\tau_{2,n}), \bar{B}_{n,\bar{\ell}(n)}(\tau_{1,n})]$, and 
\[
\sum_{\ell = 1}^{s} \tilde{Q}_{n,\ell}[\tau_{2,n}, \tau_{1,n}] = \sqrt{n}\bar{E}_{n}[\tau_{2,n}, \tau_{1,n}] - \sum_{\ell = 1}^{s} \sqrt{n} \bar{S}_{n,\ell}[\bar{B}_{n,\ell}(\tau_{2,n}), \bar{B}_{n,\ell}(\tau_{1,n})].
\]
By \eqref{eq:Q-tau-1}, \eqref{eq:Q-tau-2}, and the assumption $ \mu_{1} \leq \cdots \leq \mu_{s} $,
\begin{align*}
	\mu_{1} C_{\operatorname{JSQ}} \frac{\log n}{\sqrt{n}} & \leq \mu \tilde{L}_{n,\bar{\ell}(n)}[\tau_{2,n}, \tau_{1,n}] - \sum_{\ell = 1}^{s} \tilde{Q}_{n,\ell}[\tau_{2,n}, \tau_{1,n}]\\ 
	& \leq \frac{\mu}{\mu_{1}} \big(\Delta\tilde{E}_{n}(\tau_{2,n}) - \sqrt{n} \bar{S}_{n,\bar{\ell}(n)}[\bar{B}_{n,\bar{\ell}(n)}(\tau_{2,n}), \bar{B}_{n,\bar{\ell}(n)}(\tau_{1,n})] \big)\\
	& \quad -\sqrt{n}\bar{E}_{n}[\tau_{2,n}, \tau_{1,n}] + \sum_{\ell = 1}^{s} \sqrt{n} \bar{S}_{n,\ell}[\bar{B}_{n,\ell}(\tau_{2,n}), \bar{B}_{n,\ell}(\tau_{1,n})].
\end{align*}
Because $ \bar{B}_{n,\bar{\ell}(n)}(\tau_{1,n}) - \bar{B}_{n,\bar{\ell}(n)}(\tau_{2,n}) = \tau_{1,n} - \tau_{2,n} $ and $ \bar{B}_{n,\ell}(\tau_{1,n}) - \bar{B}_{n,\ell}(\tau_{2,n}) \leq \tau_{1,n} - \tau_{2,n} $ for $ \ell \neq \bar{\ell}(n) $,
\begin{align*}
	\mu_{1} C_{\operatorname{JSQ}} \frac{\log n}{\sqrt{n}} & \leq \frac{\mu}{\mu_{1}} \big(\Delta\tilde{E}_{n}(\tau_{2,n}) - \tilde{S}_{n,\bar{\ell}(n)}[\bar{B}_{n,\bar{\ell}(n)}(\tau_{2,n}), \bar{B}_{n,\bar{\ell}(n)}(\tau_{1,n})]\big)\\
	& \quad  - \tilde{E}_{n}[\tau_{2,n}, \tau_{1,n}] + \sum_{\ell = 1}^{s} \tilde{S}_{n,\ell}[\bar{B}_{n,\ell}(\tau_{2,n}), \bar{B}_{n,\ell}(\tau_{1,n})] - \sqrt{n}\lambda_{n}(\tau_{1,n} - \tau_{2,n}).
\end{align*}
Then, we obtain 
\begin{equation}\label{eq:JSQ-bound}
	\mu_{1} C_{\operatorname{JSQ}} \leq \mathcal{D}_{n}(\tau_{2,n}, \tau_{1,n}),
\end{equation}
where for $ 0 \leq t_{2} \leq t_{1} $,
\begin{align*}
	\mathcal{D}_{n}(t_{2}, t_{1}) & = \frac{\mu\sqrt{n}}{\mu_{1}\log n} \Delta\tilde{E}_{n}(t_{2}) + \frac{\mu\sqrt{n}}{\mu_{1}\log n} \max_{\ell = 1, \ldots, s} | \tilde{S}_{n,\ell}[\bar{B}_{n,\ell}(t_{2}), \bar{B}_{n,\ell}(t_{1})] |\\
	& \quad  + \frac{\sqrt{n}}{\log n}| \tilde{E}_{n}[t_{2}, t_{1}] | + \frac{\sqrt{n}}{\log n}\sum_{\ell = 1}^{s} |\tilde{S}_{n,\ell}[\bar{B}_{n,\ell}(t_{2}), \bar{B}_{n,\ell}(t_{1})]| - \frac{n}{\log n}\lambda_{n}(t_{1} - t_{2}).
\end{align*}

By Lemma~\ref{lemma:strong-approximations}, there exist mutually independent standard Brownian motions $ \check{B}_{0} , \ldots, \check{B}_{s} $ and positive numbers $ \hat{C}_{0}, \ldots \hat{C}_{s} $, such that
\begin{align}
	\limsup_{n \to \infty} \frac{1}{\log n} \sup_{0 \leq t \leq nT} \vert E_{n}(t) - \lambda_{n}t - c_{0} \check{B}_{0}(\lambda_{n} t) \vert & \asl \hat{C}_{0}, \label{eq:strong-E-JSQ}\\
	\limsup_{n \to \infty} \frac{1}{\log n} \sup_{0 \leq t \leq nT} \vert S_{k}(t) - \mu_{k}t - c_{k} \check{B}_{k}(\mu_{k}t) \vert & \asl \hat{C}_{k} \label{eq:strong-S-JSQ}
\end{align}
for $ k = 1, \ldots, s $. Then by Lemma~\ref{lemma:Wiener-intervals},
\begin{align}
	\limsup_{n \to \infty} \sup_{0 \leq t \leq nT} \sup_{0 \leq u \leq a \log n} \frac{1}{\log n} |E_{n}(t+u) - E_{n}(t) - \lambda_{n} u | & \asl 2\hat{C}_{0} + c_{0}\sqrt{2 a \mu}, \label{eq:E-increments-JSQ}\\
	\limsup_{n \to \infty} \sup_{0 \leq t \leq n T } \sup_{0 \leq u \leq a \log n} \frac{1}{\log n} |S_{k}(t+u) - S_{k}(t) - \mu_{k} u | & \asl 2\hat{C}_{k} + c_{k}\sqrt{2 a \mu_{k}} \label{eq:S-increments-JSQ}
\end{align}
for $ k = 1, \ldots, s $ and all $ a > 0 $. By \eqref{eq:strong-E-JSQ}, we also obtain
\begin{equation}\label{eq:jump-bound}
	\limsup_{n \to \infty} \sup_{0 \leq t \leq T} \frac{\sqrt{n}}{\log n}\Delta \tilde{E}_{n}(t) \asl 2\hat{C}_{0}.
\end{equation}

Let $ a' $ be a sufficiently large positive number such that
\[
\frac{2\mu}{\mu_{1}} \hat{C}_{0} +  \frac{\mu}{\mu_{1}} \max_{k = 1, \ldots, s} \big(2\hat{C}_{k} + c_{k}\sqrt{2 a' \mu_{k}}\big) + \big( 2\hat{C}_{0} + c_{0}\sqrt{2 a' \mu} \big) + \sum_{k = 1}^{s} \big(2\hat{C}_{k} + c_{k}\sqrt{2 a' \mu_{k}}\big) - \mu a' < 0.
\]
Then by \eqref{eq:heavy-traffic} and \eqref{eq:E-increments-JSQ}--\eqref{eq:jump-bound}, 
\begin{equation}\label{eq:negative-increment}
	\limsup_{n\rightarrow \infty} \sup_{0 \leq t \leq T} \mathcal{D}_{n}\Big(t, t + \frac{a'\log n}{n} \Big) \asl 0. 
\end{equation}
Put 
\[
\tau_{2,n}' = \tau_{2,n} + \frac{a'\log n}{n} \Big\lfloor \frac{n(\tau_{1,n} - \tau_{2,n})}{a'\log n} \Big\rfloor .
\] 
Then, $ 0 \leq \tau_{1,n} - \tau_{2,n}' < a' \log n/n $. Using \eqref{eq:E-increments-JSQ}--\eqref{eq:jump-bound} again, we obtain
\begin{align*}
	\limsup_{n \to \infty}\mathcal{D}_{n}(\tau_{2,n}', \tau_{1,n}) & \asl \frac{2\mu}{\mu_{1}} \hat{C}_{0} +  \frac{\mu}{\mu_{1}} \max_{k = 1, \ldots, s} \big(2\hat{C}_{k} + c_{k}\sqrt{2 a' \mu_{k}}\big) \\
	& \quad + \big( 2\hat{C}_{0} + c_{0}\sqrt{2 a' \mu} \big) + \sum_{k = 1}^{s} \big(2\hat{C}_{k} + c_{k}\sqrt{2 a' \mu_{k}}\big) \\
	& < \mu a'.
\end{align*}
Since $ \mathcal{D}_{n}(t_{2}, t_{1}) \leq \mathcal{D}_{n}(t_{2}, t'_{2}) + \mathcal{D}_{n}(t'_{2}, t_{1}) $ for $ 0 \leq t_{2} \leq t'_{2} \leq t_{1} $, it follows from \eqref{eq:negative-increment} that
\[
\limsup_{n \to \infty} \mathcal{D}_{n}(\tau_{2,n}, \tau_{1,n})  \asleq \limsup_{n \to \infty} \mathcal{D}_{n}(\tau_{2,n}', \tau_{1,n}) \asl \mu a'.
\]
Taking $ C_{\operatorname{JSQ}} \geq \mu a'/\mu_{1} $, we deduce from \eqref{eq:JSQ-bound} that $ \mathbb{P}[\limsup_{n \to \infty} \Theta_{n}] = 0 $. 

\subsection{Proposition~\ref{prop:stationed-workload-comparision}}
\label{sec:proof-stationed-workload-comparison}

At any time, the cumulative busy time of a server either increases at rate $ 1 $ or stays unchanged. Hence, $B_{n,1}, \ldots, B_{n,s}$, and $B_{n}^{\star}$ all have continuous, piecewise linear sample paths, with slopes either 1 or 0. By \eqref{eq:Gamma}, $ \Gamma^{\star}_{n} $ must have continuous, piecewise linear sample paths. 

Consider $ \tau_{n}^{\star} = \inf \{ t > 0 : \Gamma^{\star}_{n}(t) < 0 \}$ and suppose $ \tau_{n}^{\star} < \infty$. Since $ \Gamma^{\star}_{n} $ has continuous sample paths with $\Gamma^{\star}_{n}(0) = 0$, we have $ \Gamma^{\star}_{n}(\tau_{n}^{\star}) = 0 $. Let $\delta$ be a positive number. If $B_{n}^{\star}$ is strictly increasing on $[\tau_{n}^{\star}, \tau_{n}^{\star}+\delta]$, by \eqref{eq:capacity} and \eqref{eq:Gamma}, we should have $ \Gamma^{\star}_{n}(t) \geq 0$ for $ t \in [\tau_{n}^{\star}, \tau_{n}^{\star}+\delta]$, which is a contradiction of the definition of $\tau_{n}^{\star}$. Hence, we should be able to find $\delta > 0$ such that $ B_{n}^{\star}(\tau_{n}^{\star}) = B_{n}^{\star}(\tau_{n}^{\star}+\delta) $. This implies that $ W_{n}^{\star}(t) = 0 $ for $ t \in [\tau_{n}^{\star}, \tau_{n}^{\star}+\delta) $, since the server is work-conserving. However, we should also be able to find $ t' \in (\tau_{n}^{\star}, \tau_{n}^{\star}+\delta) $ such that $ \Gamma^{\star}_{n}(t') < 0 $. A contradiction of \eqref{eq:original-auxiliary} follows from
\[  
\mu B_{n}^{\star}(t') + W_{n}^{\star}(t') = \mu B_{n}^{\star}(t') < \sum_{k = 1}^{s} \mu_{k} B_{n,k}(t') \leq \sum_{k = 1}^{s} \mu_{k}B_{n,k}(t') + W_{n}(t').
\]

\subsection{Proposition~\ref{prop:JNQ-workload-comparision}}
\label{sec:proof-JNQ-workload-comparison}

The proof is mostly identical to that of Proposition~\ref{prop:stationed-workload-comparision}. Let $B^{\dagger}_{n}(t)$ be the cumulative busy time of the server in the MDSP. Since the cumulative workloads are equal in the two systems,
\begin{equation}\label{eq:cumulative-system-workloads}
	\sum_{k = 1}^{s} \mu_{k}B_{n,k}(t) + W_{n}(t) + U_{n}(t) = \mu B_{n}^{\dagger}(t) + W_{n}^{\dagger}(t) + U_{n}^{\dagger}(t),
\end{equation}
which, along with  \eqref{eq:workload-difference-JNQ}, yields $\Gamma^{\dagger}_{n}(t) = \mu B_{n}^{\dagger}(t) - \sum_{k = 1}^{s} \mu_{k} B_{n,k}(t) $. It follows that $\Gamma^{\dagger}$ has continuous, piecewise linear sample paths.

Consider $ \tau_{n}^{\dagger} = \inf \{ t > 0 : \Gamma^{\dagger}_{n}(t) < 0 \}$ and suppose $ \tau_{n}^{\dagger} < \infty$. As in the proof of Proposition~\ref{prop:stationed-workload-comparision}, we should have $ \Gamma^{\dagger}_{n}(\tau_{n}^{\dagger}) = 0 $ and should be able to find some $\delta > 0$ such that $ B_{n}^{\dagger}(\tau_{n}^{\dagger}) = B_{n}^{\dagger}(\tau_{n}^{\dagger}+\delta) $. Then, $ W_{n}^{\dagger}(t) = 0 $ for $ t \in [\tau_{n}^{\dagger}, \tau_{n}^{\dagger}+\delta) $. However, we should also be able to find $ t' \in (\tau_{n}^{\dagger}, \tau_{n}^{\dagger}+\delta) $ such that $ \Gamma^{\dagger}_{n}(t') < 0 $. Since the traveling delay of a customer in the distributed system cannot be shorter than that of the counterpart in the MDSP, we have $ U^{\dagger}_{n}(t') \leq U_{n}(t')$. Then, a contradiction of \eqref{eq:cumulative-system-workloads} follows from
\begin{align*}
	\mu B_{n}^{\dagger}(t') + W_{n}^{\dagger}(t') + U_{n}^{\dagger}(t') & = \mu B_{n}^{\dagger}(t') + U_{n}^{\dagger}(t') \\
	& < \sum_{k = 1}^{s} \mu_{k} B_{n,k}(t') + U_{n}(t') \\
	& \leq \sum_{k = 1}^{s} \mu_{k}B_{n,k}(t') + W_{n}(t') + U_{n}(t').
\end{align*}

\subsection{Proposition~\ref{prop:asymptotic-equivalence}}
\label{sec:Proof-Customer-Count-Equivalence}
By \eqref{eq:queue-length}, \eqref{eq:Q-star}, and \eqref{eq:tilde-Gamma}, 
\[  
\sum_{k = 1}^{s} \tilde{Q}_{n,k}(t) - \tilde{Q}^{\star}_{n}(t) = \tilde{\Gamma}^{\star}_{n} + \sum_{k = 1}^{s} \bigg(\tilde{S}_{n,k}\Big(\frac{\mu}{\mu_{k}} \bar{B}_{n,k}^{\star}(t)\Big) - \tilde{S}_{n,k}\big(\bar{B}_{n,k}(t)\big) \bigg).
\]
By \eqref{eq:FCLT} and Lemma~3.2 in \citet{Iglehart.Whitt.1970a}, $ \{\tilde{S}_{n,k} : n\in\mathbb{N} \} $ is C-tight for $k = 1, \ldots, s$. Then by Lemmas~\ref{lemma:busy-time} and~\ref{lemma:B-star}, 
\[  
\sup_{0 \leq t \leq T} \Big\vert \tilde{S}_{n,k}\Big(\frac{\mu}{\mu_{k}} \bar{B}_{n,k}^{\star}(t)\Big) - \tilde{S}_{n,k}\big(\bar{B}_{n,k}(t)\big) \Big\vert \Rightarrow 0 \quad \mbox{as $ n \to \infty $.}
\]
The assertion follows from Theorem~\ref{theorem:workload-difference}.

\subsection{Proposition~\ref{prop:star-limit}}
\label{sec:Proof-Star-Limit}

By \eqref{eq:tilde-A} and \eqref{eq:Q-star}, $ \tilde{Q}_{n}^{\star} (t) = \tilde{Z}_{n}^{\star}(t) + \mu \tilde{I}_{n}^{\star}(t) $, where $ \tilde{I}_{n}^{\star}(t) = \sqrt{n} \big( t - \bar{B}_{n}^{\star}(t) ) $ is the scaled cumulative idle time and
\[  
\tilde{Z}_{n}^{\star}(t) = \tilde{A}_{n}(t) - \sum_{k = 1}^{s} \tilde{S}_{k}\Big(\frac{\mu}{\mu_{k}} \bar{B}_{n,k}^{\star}(t)\Big) - \sqrt{n}(1 - \rho_{n}) \mu  t - \rho_{n} \mu \sum_{k = 1}^{s} \sum_{m = 1}^{b}  p_{m}r_{m,k} \big( \sqrt{n} t - ( \sqrt{n} t - d_{m,k})^{+} \big).
\]
For ease of notation, write $ \iota_{n}(t) = \rho_{n} \mu \sum_{k = 1}^{s} \sum_{m = 1}^{b}  p_{m}r_{m,k} ( \sqrt{n} t - ( \sqrt{n} t - d_{m,k})^{+} ) $.

Because the server is work-conserving, $ \tilde{Q}_{n}^{\star} = \Phi(\tilde{Z}_{n}^{\star}) $ and $ \tilde{I}_{n}^{\star} = \Psi(\tilde{Z}_{n}^{\star})/\mu $, where $ (\Phi, \Psi) $ is the reflection map---that is, $\Phi(f)(t) = f(t) + \sup_{0 \leq u \leq t} (-f(u))^{+}$ and $ \Psi(f)(t) = \sup_{0 \leq u \leq t} (-f(u))^{+} $ for $f \in \mathbb{D}$ (Section~6.2 in \citealp{Chen.Yao.2001}).

To prove the proposition, one may wish to follow a standard continuous mapping approach. However, $ \iota_{n}(0) = 0 $ for $ n \in \mathbb{N} $ but $ \lim_{t\downarrow 0} \lim_{n \to \infty} \iota_{n}(t) = \mu  \sum_{k = 1}^{s} \sum_{m = 1}^{b} p_{m} r_{m,k} d_{m,k} $. Since the limit process is not right-continuous at $0$, $\tilde{Z}^{\star}_{n}$ does not converge in $\mathbb{D}$. We cannot apply the continuous mapping theorem to $\{ \tilde{Z}^{\star}_{n} : n \in \mathbb{N} \}$ directly.

Let us consider $\tilde{Z}^{\diamond}_{n}(t) = \tilde{Z}^{\star}_{n}(t) + \iota_{n}(t) $. By \eqref{eq:FCLT}, Lemma~\ref{lemma:A-star}, and the convergence-together theorem (Theorem~5.4 in \citealp{Chen.Yao.2001}), $ (\tilde{A}_{n}, \tilde{S}_{n,1}, \ldots, \tilde{S}_{n,s}) \Rightarrow (\tilde{E}, \tilde{S}_{1}, \ldots, \tilde{S}_{s}) $ as $ n \to \infty $. By Lemma~\ref{lemma:B-star} and the random time-change theorem (Theorem~5.3 in \citealp{Chen.Yao.2001}), $ \tilde{Z}_{n}^{\diamond} \Rightarrow \tilde{Z}^{\diamond} $ as $ n \to \infty $, where $ \tilde{Z}^{\diamond}(t) = \tilde{E}(t) - \sum_{k = 1}^{s} \tilde{S}_{k}(t) - \beta \mu t$. Hence, $\{ \tilde{Z}_{n}^{\diamond} : n \in \mathbb{N} \}$ is C-tight (Lemma~3.2 in \citealp{Iglehart.Whitt.1970a}). Because $\tilde{Z}^{\diamond}_{n}(0) = 0$,
\begin{equation}\label{eq:dagger-star-0}
	\sup_{0 \leq u \leq n^{-1/2}d_{0}} \vert \tilde{Z}^{\diamond}_{n}(u) \vert \Rightarrow 0 \quad \mbox{as $n \to \infty$.}
\end{equation}

Put $\tilde{Q}^{\diamond}_{n} = \Phi(\tilde{Z}^{\diamond}_{n})$ and $\tilde{I}^{\diamond}_{n} = \Psi(\tilde{Z}^{\diamond}_{n}) / \mu$. By Theorem~6.1 in \citet{Chen.Yao.2001} and the continuous mapping theorem (Theorem~5.2 in \citealp{Chen.Yao.2001}), $ (\tilde{Q}_{n}^{\diamond}, \tilde{I}_{n}^{\diamond}) \Rightarrow (\tilde{Q}, \tilde{I}^{\diamond}) $ as $ n \to \infty $, where $\tilde{Q} = \Phi(\tilde{Z}^{\diamond})$ and $\tilde{I}^{\diamond} = \Psi(\tilde{Z}^{\diamond}) / \mu$. Then, $ \tilde{Q} $ is a one-dimensional reflected Brownian motion starting from $0$ with drift $ -\beta \mu $ and variance $ \mu c^{2}_{0} + \sum_{k = 1}^{s} \mu_{k} c^{2}_{k} $.

It remains to show $ \sup_{0 \leq t \leq T} \vert \tilde{Q}^{\diamond}_{n}(t) - \tilde{Q}^{\star}_{n}(t) \vert \Rightarrow 0 $ as $n \to \infty$, which, along with the convergence-together theorem, will complete the proof. By the definition of the reflection map,
\begin{align*}
	\tilde{Q}^{\diamond}_{n}(t) - \tilde{Q}^{\star}_{n}(t) & = \big(\tilde{Z}^{\diamond}_{n}(t) + \mu \tilde{I}^{\diamond}_{n}(t)\big) - \big(\tilde{Z}^{\star}_{n}(t) + \mu \tilde{I}^{\star}_{n}(t)\big) \\
	& = \big(\tilde{Z}^{\diamond}_{n}(t) - \tilde{Z}^{\star}_{n}(t) \big) + \mu \tilde{I}^{\diamond}_{n}(t)  - \mu \tilde{I}^{\star}_{n}(t) \\
	& = \iota_{n}(t) + \sup_{0 \leq u \leq t} \big(-\tilde{Z}^{\diamond}_{n}(u)\big)^{+}  - \sup_{0 \leq u \leq t} \big(-\tilde{Z}^{\diamond}_{n}(u) + \iota_{n}(u) \big)^{+} \\
	& = \iota_{n}(t) - \inf_{0 \leq u \leq t} \tilde{Z}^{\diamond}_{n}(u)  + \inf_{0 \leq u \leq t} \big(\tilde{Z}^{\diamond}_{n}(u) - \iota_{n}(u) \big),
\end{align*}
where the last equality holds because $\tilde{Z}^{\diamond}_{n}(0) = 0$ and $\iota_{n}(0) = 0$. We use this identity to establish lower and upper bounds for $\tilde{Q}^{\diamond}_{n}(t) - \tilde{Q}^{\star}_{n}(t)$. 

Since $\iota_{n}$ is a nonnegative and nondecreasing function, the above identity yields
\begin{equation}\label{eq:dagger-star-1}
	\tilde{Q}^{\diamond}_{n}(t) - \tilde{Q}^{\star}_{n}(t) \geq 0 \quad \mbox{for $t \geq 0$.}
\end{equation}
When $0 \leq t < n^{-1/2}d_{0}$,
\begin{equation}\label{eq:dagger-star-2}
	\tilde{Q}^{\diamond}_{n}(t) - \tilde{Q}^{\star}_{n}(t) \leq \iota_{n}(t) - \inf_{0 \leq u \leq t} \tilde{Z}^{\diamond}_{n}(u) + \big(\tilde{Z}^{\diamond}_{n}(t) - \iota_{n}(t) \big) \leq \sup_{0 \leq u \leq n^{-1/2}d_{0}} \vert 2\tilde{Z}^{\diamond}_{n}(u) \vert.
\end{equation}
When $t \geq n^{-1/2}{d_{0}}$, $ \iota_{n}(t) = \iota_{n}(n^{-1/2}d_{0}) $. In this case, if $\inf_{0 \leq u \leq t} \tilde{Z}^{\diamond}_{n}(u) = \inf_{0 \leq u < n^{-1/2}{d_{0}}} \tilde{Z}^{\diamond}_{n}(u) $ holds,
\begin{equation}\label{eq:dagger-star-3}
	\tilde{Q}^{\diamond}_{n}(t) - \tilde{Q}^{\star}_{n}(t) \leq \iota_{n}(t) -\inf_{0 \leq u \leq t} \tilde{Z}^{\diamond}_{n}(u) + \big(\tilde{Z}^{\diamond}_{n}(n^{-1/2}d_{0}) - \iota_{n}(n^{-1/2}d_{0})\big) \leq \sup_{0 \leq u \leq n^{-1/2}d_{0}} \vert 2\tilde{Z}^{\diamond}_{n}(u) \vert;
\end{equation}
if $\inf_{0 \leq u \leq t} \tilde{Z}^{\diamond}_{n}(u) = \inf_{n^{-1/2}{d_{0}} \leq u \leq t } \tilde{Z}^{\diamond}_{n}(u) $ holds,
\begin{equation}\label{eq:dagger-star-4}
	\tilde{Q}^{\diamond}_{n}(t) - \tilde{Q}^{\star}_{n}(t) \leq \iota_{n}(t) - \inf_{0 \leq u \leq t} \tilde{Z}^{\diamond}_{n}(u)  + \inf_{n^{-1/2}d_{0} \leq u \leq t} \big(\tilde{Z}^{\diamond}_{n}(u) - \iota_{n}(u) \big) = 0.
\end{equation} 
The asymptotic equivalence between $\tilde{Q}_{n}^{\diamond}$ and $\tilde{Q}_{n}^{\star}$ follows from \eqref{eq:dagger-star-0}--\eqref{eq:dagger-star-4}.

\section{Proofs of Additional Lemmas}
\label{sec:Proofs-Additional-Lemmas}

This section provides the proofs of the lemmas given in the appendices.

\subsection{Lemma~\ref{lemma:after-initial}}
\label{sec:proof-initial}

Let $\tilde{I}_{n,k}(t) = \sqrt{n} (t - \bar{B}_{n,k}(t))$, which is the scaled cumulative idle time of the server in station $k$. By \eqref{eq:tilde-Ank-decomposition} and \eqref{eq:diffusion-scaled-queue}, $\tilde{Q}_{n,k}(t) = \tilde{Z}_{n,k}(t) + \grave{Z}_{n,k}(t) + \mu_{k} \tilde{I}_{n,k}(t)$, where 
\begin{align*} 
	\tilde{Z}_{n,k}(t) & = \sum_{m = 1}^{b} \big(p_{m}r_{m,k} \tilde{E}^{m,\Delta}_{n,k}(t) + \tilde{G}^{m,\Delta}_{n,k}(t) + \tilde{\mathcal{E}}^{m,\Delta}_{n,k}(t) + \tilde{M}^{m,\Delta}_{n,k}(t) \big) - \tilde{S}_{n,k}(\bar{B}_{n,k}(t)) - \sqrt{n} (1 - \rho_{n}) \mu_{k} t, \\
	\grave{Z}_{n,k}(t) & = \sqrt{n}\rho_{n}\mu \sum_{m = 1}^{b}  p_{m}r_{m,k}(t - n^{-1/2} d_{m,k})^{+} - \sqrt{n} \rho_{n} \mu_{k} t.
\end{align*}
Since the server is work-conserving, $\mu_{k}\tilde{I}_{n,k}(t) = \sup_{0 \leq u \leq t} (-\tilde{Z}_{n,k}(u) - \grave{Z}_{n,k}(u) )^{+} $ (Section~6.2 in \citealp{Chen.Yao.2001}). By \eqref{eq:rmk-heavy-traffic}, $ \grave{Z}_{n,k} $ is nonincreasing and $ \grave{Z}_{n,k}(t) \leq 0 $ for $ t \geq 0 $. Then,
\begin{align*}
	\tilde{Q}_{n,k}(t) & \leq \tilde{Z}_{n,k}(t) + \grave{Z}_{n,k}(t) + \sup_{0 \leq u \leq t} \big(-\tilde{Z}_{n,k}(u) \big)^{+} + \sup_{0 \leq u \leq t} \big( - \grave{Z}_{n,k}(u) \big)^{+} \\
	& = \tilde{Z}_{n,k}(t) + \sup_{0 \leq u \leq t} \big(-\tilde{Z}_{n,k}(u) \big)^{+} \\
	& \leq \sup_{0 \leq u \leq t} 2| \tilde{Z}_{n,k}(u) |.
\end{align*}

By the strong law of large numbers for renewal processes, $\lim_{n \to \infty} E_{n}(\sqrt{n}(T_{0} - d_{m,k})^{+})/\sqrt{n} \aseq \mu (T_{0} - d_{m,k})^{+}$. By Lemma~\ref{lemma:LIL},
\begin{align*} 
	\limsup_{n \to \infty}  \frac{\sup_{0 \leq t \leq T_{0}} |\tilde{E}^{m,\Delta}_{n,k}(n^{-1/2}t)|}{n^{-1/4}\sqrt{\log\log n}} & = \limsup_{n \to \infty}  \frac{\sup_{0 \leq t \leq \sqrt{n}(T_{0}-d_{m,k})^{+}} |E_{n}(t)- \lambda_{n}t|}{\sqrt{n^{1/2}\log\log n}}  \\
	& \aseq \sqrt{2 \mu (T_{0} - d_{m,k})^{+}} c_{0}.
\end{align*}
By \eqref{eq:Gmk-LIL},
\begin{align*} 
	\limsup_{n \to \infty}  \frac{\sup_{0 \leq t \leq T_{0}} |\tilde{G}^{m,\Delta}_{n,k}(n^{-1/2}t)|}{n^{-1/4}\sqrt{\log\log n}} & = \limsup_{n \to \infty} \frac{\max_{1 \leq j \leq E_{n}(\sqrt{n}(T_{0}-d_{m,k})^{+}) }|G^{m}_{n,k}(j)|}{\sqrt{n^{1/2}\log \log n}}\\
	& \aseq \sqrt{2 \mu p_{m}r_{m,k}(1 - p_{m}r_{m,k})(T_{0} - d_{m,k})^{+}}.
\end{align*}
By \eqref{eq:pm-epsilon-chi}, $ \sup_{0 \leq t \leq T_{0}} |\tilde{\mathcal{E}}^{m,\Delta}_{n,k}(n^{-1/2}t)| \leq \chi_{n} E_{n}(\sqrt{n}(T_{0} - d_{m,k})^{+})/\sqrt{n}$, and thus
\[
\limsup_{n \to \infty} \sup_{0 \leq t \leq T_{0}} \frac{|\tilde{\mathcal{E}}^{m,\Delta}_{n,k}(n^{-1/2}t)|}{\chi_{n}} \asleq \mu (T_{0} - d_{m,k})^{+}.
\]
By \eqref{eq:Mmnk-fluctuations},
\begin{align*}
	\limsup_{n \to \infty}  \frac{\sup_{0 \leq t \leq T_{0}} |\tilde{M}^{m,\Delta}_{n,k}(n^{-1/2}t)|}{n^{-1/4}\sqrt{\chi_{n} \log n \vee 1}} & = \limsup_{n \to \infty} \frac{\max_{1 \leq j \leq E_{n}(\sqrt{n}(T_{0}-d_{m,k})^{+}) }|M^{m}_{n,k}(j)|}{\sqrt{n^{1/2}(\chi_{n} \log n \vee 1)}} \\
	& \asleq \sqrt{10\mu (T_{0} - d_{m,k})^{+}}.
\end{align*}
If $ \chi_{n} \log n \geq 1 $, $ n^{-1/4} \sqrt{\chi_{n} \log n \vee 1} /\chi_{n} \leq n^{-1/4} (\chi_{n} \log n \vee 1) /\chi_{n} = n^{-1/4} \log n $. This implies
\[
\lim_{n \to \infty} \frac{n^{-1/4} \sqrt{\chi_{n} \log n \vee 1}}{\chi_{n} \vee n^{-1/4}\sqrt{\log \log n}} = 0,
\]
by which we obtain
\[
\limsup_{n \to \infty}  \frac{\sup_{0 \leq t \leq T_{0}} |\tilde{M}^{m,\Delta}_{n,k}(n^{-1/2}t)|}{\chi_{n} \vee n^{-1/4}\sqrt{\log \log n}} \aseq 0.
\]
By \eqref{eq:Gmk-LIL} and \eqref{eq:Mmnk-fluctuations},
\[
\lim_{n \to \infty}\frac{1}{\sqrt{n}} \max_{1 \leq j \leq \sqrt{n}t} |G^{m}_{k}(j) + M^{m}_{n,k}(j)| \aseq 0 \quad \mbox{for $ t \geq 0 $.}
\]
By \eqref{eq:pm-epsilon-chi}, \eqref{eq:epsilon-1}, and \eqref{eq:Rmnk-decomposition}, $ \lim_{n \to \infty} R^{m}_{n,k}(E_{n}(\sqrt{n}(T_{0} - d_{m,k})^{+}))/\sqrt{n} \aseq \mu p_{m}r_{m,k} (T_{0} - d_{m,k})^{+} $. By \eqref{eq:station-arrivals-analysis} and \eqref{eq:V-nk}, $ \lim_{n \to \infty} A_{n,k}(\sqrt{n}T_{0})/\sqrt{n}\aseq \lim_{n \to \infty} V_{k}(A_{n,k}(\sqrt{n}T_{0}))/\sqrt{n} \aseq \mu \sum_{m = 1}^{b} p_{m}r_{m,k} (T_{0} - d_{m,k})^{+} $. Since $ B_{n,k}(\sqrt{n}T_{0}) \leq V_{k}(A_{n,k}(\sqrt{n}T_{0}))/\mu_{k} $ by \eqref{eq:B-V}, we deduce from Lemma~\ref{lemma:LIL} that
\begin{align*} 
	\limsup_{n \to \infty}  \frac{\sup_{0 \leq t \leq T_{0}} |\tilde{S}_{n,k}(\bar{B}_{n,k}(n^{-1/2}t))|}{n^{-1/4}\sqrt{\log\log n}} & = \limsup_{n \to \infty}  \frac{\sup_{0 \leq t \leq B_{n,k}(\sqrt{n}T_{0})} |S_{k}(t)- \mu_{k}t|}{\sqrt{n^{1/2}\log\log n}}  \\
	& \leq \limsup_{n \to \infty}  \frac{\sup_{0 \leq t \leq V_{k}(A_{n,k}(\sqrt{n}T_{0}))/\mu_{k}} |S_{k}(t)- \mu_{k}t|}{\sqrt{n^{1/2}\log\log n}} \\
	& \aseq \Big(2 \mu \sum_{m = 1}^{b} p_{m}r_{m,k} (T_{0} - d_{m,k})^{+}\Big)^{1/2} c_{k}.
\end{align*}
By the above results and \eqref{eq:heavy-traffic}, there exist two positive numbers $ C_{k} $ and $ C'_{k} $, both depending on $ T_{0} $, such that $ \sup_{0 \leq t \leq T_{0}} 2|\tilde{Z}_{n,k}(n^{-1/2} t)| \asl C_{k}\chi_{n} + C'_{k}n^{-1/4}\sqrt{\log\log n}$ for $n$ sufficiently large.

\subsection{Lemma~\ref{lemma:A-star}}
\label{sec:proof-A-star}

By \eqref{eq:Rmnk} and \eqref{eq:station-arrivals-analysis},
\[
E_{n}(t) - \sum_{k = 1}^{s} A_{n,k}(t) = \sum_{k = 1}^{s} \sum_{m = 1}^{b} \big(R^{m}_{n,k}(E_{n}(t)) - R^{m}_{n,k}\big(E_{n}((t-\sqrt{n} d_{m,k})^{+})\big)\big).
\]
Then by \eqref{eq:Rmnk-fluid},
\begin{align*}
	\begin{split}
		\tilde{E}_{n}(t) - \tilde{A}_{n}(t) = \sum_{k = 1}^{s} \sum_{m = 1}^{b} \big( & p_{m}r_{m,k}  \tilde{E}_{n}\big( (t - n^{-1/2}d_{m,k})^{+}, t \big] + \tilde{G}^{m}_{n,k}\big((t - n^{-1/2}d_{m,k})^{+}, t \big]  \\
		& + \tilde{M}^{m}_{n,k}\big((t - n^{-1/2}d_{m,k})^{+}, t \big] + \tilde{\mathcal{E}}^{m}_{n,k} \big( (t - n^{-1/2}d_{m,k})^{+}, t \big]  \big).
	\end{split}
\end{align*}

By \eqref{eq:FCLT} and Lemma~3.2 in \citet{Iglehart.Whitt.1970a}, $ \{\tilde{E}_{n} : n\in\mathbb{N} \} $ is C-tight. By \eqref{eq:Gmk-strong-approximation}, \eqref{eq:FSLLN}, and the random time-change theorem (Theorem~5.3 in \citealp{Chen.Yao.2001}), $ \tilde{G}^{m}_{n,k} \Rightarrow \tilde{G}^{m}_{k} $ as $ n \to \infty $, where $ \tilde{G}^{m}_{k} $ is a driftless Brownian motion starting from $ 0 $ with variance $ p_{m}r_{m,k}(1 - p_{m}r_{m,k}) \mu $. Then, $ \{ \tilde{G}^{m}_{n,k} : n \in \mathbb{N} \} $ is C-tight. Following a similar argument, we prove the C-tightness of $\{ \tilde{M}^{m}_{n,k} : n\in\mathbb{N} \}$ by the last assertion of Lemma~\ref{lemma:martingale}. Then for $m = 1, \ldots, b$ and $k = 1, \ldots, s$, 
\[
\sup_{0 \leq t \leq T} \big\vert \tilde{E}_{n}\big( (t - n^{-1/2}d_{m,k})^{+}, t \big] \big\vert \vee \big\vert \tilde{G}^{m}_{n,k}\big((t - n^{-1/2}d_{m,k})^{+}, t \big]\big\vert \vee \big\vert \tilde{M}^{m}_{n,k}\big((t - n^{-1/2}d_{m,k})^{+}, t \big] \big\vert \Rightarrow 0
\]
as $ n \to \infty $. Note that $ \vert \tilde{\mathcal{E}}^{m}_{n,k}( (t - n^{-1/2}d_{m,k})^{+}, t ] \vert \leq  \chi_{n} ( \rho_{n} \mu d_{m,k} + \vert \tilde{E}_{n}( (t - n^{-1/2}d_{m,k})^{+}, t ] \vert ) $. Then by \eqref{eq:heavy-traffic} and \eqref{eq:epsilon-1}, $ \sup_{0 \leq t \leq T} \vert \tilde{\mathcal{E}}^{m}_{n,k}( (t - n^{-1/2}d_{m,k})^{+}, t ] \vert \Rightarrow 0 $ as $ n \to \infty $, which completes the proof.

\subsection{Lemma~\ref{lemma:stochastic-boundedness}}
\label{sec:proof-boundedness}

If the assertion is not true, we may find an increasing sequence of positive integers $ \{ n(i) : i \in \mathbb{N} \} $, along with $ \delta > 0 $ and $ T > 0 $, such that $ \mathbb{P}[ \sup_{0 \leq t \leq T} \sum_{k = 1}^{s} \tilde{Q}_{n(i),k}(t) > 2i ] > \delta $ for $ i \in \mathbb{N} $.

Consider the event 
\[  
\Omega_{i} = \Big\{ \sup_{0 \leq t \leq T} \sum_{k = 1}^{s} \tilde{Q}_{n(i),k}(t) > 2i,\, \max_{k,\ell = 1, \ldots, s} \sup_{0 \leq t \leq T} | \tilde{L}_{n(i),k}(t) - \tilde{L}_{n(i),\ell}(t)| < 1 \Big\}.
\]
Write $ \tau_{i,1} = \inf\{ t\geq 0: \sum_{k = 1}^{s} \tilde{Q}_{n(i),k}(t) > 2i \} $ and $ \tau_{i,2} = \sup \{ t \in (0, \tau_{i,1}) : \sum_{k = 1}^{s} \tilde{Q}_{n(i),k}(t) < i \} $. Since $ \sum_{k = 1}^{s} \tilde{Q}_{n(i),k}(0) = 0 $, we must have $ 0 < \tau_{i,2} \leq \tau_{i,1} $ and $ \sum_{k = 1}^{s} \tilde{Q}_{n(i),k}[\tau_{i,2}, \tau_{i,1}] \geq i $ on $ \Omega_{i} $. Because the differences between the scaled queue lengths are all less than 1, we must have $ \tilde{L}_{n(i),k}(t) > 0 $ for all $ t \in [\tau_{i,2}, \tau_{i,1}] $, $k = 1, \ldots, s$, and $i$ sufficiently large. This implies that $ \bar{B}_{n(i),k}[\tau_{i,2}, \tau_{i,1}] = \tau_{i,1} - \tau_{i,2} $. By \eqref{eq:commute-delays} and \eqref{eq:station-arrivals}, $ \sum_{k = 1}^{s} \bar{A}_{n(i),k}[\tau_{i,2}, \tau_{i,1}] \leq \bar{E}_{n(i)}[(\tau_{i,2} - n(i)^{-1/2}d_{0})^{+}, \tau_{i,1}] $. Then by \eqref{eq:heavy-traffic} and \eqref{eq:diffusion-scaled-queue},
\begin{align*}
	\sum_{k = 1}^{s} \tilde{Q}_{n(i),k}[\tau_{i,2}, \tau_{i,1}] & \leq \tilde{E}_{n(i)}[(\tau_{i,2} - n(i)^{-1/2}d_{0})^{+}, \tau_{i,1}] - \sum_{k = 1}^{s} \tilde{S}_{n(i),k}[\bar{B}_{n(i),k}(\tau_{i,2}), \bar{B}_{n(i),k}(\tau_{i,1})] \\
	& \quad - \sqrt{n(i)} (1 - \rho_{n(i)}) \mu (\tau_{i,1} - \tau_{i,2}) + \lambda_{n(i)} d_{0}  \\
	& \leq 2 \sup_{0 \leq t \leq T} | \tilde{E}_{n(i)}(t) | + 2 \sum_{k = 1}^{s} \sup_{0 \leq t \leq T}|\tilde{S}_{n(i),k}(t)| + \sqrt{n(i)} |1 - \rho_{n(i)}| \mu T + \lambda_{n(i)} d_{0},
\end{align*}
the right side of which is stochastically bounded by \eqref{eq:heavy-traffic} and \eqref{eq:FCLT}. Since $ \sum_{k = 1}^{s} \tilde{Q}_{n(i),k}[\tau_{i,2}, \tau_{i,1}] \geq i $ on $ \Omega_{i} $, the above inequality implies that $ \lim_{i \to \infty} \mathbb{P}[\Omega_{i}] = 0 $. Then, we deduce from Theorem~\ref{theorem:load-balancing} that
\begin{align*}
	\lim_{i \to \infty}\mathbb{P}\Big[ \sup_{0 \leq t \leq T} \sum_{k = 1}^{s} \tilde{Q}_{n(i),k}(t) > 2i \Big] & \leq  \lim_{i \to \infty} \mathbb{P}[\Omega_{i}] + \lim_{i \to \infty} \mathbb{P} \Big[ \max_{k,\ell = 1, \ldots, s} \sup_{0 \leq t \leq T} | \tilde{L}_{n(i),k}(t) - \tilde{L}_{n(i),\ell}(t)| \geq 1 \Big] \\
	& = 0,
\end{align*}
which is a contradiction.

\subsection{Lemma~\ref{lemma:busy-time}}
\label{sec:proof-busy}

By \eqref{eq:commute-delays} and \eqref{eq:station-arrivals}, $ \sum_{k = 1}^{s} \bar{A}_{n,k}(t) \geq \bar{E}_{n}((t - n^{-1/2}d_{0})^{+})$. Write $ \bar{Q}_{n,k}(t) = Q_{n,k}(nt)/n $. By \eqref{eq:queue-length}, \eqref{eq:busy-inequality}, and \eqref{eq:diffusion-scaled-queue},
\begin{equation}\label{eq:bar-Q-inequality} 
	\sum_{k = 1}^{s} \bar{Q}_{n,k}(t) \geq \bar{E}_{n}\big((t - n^{-1/2}d_{0})^{+}\big) - \sum_{k = 1}^{s} \bar{S}_{n,k}\big(\bar{B}_{n,k}(t)\big) \geq \bar{E}_{n}\big((t - n^{-1/2}d_{0})^{+}\big) - \sum_{k = 1}^{s} \bar{S}_{n,k}(t).
\end{equation}
By \eqref{eq:FSLLN}, $ \sup_{0 \leq t \leq T} \vert \bar{E}_{n}((t - n^{-1/2}d_{0})^{+}) - \mu t \vert \Rightarrow 0 $ and $ \sup_{0 \leq t \leq T} \vert \bar{S}_{n,k}(t) - \mu_{k}t \vert \Rightarrow 0 $ as $ n \to \infty $. Hence,
\[  
\sup_{0 \leq t \leq T} \Big\vert \bar{E}_{n}\big((t - n^{-1/2}d_{0})^{+}\big) - \sum_{k = 1}^{s} \bar{S}_{n,k}(t) \Big\vert \Rightarrow 0 \quad \mbox{as $ n \to \infty $.}
\]
By Lemma~\ref{lemma:stochastic-boundedness}, $\sup_{0 \leq t \leq T} \sum_{k = 1}^{s} \bar{Q}_{n,k}(t) \Rightarrow 0 $ as $ n \to \infty $. Then, we deduce from \eqref{eq:bar-Q-inequality} that
\[ 
\sup_{0 \leq t \leq T} \sum_{k = 1}^{s} \big( \bar{S}_{n,k}(t) - \bar{S}_{n,k}\big(\bar{B}_{n,k}(t)\big) \big) \Rightarrow 0 \quad \mbox{as $ n \to \infty $.}
\]
By  \eqref{eq:FSLLN} and \eqref{eq:busy-inequality}, $ \sup_{0 \leq t \leq T} \vert \bar{S}_{n,k}(\bar{B}_{n,k}(t)) - \mu_{k} \bar{B}_{n,k}(t) \vert \Rightarrow 0 $ as $ n \to \infty $.
Note that
\begin{align*} 
	0 \leq \sup_{0 \leq t \leq T} \sum_{k = 1}^{s} \mu_{k} \big( t - \bar{B}_{n,k}(t) \big) & \leq \sum_{k = 1}^{s} \sup_{0 \leq t \leq T} \big( \bar{S}_{n,k}(t) - \bar{S}_{n,k}\big(\bar{B}_{n,k}(t)\big) \big) + \sum_{k = 1}^{s} \sup_{0 \leq t \leq T} \vert \mu_{k}t - \bar{S}_{n,k}(t) \vert\\
	& \quad + \sum_{k = 1}^{s}  \sup_{0 \leq t \leq T} \big\vert \bar{S}_{n,k}\big(\bar{B}_{n,k}(t)\big) - \mu_{k} \bar{B}_{n,k}(t) \big\vert.
\end{align*}
Then, using the above convergence results, we obtain $ \sup_{0 \leq t \leq T} \sum_{k = 1}^{s} \mu_{k} ( t - \bar{B}_{n,k}(t) ) \Rightarrow 0 $ as $ n \to \infty $, which implies $ \bar{B}_{n,k} \Rightarrow \mathcal{I} $ as $ n \to \infty $ for $ k = 1, \ldots, s $. Finally, the joint convergence follows from Theorem~3.9 in \citet{Billingsley.1999}.

\subsection{Lemma~\ref{lemma:B-star}}
\label{sec:proof-B-star}

Similar to \eqref{eq:original-auxiliary}, the cumulative workload arriving at station~$ k $ in the distributed system satisfies $ \mu_{k}B_{n,k}(t) + W_{n,k}(t) = \mu B_{n,k}^{\star}(t) + W_{n,k}^{\star}(t) $. Then, $ -W_{n}^{\star}(t) \leq \mu B_{n,k}^{\star}(t) - \mu_{k}B_{n,k}(t) \leq W_{n,k}(t)$. By \eqref{eq:workload-difference} and Proposition~\ref{prop:stationed-workload-comparision}, $ W_{n}(t) - W_{n}^{\star}(t) \geq 0 $. Then by \eqref{eq:workload-inequality},
\[  
\vert \mu B_{n,k}^{\star}(t) - \mu_{k}B_{n,k}(t) \vert \leq W_{n}(t) \leq \sum_{k = 1}^{s} \big( V_{n,k}(A_{n,k}(t)) - V_{n,k}\big(A_{n,k}(t) - Q_{n,k}(t)\big) \big).
\]
Lemma~\ref{lemma:stochastic-boundedness}, along with \eqref{eq:FSLLN-V} and \eqref{eq:A-bar-convergence}, implies $ \sup_{0 \leq t \leq T} \bar{V}_{n,k}\big(\bar{A}_{n,k}(t) - \bar{Q}_{n,k}(t), \bar{A}_{n,k}(t)\big] \Rightarrow 0 $ as $ n \to \infty $. Hence, $ \sup_{0 \leq t \leq T} \vert \mu \bar{B}_{n,k}^{\star}(t) - \mu_{k}\bar{B}_{n,k}(t) \vert \Rightarrow 0 $ as $ n \to \infty $. The assertion follows from Lemma~\ref{lemma:busy-time} and the convergence-together theorem (Theorem~5.4 in \citealp{Chen.Yao.2001}).

\subsection{Lemma~\ref{lemma:comparison-MDSP-arrivals}}
\label{sec:proof-comparison-MDSP-arrivals}																								
Since the traveling delays of customers from origin~$m$ are all equal to $\sqrt{n}\underline{d}_{m}$, by \eqref{eq:Rmnk} we obtain $ A^{\dagger}_{n,k}(t) = \sum_{m = 1}^{b} R^{m}_{n,k}(E_{n}((t-\sqrt{n} \underline{d}_{m})^{+}))$, which, along with \eqref{eq:bar-Ank}, yields
\[
\bar{A}^{\dagger}_{n,k}(t) - \bar{A}_{n,k}(t) = \sum_{m = 1}^{b} \big( \bar{R}^{m}_{n,k}\big((t-n^{-1/2}\underline{d}_{m})^{+}\big) -  \bar{R}^{m}_{n,k}\big((t-n^{-1/2} d_{m,k})^{+}\big)\big).
\]
If $\bar{A}^{\dagger}_{n,k}(t) > \bar{A}_{n,k}(t)$, there exists some $m$ such that $d_{m,k} > \underline{d}_{m}$. For such an $m$, $r_{m,k} = 0$ by \eqref{eq:nearest-station} and $\mathbb{1}_{ \{ \kappa^{m}_{k-1} \leq u(i) < \kappa^{m}_{k}, \, \boldsymbol{\gamma}(i) = \boldsymbol{d}_{m}  \}} = 0$, where $ \kappa^{m}_{0} = 0 $ and $ \kappa^{m}_{k} = \sum_{\ell = 1}^{k} r_{m,\ell} $. Then by \eqref{eq:Rmnk-decomposition}--\eqref{eq:Mmnk} and \eqref{eq:Rmnk-fluid},
\begin{align*}
	\bar{R}^{m}_{n,k}\big((t-n^{-1/2}\underline{d}_{m})^{+}\big) -  \bar{R}^{m}_{n,k}\big((t-n^{-1/2} d_{m,k})^{+}\big) & = \frac{1}{\sqrt{n}}\big(\tilde{\mathcal{E}}^{m}_{n,k}\big((t- n^{-1/2} \underline{d}_{m})^{+}\big) - \tilde{\mathcal{E}}^{m, \Delta}_{n,k}(t)\big)\\
	& \quad + \frac{1}{\sqrt{n}} \big(\tilde{M}^{m}_{n,k}\big((t- n^{-1/2} \underline{d}_{m})^{+}\big) - \tilde{M}^{m,\Delta}_{n,k}(t)\big).
\end{align*}
By \eqref{eq:pm-epsilon-chi}, \eqref{eq:cali-Emnk}, and \eqref{eq:FSLLN},
\[
\limsup_{n \to \infty} \sup_{0 \leq t \leq T} \frac{1}{\chi_{n}} \big(\tilde{\mathcal{E}}^{m}_{n,k}\big((t- n^{-1/2} \underline{d}_{m})^{+}\big) - \tilde{\mathcal{E}}^{m, \Delta}_{n,k}(t)\big) \asleq \mu (d_{m,k} - \underline{d}_{m}).
\]
By \eqref{eq:Mmnk-increments} and \eqref{eq:FSLLN},
\[
\limsup_{n \to \infty} \sup_{0 \leq t \leq T}  \frac{\big\vert\tilde{M}^{m}_{n,k}\big((t- n^{-1/2} \underline{d}_{m})^{+}\big) - \tilde{M}^{m, \Delta}_{n,k}(t)\big\vert}{n^{-1/4}\sqrt{\chi_{n}\log n \vee 1}} \asleq \sqrt{20 \mu (d_{m,k} - \underline{d}_{m})}.
\]
We complete the proof by taking $ C^{\dagger}_{k} > \sum_{m = 1}^{b} \mu (d_{m,k} - \underline{d}_{m}) $ and $ C^{\ddagger}_{k} > \sum_{m = 1}^{b} \sqrt{20 \mu (d_{m,k} - \underline{d}_{m})} $.

\end{document}